\documentclass{amsart}[12pt]
\usepackage{amssymb,amsmath,color}
\usepackage{color}
\usepackage{tikz}
\usepackage{comment}
\usetikzlibrary{positioning}
\usetikzlibrary{calc}
 \usepackage{array,multirow,graphicx}
\usepackage{pgf}
\usepackage{pgfplots}
\pgfplotsset{compat=1.10}
\usepgfplotslibrary{fillbetween}
\usepackage{bm}
\usetikzlibrary{shapes.geometric}
\usetikzlibrary{graphs,graphs.standard,quotes}
\usetikzlibrary{calc}
\usepackage{bm}
 \usepackage[left=1in,right=1in,top=0.9in,bottom=0.9in]{geometry}

\newtheorem{theorem}{Theorem}[section]
\newtheorem{conjecture}[theorem]{Conjecture}
\newtheorem{lemma}[theorem]{Lemma}
\newtheorem{corollary}[theorem]{Corollary}
\newtheorem{proposition}[theorem]{Proposition}
\theoremstyle{definition}
\newtheorem{definition}[theorem]{Definition}

\theoremstyle{remark}
\newtheorem{remark}[theorem]{Remark}

\newtheorem{exampleth}[theorem]{Example}
\newenvironment{example}{\begin{exampleth}}{\hfill $\diamond$\\ \end{exampleth}}

\def \trop {\operatorname{trop}}
\def \hom {\operatorname{hom}}
\def \Hom {\operatorname{Hom}}
\def \HDE {\operatorname{HDE}}
\newcommand{\GU}{{\mathcal{G}_\mathcal{U}}}
\newcommand{\NU}{{\mathcal{N}_\mathcal{U}}}
\newcommand{\tropNU}{\trop(\NU)}
\newcommand{\tropGU}{\trop(\GU)}
\newcommand{\tropn}{\trop(\mathcal{N}_{\mathcal{U}_n})}
\newcommand{\troptn}{\trop(\mathcal{N}_{\mathcal{U}_{2n+1}})}

\newcommand{\projC}{\textup{proj}_{2n+1}(C)}
\newcommand{\RR}{\mathbb{R}}
\newcommand{\ZZ}{\mathbb{Z}}
\newcommand \U {\mathcal{U}}

\DeclareMathOperator*{\argmax}{arg\,max}

\numberwithin{equation}{section} 

\title{A Path Forward: Tropicalization in Extremal Combinatorics}
\thanks{Grigoriy Blekherman was partially supported by NSF grant DMS-1901950. Annie Raymond was partially supported by NSF grant DMS-2054404.}

\author{Grigoriy Blekherman}
\address{School of Mathematics, Georgia Institute of Technology,
686 Cherry Street
Atlanta, GA 30332}\email{greg@math.gatech.edu}

\author{Annie Raymond}
\address{Department of Mathematics and Statistics,
Lederle Graduate Research Tower, 1623D,
University of Massachusetts Amherst
710 N. Pleasant Street
Amherst, MA 01003} \email{raymond@math.umass.edu}

\begin{document}

\begin{abstract}
Many important problems in extremal combinatorics can be be stated as proving a pure binomial inequality in graph homomorphism numbers, i.e., proving that $\hom (H_1,G)^{a_1}\cdots \hom (H_k,G)^{a_k}\geq \hom (H_{k+1},G)^{a_{k+1}}\cdots \hom (H_m,G)^{a_m}$ holds for some fixed graphs $H_1,\dots,H_m$ and all graphs $G$. One prominent example is Sidorenko's conjecture. For a  fixed collection of graphs $\mathcal{U}=\{H_1,\dots,H_m\}$, the exponent vectors of valid pure binomial inequalities in graphs of $\mathcal{U}$ form a convex cone. We compute this cone for several families of graphs including complete graphs, even cycles, stars and paths; the latter is the most interesting and intricate case that we compute. In all of these cases, we observe a tantalizing polyhedrality phenomenon: the cone of valid pure binomial inequalities is actually rational polyhedral, and therefore all valid pure binomial inequalities can be generated from the finite collection of exponent vectors of the extreme rays. Using the work of Kopparty and Rossman (\cite{koppartyrossman}), we show that the cone of valid inequalities is indeed rational polyhedral when all graphs $H_i$ are series-parallel and chordal, and we conjecture that polyhedrality holds for any finite collection $\mathcal{U}$. We demonstrate that the polyhedrality phenomenon also occurs in matroids and simplicial complexes. Our description of the inequalities for paths involves a generalization of the Erd\H{o}s-Simonovits conjecture recently proved in its original form in \cite{saglam} and a new family of inequalities not observed previously. We also solve an open problem of Kopparty and Rossman on the homomorphism domination exponent of paths. One of our main tools is tropicalization, a well-known technique in complex algebraic geometry, first applied in extremal combinatorics in \cite{BRST2}. We prove several results about tropicalizations which may be of independent interest.
\end{abstract}

\maketitle 


\section{Introduction}The \emph{number of homomorphisms} from a graph $H$ to a graph $G$, denoted by $\hom(H;G)$, is the number of maps from $V(H)$ to $V(G)$ that yield a graph homomorphism, i.e., that map every edge of $H$ to an edge of $G$. 
Many important problems and results in extremal graph theory can be framed as certifying the validity of polynomial inequalities in the number of graph homomorphisms which are valid on all graphs. We use $P_k$ to denote a path with $k$ edges and $K_m$ for the complete graph on $m$ vertices; note that $P_0=K_1$ and $P_1=K_2$. We will often use $H$ as a short-hand for $\hom(H;G)$ for the purposes of writing inequalities. By $H^k$, we denote both  $\hom(H;G)^k$ and $\hom(k \textup{ disjoint copies of } H; G)$ as they are equal. For example, the Goodman bound \cite{goodman} (which implies Mantel's theorem \cite{mantel}) states that $K_1 K_3 \geq 2K_2^2-K_2K_1^2$ and can be derived from $P_0P_2\geq P_1^2$,
and Sidorenko's conjecture \cite{Sid93} can be stated as
$P_0^{2|E(H)|-|V(H)|} \cdot H \geq  P_1^{|E(H)|},$
for any bipartite graph $H$.

Instead of homomorphism numbers, many papers consider \emph{homomorphism densities} where $t(H;G):=\frac{\hom(H;G)}{|V(G)|^{|V(H)|}}$ is the probability that a random map from $V(H)$ to $V(G)$ yields a graph homomorphism. Understanding all $s$-tuples of numbers that can occur as either homomorphism numbers or densities for a fixed collection $\U=\{U_1,\dots,U_s\}$ is an extremely complicated problem.  We will call the set of all $s$-tuples the \emph{number} (resp. \emph{density}) \emph{profile} of the collection $\U$. To the best of our knowledge, full descriptions of all $s$-tuples are only known for pairs of graphs, and even then in a very limited number of cases. For instance an important result of Razborov (\cite{RazTriangle}) completely describes the density profile of $\mathcal{U}=\{K_2,K_3\}$ (see picture on the left of Figure~\ref{fig:edge-triangle-profile1}). This was extended by Nikiforov to $\U=\{K_2,K_4\}$ in \cite{Nikiforov}, and generalized by Reiher to $\U=\{K_2,K_n\}$ \cite{MR3549620}. 
Understanding all $s$-tuples is essentially equivalent to understanding all polynomial inequalities in homomorphism densities or numbers which are valid on all graphs. It is known that the problem of checking whether a polynomial expression in either numbers or densities is nonnegative on all graphs is undecidable \cite{ioannidisramakrishnan, HN11}.

\begin{figure}[ht]
  \centering
\begin{tikzpicture}[scale=0.7]
\begin{axis}[xlabel=\large{Edge density $t(K_2;G)$}, ylabel=\large{Triangle density $t(K_3;G)$}, xmin=0, ymin=0, xmax=1, ymax=1, xtick={0,1}, ytick={0,1}, samples=50]
\addplot[fill=gray!50, draw=none, domain=0:1] {x^(3/2)} \closedcycle;
\addplot[thick, samples=50,domain=0:1] {x^(3/2)};
\addplot[fill=white, draw=none, domain=0:1] {-8*(x-2/3)^2+2/9} \closedcycle;
\addplot[thick, samples=30,domain=1/2:2/3] {-8*(x-2/3)^2+2/9};
\addplot[fill=white, draw=none, domain=0:1]    {-22*(x-3/4)^2 +0.375} \closedcycle;
\addplot[thick, samples=30,domain=2/3:3/4] {-22*(x-3/4)^2 +0.375};
\addplot[fill=white, draw=none, domain=0:1]   {-42*(x-4/5)^2+0.48} \closedcycle;
\addplot[thick, samples=30,domain=3/4:4/5] {-42*(x-4/5)^2+0.48};
\addplot[fill=white, draw=none, domain=0:1]  {-68*(x-5/6)^2+5/9} \closedcycle;
\addplot[thick, samples=30,domain=4/5:5/6] {-68*(x-5/6)^2+5/9};
\addplot[fill=white, draw=none, domain=0:1] {-100*(x-6/7)^2+0.6122}  \closedcycle;
\addplot[thick, samples=30,domain=5/6:6/7] {-100*(x-6/7)^2+0.6122};
\addplot[fill=white, draw=none, domain=0:1] {2.7146*x-1.7146} \closedcycle;
\addplot[thick, samples=30, domain=6/7:1] {2.7146*x-1.7146};
\addplot[thick, samples=30, domain=0:1/2] {0};
\end{axis}
\end{tikzpicture} \quad \quad \quad \quad \begin{tikzpicture}[scale=0.8]
\begin{axis}[
  axis x line=center,
  axis y line=center,
  xtick={0},
  ytick={0},
  xlabel={\footnotesize{$\log(t(K_2;G))$}},
  ylabel={\footnotesize{$\log(t(K_3;G))$}},
  xlabel style={right},
  ylabel style={above},
  xmin=-5.5,
  xmax=1.5,
  ymin=-5.5,
  ymax=1.5]
\addplot +[mark=none, thick, black] coordinates {(0,0) (0,-5.5)};
\addplot[name path=A, thick, black, domain=-11/3:0] {3*x/2};
\addplot[name path=B, thick, black, domain=-5.5/10000:0] {10000*x};
\addplot[gray!50] fill between[of=A and B];
\end{axis}
\end{tikzpicture}
\caption{\label{fig:edge-triangle-profile1} The density profile of an edge and triangle, and its tropicalization.}
\end{figure}
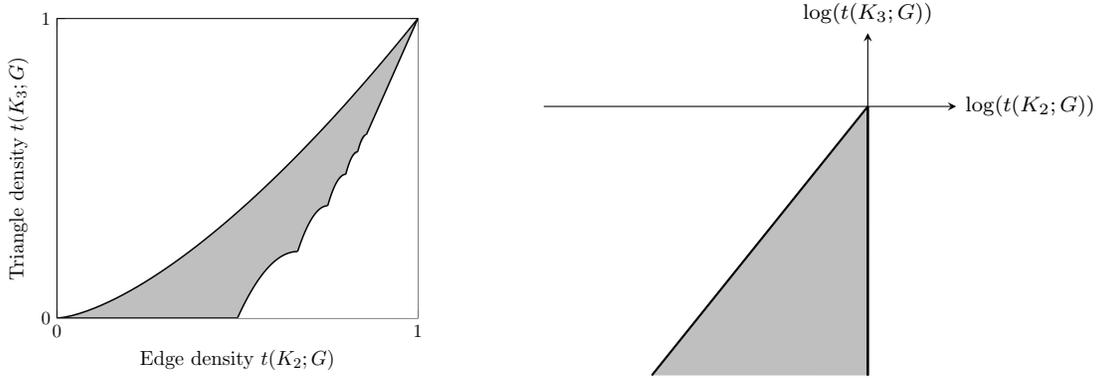

\emph{A pure binomial inequality} has the form $\mathbf{x}^{\bm{\alpha}} \geq \mathbf{x}^{\bm{\beta}}$ where $\mathbf{x}=(x_1,\dots,x_s) \in \mathbb{R}^s_{\geq 0}$ with $\bm{\alpha}, \bm{\beta} \in \mathbb{R}_{\geq 0}^s$, and it is equivalent to \emph{a linear inequality in logarithms:} $\langle \bm{\alpha}, \log \mathbf{x} \rangle \geq \langle \bm{\beta}, \log \mathbf{x} \rangle$. This inequality holds regardless of the base of the logarithm. \emph{Tropicalization} provides a base-independent way of studying the image of the logarithm map. It is a well-known technique in complex algebraic geometry and it was first applied to graph homomorphism problems in \cite{BRST2}. 
Instead of trying to understand the validity of a single pure binomial inequality, or a family of such inequalities, \emph{tropicalization allows us to analyze all valid pure binomial inequalities for a finite collection of graphs.}

While extremely complicated, both homomorphisms numbers and density $s$-tuples are closed under coordinatewise (Hadamard) multiplication. This implies that the logarithmic image is closed under addition, and tropicalization of any set closed under Hadamard multiplication is a closed, convex cone. Therefore analyzing binomial inequalities that are valid on an $s$-tuple of graphs appears to be a significantly simpler problem.  For instance, for any pair of graphs, the tropicalization of its number or density profile is a $2$-dimensional closed convex cone, so it is defined by two inequalities. (See the right side of Figure~\ref{fig:edge-triangle-profile1} for the tropicalization of the density profile of $\{K_2,K_3\}$.) 

In particular, there seems to be a hidden \emph{polyhedrality phenomenon}: tropicalizations that we are able to compute are always rational polyhedral cones. Dually, this allows us to generate all binomial inequalities for a fixed $\U$ from a finite collection of binomial inequalities. This polyhedrality phenomenon is not limited to graphs and counting homomorphisms, we also find it in \emph{simplicial complexes and matroids}. We stress that understanding binomial inequalities in graph homomorphism numbers is far from simple, for instance, despite some considerable amount of attention and progress, the Sidorenko conjecture remains open \cite{Sid93, li2011logarithmic, szegedy2014information, MR3456171,conlon2018sidorenko, conlonkimleelee}. Another interesting example of a binomial inequality is the Erd\H{o}s-Simonovits conjecture on paths which states that $P_0^{2v-2w}\cdot P_{2v+1}^{2w+1}  \geq  P_{2w+1}^{2v+1}$ for any $w<v$, and which was recently solved in \cite{saglam}.

We start with an illustrative example, which demonstrates the type of results that we can prove:

\begin{example}\label{ex:intro}
Consider the collection of even cycles $C_4,C_6,\dots,C_{2k}$. The following pure binomial inequalities hold for even cycles:
\begin{align*}
&\textrm{log-convexity:} \,\, C_{2m-2} C_{2m+2} \geq C_{2m}^2,  \,\,\, 3\leq m \leq k,\\ 
&\textrm{non-decreasing:} \,\, C_{2m-2} \leq C_{2m}, \,\,\, 3\leq m\leq k,\\
&\textrm{log-sublinear growth:} \,\, C_{2m-2}^m \geq C_{2m}^{m-1}, \,\,\, 3\leq m\leq k.
\end{align*}
Moreover, any pure binomial inequality in even cycles can be deduced in a finite way from the above inequalities, by exponentiating these inequalities and multiplying them together. For example, $C_4^3 C_{10}^2 \geq C_8^4$ is implied by the inequalities above, namely $C_6 C_{10} \geq C_8^2$ and $C_4^3 \geq C_6^2$, since they can be combined as $C_4^3 C_{10}^2 \geq C_6^2 C_{10}^2 \geq C_8^4.$
\end{example}
We compute all pure binomial inequalities in homomorphism numbers for several natural collections of graphs, including paths, complete graphs, star graphs, even cycles and (separately) odd cycles. In all these cases, we observe the same tantalizing phenomenon: similarly to Example \ref{ex:intro}, all pure binomial inequalities can be deduced from a \emph{finite collection of pure binomial inequalities}. This was observed in some of these families for density inequalities in \cite{BRST2}, but the statements for homomorphism numbers are strictly more general. 

Building on the work of Kopparty and Rossman \cite{koppartyrossman}, we show the following:

\begin{theorem}\label{thm:informalchordalseriesparallel}
Let $\mathcal{U}$ be a finite collection of chordal series-parallel graphs. Then there exists a finite collection of binomial inequalities, such that any pure binomial inequality in the graphs of $\U$ can be deduced in a finite way from this finite collection.
\end{theorem}

We conjecture that this polyhedrality phenomenon holds for any finite collection of graphs (see Conjecture \ref{conj:polyhedrality} for a precise formulation). 
The most complicated case we are able to completely understand is the full characterization of pure binomial inequalities in paths. Here, instead of indexing paths by their number of edges as before, it is more intuitive to index them by the number of vertices. We let $Q_u$ be a path with $u$ vertices. 

\begin{theorem}\label{thm:pathintro}
The following inequalities hold for homomorphism numbers of paths into any graph $G$ with no isolated vertices: 
\begin{align}
&\textrm{log-convexity between odd paths:} \,\, Q_a^{c-b}Q_{c}^{b-a} \geq Q_b^{c-a},\,\, 0\leq a\leq b \leq c, \,\, \textrm{and}\,\, a,c \,\, \textrm{odd},\\
\label{ineq:erds}&\textrm{log-convexity for odd and even paths, even middle:} \,\, Q_a^{c-b}Q_{c}^{b-a} \geq Q_b^{c-a},\,\, a\leq b \leq c, \,\, a\,\, \text{odd}, \,\, b,c \,\, \textrm{even},\\
\label{ineq:new} &\textrm{``weak convexity''  for odd and even path, odd middle:}\,\, Q_a^{\frac{c}{2}}Q_{c} \geq Q_b^{\frac{c}{2}},\,\, a\leq b \leq c, \,\, a,b\,\, \textrm{odd}, \,\, c \,\, \textrm{even},\\
&\textrm{non-decreasing:} \,\, Q_a \leq Q_b, \,\, a \leq b,\\
&\textrm{log-subadditivity:} \,\, Q_aQ_b \leq Q_{a+b}.
\end{align}
Moreover, any pure binomial inequality in paths can be deduced in a finite way from the above inequalities. In particular, for a binomial inequality where the largest path has $v$ vertices, only inequalities involving paths on at most $2v$ vertices need to be considered.
\end{theorem}

\begin{remark}
We prove a more general version of this theorem where we consider inequalities valid for any graph (instead of graphs with no isolated vertices), but restricting to the slightly simpler case of no isolated vertices avoids some slightly technical inequalities, which do not contribute additional insight into the big picture. Note that inequality \eqref{ineq:erds} is a generalization of the Erd\H{o}s-Simonovits inequality (which is restricted to the case when $a=1$), while inequality \eqref{ineq:new} is a new inequality, which we have not found in the literature. 
\end{remark}

Some of the proofs for the validity of inequalities in Theorem \ref{thm:pathintro} are based on the work in \cite{koppartyrossman}. Our global approach of analyzing all binomial inequalities between several paths allows us to solve the following open problem for the homomorphism domination exponent of any two paths that was posed in that paper:  find the largest $c_{m,n}$ so that $Q_m \geq Q_n^{c_{m,n}}$ is a valid inequality for a general $m$ and $n$.

\begin{theorem}\label{thm:iHDE}
The largest $c_{m,n}$ so that $Q_m \geq Q_n^{c_{m,n}}$ is a valid inequality is

$$c_{m,n}=\left\{\begin{array}{ll}
\frac{m}{n+1} & \textup{when }m \textup{ is even and }n \textup{ is odd and } m\leq n\\
\frac{km-(m-1)}{k(n-1)+2k-n} & \textup{when } m \textup{ and } n \textup{ are both even and } m\leq n\\
\frac{m}{n} & \textup{when } m \textup{ is odd and } m\leq n\\
1 & \textup{when } m\geq n
\end{array}\right.$$
where $k$ is the smallest integer such that $k\cdot m\geq n$. Note that the last two lines were already proven in \cite{koppartyrossman}. Our contribution is the first two lines.
\end{theorem}

Pure binomial inequalities in path homomorphism numbers have a long history and often go under the name of ``walks in graphs''. 
Besides the  Erd\H{o}s-Simonovits conjecture and its solution and the results of Kopparty and Rossman, Lagarias, Mazo, Shepp and McKay showed in \cite{Lagarias} that $P_0P_{2a+2b} \geq P_{2a+b}P_b.$ In \cite{DressGutman}, Dress and Gutman showed that $P_{2a}P_{2b} \geq P_{a+b}^2.$ In \cite{Hemmecke}, Hemmecke, Kosub, Mayr, T\"{a}ubig and Weihmann generalized all previous listed inequalities by certifying the following two types of inequalities: $P_{2a}P_{2(a+b+c)} \geq P_{2a+c}P_{2(a+b)+c}$ and $P_{2l+pk}P_{2l}^{k-1}\geq P_{2l+p}^k.$ 
Note that pure binomial inequalities in path numbers are also related to spectral properties of the graph, such as the spectral radius \cite{MR2257594}.  A related interesting question of what can be said about graphs that have the same path homomorphism numbers for all paths was addressed in \cite{MR3829971}.

We develop new techniques for analyzing tropicalizations. Some of our results mirror results from ``tropical convexity", but as we will explain, there is a slight difference in setup, which forces us to provide proofs. We believe that some of our tropicalization results are of independent interest.

We now discuss some definitions to make certain concepts more precise and state some important results more formally. We also discuss the proof strategies that we use.

\subsection{Definitions, Results and Proof Strategies:}

The {\em number graph profile} of a collection of connected graphs $\mathcal{U} = \{H_1, \ldots, H_s \}$, denoted as $\mathcal{N}_\mathcal{U}$,
 is the set of all vectors $(\hom(H_1;G), \hom(H_2;G), \ldots, \hom(H_s;G))$ as $G$ varies over all graphs. The \emph{density graph profile} of $\mathcal{U}$, denoted as $\mathcal{G}_\mathcal{U}$, is the closure of all vectors $(t(H_1;G), t(H_2; G), \ldots, t(H_s;G))$ as $G$ varies over all graphs. For any $\mathcal{U}$, the density graph profile $\GU$ is contained in $[0,1]^s$ and the number graph profile $\NU$ is contained in $\mathbb{N}^s$. We say that a subset $\mathcal{S}$ of $\RR^s$ has \emph{the Hadamard property} if $\mathcal{S}$ is closed under coordinatewise (Hadamard) multiplication. By considering tensor product of graphs, it is easy to see that both the numbers and density profiles have the Hadamard property. 
 
Pure binomial inequalities in homomorphism numbers are more general than pure binomial inequalities in densities. Any valid inequality for the density graph profile $\mathcal{G}_\mathcal{U}$ corresponds to a valid inequality for the number graph profile $\mathcal{N}_{\mathcal{U}'}$ where $\mathcal{U}'$ is the union of $\mathcal{U}$ and the one vertex graph $P_0$. Indeed, one can multiply the densities by $|V(G)|$ to a high enough power to cancel out all denominators, and thus obtain an inequality involving homomorphism numbers where different terms are multiplied by $|V(G)|$ to different powers. Since $\hom(P_0;G)=|V(G)|$ this is an equivalent inequality in homomorphism numbers. Note that such an inequality must have the same number of vertices in each term, which doesn't need to be true for inequalities for the number graph profile. In this paper, we will only consider number graph profiles, which we hereafter sometimes refer to simply as \emph{profiles} or \emph{graph profiles}. 
 
 Let $\log \,:\, \RR^s_{>0} \rightarrow \RR^s$ be defined as $\log({\bf v}) := (\log_e(v_1), \ldots, \log_e(v_s))$. If we need to change the
base of the log from $e$ to $\alpha$, then we will explicitly write $\log_\alpha$.
For a set $\mathcal{S} \subseteq \RR^s_{\geq 0}$, we define $\log(\mathcal{S}) := \log (\mathcal{S} \cap \RR^s_{>0})$.
The tropicalization of $\mathcal{S}$, which is also called the {\em logarithmic limit set} of $\mathcal{S}$, is defined to be:
$$\trop(\mathcal{S}) := \lim_{\tau \rightarrow \infty} \log_{\tau}(\mathcal{S}).$$ In \cite[Lemma 2.2]{BRST2}, it was shown that if $\mathcal{S}\subseteq \mathbb{R}^s_{\geq 0}$ has the Hadamard property, then $\trop(\mathcal{S})$ is a closed convex cone, and moreover the tropicalization of $\mathcal{S}$ is equal to the closure of the conical hull of $\log(\mathcal{S})$, and so $\trop(\NU) = \textup{cl}(\textup{cone}(\log(\NU)))$. The extreme rays of the dual cone $\tropNU^*$ generate all of the pure binomial inequalities valid on $\NU$. 

Note that $\trop(\NU) \subseteq \mathbb{R}_{\geq 0}^s$ since $\NU \cap \RR^s_{>0}$ contains only points where every coordinate is at least one. We show in Proposition \ref{prop:dontlose} that we do not add any spurious binomial inequalities by removing points with zero coordinates from $\NU$, except for adding the pure binomial inequalities that every coordinate is at least one.

Another nice property of $\tropNU$ is that it is \emph{max-closed}: if $(x_1, \ldots, x_s), (y_1, \ldots, y_s)\in \tropNU$, then $(\max\{x_1,y_1\}, \ldots, \max\{x_s,y_s\})\in \tropNU$. In Theorem \ref{cor:onenegative}, we use this property to show that the only pure binomial inequalities needed to fully describe $\tropNU$ have the form $H_1^{\alpha_1} \cdots H_s^{\alpha_s} \geq H_j^{\beta}$ for some $\alpha_1, \ldots, \alpha_s, \beta \geq 0$ and $j\in [s]$.

Given a subset $\mathcal{S}\in \mathbb{R}_{\geq 0}^s$, \emph{the double hull} of $\mathcal{S}$ is the smallest closed and max-closed convex cone containing $\mathcal{S}$.
It follows from Corollary \ref{cor:dualc} that the double hull of a polyhedral cone is polyhedral, and we show that a closed and max-closed convex cone in $\mathbb{R}^s_{\geq 0}$ is the double hull of its doubly extreme rays in Theorem \ref{thm:dext}.

In Section \ref{sec:examplesoftropicalizations}, we compute $\tropNU$ for different classes of graphs and also for simplicial complexes and matroids.
Our strategy is to first find a proposed $H$-description of the tropicalization. 
We certify that all of the inequalities of the $H$-description come from valid pure binomial inequalities on the profile $\NU$. Certification can be done via the sums of squares method, using AM-GM or H\"older's inequality, or using the tools presented in \cite{koppartyrossman}. This shows that the cone defined by the $H$-description contains the tropicalization $\tropNU$.

To show that the cone defined by the $H$-description is contained in the tropicalization, we use different techniques. In the simplest cases, e.g., for the density cases done in \cite{BRST2} and some cases in this paper, we find all extreme rays of the cone defined by the $H$-description and show that all these rays are \emph{realizable}. An extreme ray $\mathbf{r}$ is realizable if there exists a graph $G$ or a sequence of graphs $G_n$ on $n$ vertices such that  $\alpha(\log \hom(H_1;G), \ldots, \log \hom(H_s;G))=\mathbf{r}$ or $\alpha(\log \hom(H_1;G_n), \ldots, \log \hom(H_s;G_n))\rightarrow \mathbf{r}$ as $n\rightarrow \infty$ respectively for some constant $\alpha\in \mathbb{R}_{\geq 0}$. These sequences of graphs often arise by constructing blow-up graphs defined in Section \ref{subsec:blowup}, and taking tensor powers and disjoint unions of the blow-up graphs. 

In some cases, the set of doubly extreme rays of the cone defined by the $H$-description is significantly simpler that the set of all extreme rays. After finding the doubly extreme rays, we verify that they are realizable and that their double hull contains the cone defined by the $H$-description. An example of this strategy is the case of even cycles in Theorem \ref{thm:ecycles}.

The most challenging case is when we consider the path profile: $\mathcal{U}=\{P_0, P_1, \ldots, P_{2n+1}\}$. We do not give an explicit $H$-description of the tropicalization of the path profile. Instead we construct a \emph{lifted representation} by using inequalities involving larger paths.

\begin{theorem}
Let $\mathcal{U}:=\{P_0, P_1, \ldots, P_{2n+1}\}$, $y_i:=\log(\hom(P_i;G))$, and 
\begin{align*}
C:=\big\{\mathbf{y}=&(y_0, y_1, y_2, \ldots, y_{4n+3}) \in \mathbb{R}^{4n+4} | &&\\
 &y_{2u}-2y_{2u+1}+y_{2u+2} \geq 0 && \forall 0 \leq u \leq 2n \\
& y_{2u} - 2y_{2u+2} + y_{2u+4} \geq 0 && \forall 0 \leq u \leq 2n-1 \\
& -y_{2u} + y_{2u+1} \geq 0 && \forall 1\leq u \leq 2n+1  \\
& y_{2u+1}+y_{2v+1}-y_{2u+2v+3} \geq 0 && \forall 0 \leq u \leq v \textup{ such that } u+v \leq 2n\\
& 2y_{2u}-(2v+1-2u)y_{2v-1}+(2v-1-2u)y_{2v+1} \geq 0 && \forall 0 \leq u<v-1, v\leq 2n+1   \\
& (u+1)y_{2u-2}-(u+1)y_{2u}+y_{2u+1} \geq 0  && \forall 1 \leq u\leq 2n+1 \\
& y_1-2y_2+y_4\geq 0\\
& y_1-(2v+1)y_{2v-1}+(2v-1)y_{2v+1} \geq 0 && \forall 2 \leq v\leq 2n+1\\
& -y_1+y_2 \geq 0 \big\}.
\end{align*}

Further, let $\projC=\{(y_0, y_1, \ldots, y_{2n+1})|(y_0, y_1, \ldots, y_{4n+3})\in C\}$. Then $\projC=\tropNU$.
\end{theorem}
 
\begin{remark}
Note that the defining inequalities of the the lift cone $C$ are of a special simple form. They involve at most three variables and exactly one variable appears with a negative sign. The fact that there is at most one variable occurring with a negative sign is a consequence of max-closedness of tropicalizations of number profiles (see Corollary \ref{cor:onenegative}). We do not have an explanation for why there are at most two variables occurring with a positive sign. In fact, if we consider the $H$-description of the actual tropicalization, instead of the lifted representation, then we find inequalities involving more than three variables.
\end{remark}

We certify that all of the inequalities involved in defining $C$ above come from valid pure binomial inequalities in Section \ref{sec:truebinomialinequalitieswithpaths}. In the process we generalize the Erd\H{o}s-Simonovits conjecture and find a previously unknown family of inequalities.

The most technical part of the paper is showing that the cone $\projC$ is contained in the tropilicalization $\troptn$. We were not able to find a simple description of the extreme rays or doubly extreme rays of $\projC$. Instead in Section \ref{sec:troppaths} we build a larger collection of rays whose double hull contains $\projC$ and show that all of these are realizable.

\subsection{Organization of the paper}

In Section \ref{sec:propertiestropicalizations}, we discuss properties of tropicalizations and the double hull. We compute $\tropNU$ when $\mathcal{U}$ is a collection of even cycles, odd cycles, stars, complete graphs, simplicial complexes and matroids in Section \ref{sec:examplesoftropicalizations}. Section \ref{sec:kr} describes ideas and results from \cite{koppartyrossman} that we need in the sequel. In Section \ref{sec:truebinomialinequalitieswithpaths}, we prove pure binomial inequalities involving paths, and give among other things a proof of the Erd\H{o}s-Simonovits conjecture that was first proved by Sa\u{g}lam, as well as a generalization of the inequalities involved in that conjecture. In Section \ref{sec:troppaths}, we compute $\tropNU$ when $\mathcal{U}=\{P_0, P_1, \ldots, P_{2n+1}\}$. We end the paper with applications of our complete description of $\tropNU$ in Section \ref{sec:applications}. We compute the homomorphism domination exponent for any two paths in Section \ref{subsec:HDEPmPn}, a question that was posed in \cite{koppartyrossman}.

\section{Properties of tropicalizations}\label{sec:propertiestropicalizations}

Some theorems in this section have direct counterparts in \emph{tropical convexity}. Tropical convexity studies max-closed sets in $\RR^s$ under the additional assumption that a set $\mathcal{S}$ has the vector $\vec{1}=(1,1,\dots,1)$ in the lineality space: $\vec{1}\in \mathcal{S}$ and if $\mathbf{x}\in \mathcal{S}$, then $\mathbf{x}+\lambda \vec{1}\in \mathcal{S}$ for all $\lambda \in \RR$. Tropically convex sets, i.e., max-closed sets with $\vec{1}$ in the lineality space, are usually studied in the tropical projective space, which is $\mathbb{R}^s$ modulo the span of $\vec{1}$. We show that many of the theorems from tropical convexity also hold for cones contained in the nonnegative orthant $\RR^s_{\geq 0}$. The results about double hulls and doubly extreme rays are new, and we think they are of independent interest. All results in this section also have equivalent formulations for cones contained in $\RR^s_{\leq 0}$, which is useful for tropicalizations of density profiles; it is important that all coordinates have the same sign.

\begin{definition}
For vectors $\mathbf{x}, \mathbf{y} \in \mathbb{R}^s$, we let $\mathbf{x}\oplus \mathbf{y}$ denote their \textit{tropical sum}: $$\mathbf{x}\oplus \mathbf{y}=(\max\{x_1, y_1\}, \ldots, \max\{x_s, y_s\}).$$
We call a set $\mathcal{S}\subseteq \mathbb{R}^s$ \textit{max-closed} if for any $\mathbf{x}, \mathbf{y} \in \mathcal{S}$ we have $\mathbf{x}\oplus \mathbf{y}\in \mathcal{S}$.
\end{definition}

It turns out that tropicalizations of graph profiles have this property.

\begin{lemma}\label{tropismaxclosed1}
Let $\mathcal{S}\subset \ZZ^s_{\geq 0}$ be a semiring under coordinatewise addition and coordinatewise (Hadamard) multiplication. Then $\trop(\mathcal{S})$ is a max-closed convex cone.
\end{lemma}

\begin{proof}
The fact that $\trop(\mathcal{S})$ is a convex cone equal to $\textup{cl}(\textup{cone}(\log(\mathcal{S})))$ follows from Lemma 2.2 in \cite{BRST2} since $\mathcal{S}$ is closed under Hadamard multiplication. We now show that $\trop(\mathcal{S})$ is max-closed, that is, for any $\mathbf{x},\mathbf{y}\in \trop(\mathcal{S})$, we want to show that $\mathbf{x}\oplus \mathbf{y} \in \trop(\mathcal{S})$. Since $\trop(\mathcal{S})= \textup{cl}(\textup{cone}(\log(\mathcal{S})))$,  it suffices to consider $\mathbf{x}= \alpha \mathbf{x}'$ and $\mathbf{y}=\beta \mathbf{y}'$ for some $\mathbf{x}', \mathbf{y}' \in \log(\mathcal{S})$ and $\alpha, \beta\in \mathbb{R}_{\geq 0}$. We can assume that $\alpha, \beta \in \mathbb{Q}_{\geq 0}$ since $\trop(\mathcal{S})$ is closed, and so we can further assume that $\alpha, \beta$ have the same denominator, say $\alpha = \frac{a_1}{b}$ and $\beta=\frac{a_2}{b}$ where $a_1, a_2, b \in \mathbb{N}$. Since $\trop(\mathcal{S})$ is a cone, it suffices to show that $a_1 \mathbf{x}'\oplus a_2\mathbf{y}' \in \trop(\mathcal{S})$. Since $\mathbf{x}', \mathbf{y}' \in \log(\mathcal{S})$, we have $\mathbf{x}'=\log(\mathbf{v})$ and $\mathbf{y}'=\log(\mathbf{w})$ for some $\mathbf{v}, \mathbf{w}\in \mathcal{S}$. Observe that $\mathbf{v}^{a_1l}+\mathbf{w}^{a_2l}\in \mathcal{S}$ for any $l\in \mathbb{N}$ since $\mathcal{S}$ is closed under coordinatewise addition and Hadamard multiplication. Further, note that $$\frac{\log(\mathbf{v}^{a_1l}+\mathbf{w}^{a_2l})}{l} \rightarrow \log(\mathbf{v}^{a_1})\oplus \log(\mathbf{w}^{a_2}) = a_1 \mathbf{x}' \oplus a_2 \mathbf{y}'$$ as $l\rightarrow \infty$. Therefore, $a_1 \mathbf{x}' \oplus a_2 \mathbf{y}' \in \trop(\mathcal{S})$ since $\frac{\log(\mathbf{v}^{a_1l}+\mathbf{w}^{a_2l})}{l}\in \trop(\mathcal{S})$ for every $l$ and $\trop(\mathcal{S})$ is closed. Hence, $\trop(\mathcal{S})$ is max-closed.
\end{proof}

\begin{corollary}\label{tropismaxclosed2}
Let $\mathcal{U}=\{H_1, \ldots, H_s\}$ be a finite collection of connected graphs. Then $\tropNU$ is a max-closed convex cone.
\end{corollary}
\begin{proof}
It is well known that $\hom(H_i;G_1)\cdot \hom(H_i;G_2) = \hom(H_i; G_1 \times G_2)$ and $\hom(H_i;  G_1) + \hom(H_i; G_2) = \hom(H_i; G_1G_2)$ where $G_1\times G_2$ is the categorical product of $G_1$ and $G_2$, and $G_1G_2$ is the disjoint union of $G_1$ and $G_2$. Thus $\NU$ is closed under coordinatewise addition and Hadamard multiplication, and so the previous Lemma applies.
\end{proof}

As we mentioned in the introduction, the dual cone to $\trop(\mathcal{S})$ essentially encodes the valid pure binomial inequalities on $\mathcal{S}$. We now make this precise.

\begin{proposition}\label{prop:dontlose}
Let $\mathcal{S}\subset \ZZ^s_{\geq 0}$ be a semiring under coordinatewise addition and coordinatewise multiplication such that $\mathcal{S}$ is not contained in a coordinate hyperplane. Let $\bm\alpha \in \trop(\mathcal{S})^*$. Write $\bm\alpha=\bm\alpha_+-\bm\alpha_-$ where $\bm\alpha_+,\bm\alpha_- \in \RR^s_{\geq 0}$ are the positive and negative parts of $\bm\alpha$. If $\bm\alpha_- \neq \vec{0}$, then $\mathbf{x}^{\bm\alpha_+} \geq \mathbf{x}^{\bm\alpha_-}$ is a valid binomial inequality on $\mathcal{S}$.
\end{proposition}
\begin{proof}
Since $\bm\alpha \in \trop(\mathcal{S})^*$, we have $\langle \bm\alpha_+, \mathbf{y} \rangle \geq \langle \bm\alpha_-, \mathbf{y}\rangle$ for all $\mathbf{y} \in \log(\mathcal{S})$. Therefore, $\mathbf{x}^{\bm\alpha_+}\geq \mathbf{x}^{\bm\alpha_-}$ is a valid binomial inequality for all positive points in $\mathcal{S}$. Since $\mathcal{S}$ is not contained in a coordinate hyperplane and $\mathcal{S}$ is closed under addition, we see that $\mathcal{S}$ contains a strictly positive point $\mathbf{a}$. Suppose that $\mathbf{b}^{\bm\alpha_+} < \mathbf{b}^{\bm\alpha_-}$ for some $\mathbf{b} \in \mathcal{S}$. Consider $\mathbf{b}^k+\mathbf{a} \in \mathcal{S}$ which is a strictly positive point. For a sufficiently large $k$, we see that $(\mathbf{b}^k+\mathbf{a})^{\bm\alpha_+} < (\mathbf{b}^k+\mathbf{a})^{\bm\alpha_-}$, which is a contradiction.

\end{proof}
\begin{remark}
Observe that if ${\bm\alpha_-}=0$, then this corresponds to a pure binomial inequality $x_1^{\alpha_1}\dots x_s^{\alpha_s} \geq 1$ with 
$\alpha_i \geq 0$. This inequality is valid on all points of $\mathcal{S}$ with non-zero coordinates, but may fail if $\mathcal{S}$ has a point with a zero coordinate. Note that all such inequalities are generated from inequalities $x_i \geq 1$ for $1\leq i \leq s$. We remark that even though $x_i\geq 1$ may fail to hold on $\mathcal{S}$, a slight perturbation $x_i^{1+\varepsilon} \geq x_i^{\varepsilon}$ is valid on $\mathcal{S}$. 
\end{remark}

We call a set $\mathcal{S}\subseteq \mathbb{R}^s$ a \textit{cone} if $\mathcal{S}$ is closed under multiplication by nonnegative scalars. Note that a cone does not have to be convex.

For $i\in\{1,\dots,s\}$, let $\mathcal{Q}_i$ be the orthant of $\mathbb{R}^s$ consisting of all points where the $i$-th coordinate is nonpositive and the rest are nonnegative. The corresponding result in tropical convexity is known as the tropical Farkas Lemma \cite{MR2117413,hill2019tropical}:

\begin{lemma}\label{lem:mhull}
Let $\mathcal{S}\subset \mathbb{R}^s_{\geq 0}$ be a cone and let $$\mathcal{M}=\bigcap_{i=1}^s (\mathcal{S}+\mathcal{Q}_i),$$ where $+$ denotes Minkowski addition. Then $\mathcal{M}$ is a max-closed cone contained in $\mathbb{R}^s_{\geq 0}$, and it is the smallest max-closed cone containing $\mathcal{S}$.

\end{lemma}
\begin{proof}
It is clear that $\mathcal{M}$ is a cone contained in the nonnegative orthant. We claim that $\mathbf{x}\in \mathcal{M}$ if and only if there exist $\mathbf{y}_1,\ldots, \mathbf{y}_s \in \mathcal{S}$ such that $\mathbf{x}=\oplus_{i=1}^s \mathbf{y}_i$. We first show that if $\mathbf{x}\in \mathcal{M}$, then there exist $\mathbf{y}_1,\ldots, \mathbf{y}_s \in \mathcal{S}$ such that $\mathbf{x}=\oplus_{i=1}^s \mathbf{y}_i$. Observe that for all $i \in [s]$, we have $\mathbf{x}=\mathbf{u}_i+\mathbf{p}_i$ where $\mathbf{u}_i\in \mathcal{S}$ and $\mathbf{p}_i\in \mathcal{Q}_i$. Let $x_j$, $u_{ij}$ and $p_{ij}$ be the $j$th coordinate of $\mathbf{x}$, $\mathbf{u}_i$ and $\mathbf{p}_i$ respectively. Thus, we have that $x_j=u_{ij}+p_{ij}$ for all $i, j\in [s]$. Note that $p_{ii}\leq 0$ and $p_{ij} \geq 0$ for all $j\neq i$. Thus, $x_{i}=u_{ii}+p_{ii}$ implies that $x_i \leq u_{ii}$ for all $i\in [s]$. Similarly, $x_j=u_{ij}+p_{ij}$ implies that $x_j \geq u_{ij}$ for all $i,j\in [s]$ such that $i\neq j$. This shows that $\mathbf{x}$ lies inside the box defined by the following inequalities: $\max\{ u_{ij} | j\neq i\} \leq x_i \leq u_{ii}$ for all $i\in [s]$. 
 
We would like $\mathbf{x}$ to be the \emph{maximum corner} of that box, i.e., we would like that $x_i=u_{ii}$ for all $i$. We claim that if this is not the case, we can build a new box that still contains $\mathbf{x}$ and for which $\mathbf{x}$ is in the maximum corner. Indeed, suppose that $x_i < u_{ii}$ for some particular $i$. Then replace $\mathbf{u}_i$ with $\mathbf{u}'_i:= \lambda \mathbf{u}_i$ so that $x_i=\lambda u_{ii}=: u'_{ii}$. 
Certainly $\mathbf{u}'_i \in \mathcal{S}$ since $\mathcal{S}$ is a cone.
 Furthermore, there exists a point $\mathbf{p}'_i\in \mathcal{Q}_i$ such that $\mathbf{x}=\mathbf{u}'_i+\mathbf{p}'_i$, namely $p'_{ii}:=0$ and $p'_{ij}:= p_{ij}+(1-\lambda)u_{ij}\geq p_{ij} \geq 0$ since $\lambda <1$. Let $I\subset [s]$ be the set of indices for which $x_i=u_{ii}$. By replacing $\mathbf{u}_i$ by $\mathbf{u}'_i$ for all $i \in [s] \setminus I$ we obtain a new box where $\mathbf{x}$ is in the maximum corner. Thus, $\mathbf{x}$ can be seen as the coordinatewise maximum of $\mathbf{u}_i$ for all $i\in I$ and $\mathbf{u}'_i$ for all $i\in [s]\backslash I$.  
 
We now show the converse. Let $\mathbf{y}_1,\dots, \mathbf{y}_s \in \mathcal{S}$ and let $\mathbf{x}=\oplus_{i=1}^s \mathbf{y}_i$. Then for all $j \in [s]$, the $j$th coordinate of $\mathbf{x}$ is equal to the $j$th coordinate of some $\mathbf{y}_i$. Therefore, $\mathbf{x} \in \mathcal{S}+\mathcal{Q}_i$ for all $i \in [s]$, and thus $\mathbf{x}\in \mathcal{M}$.

To show that $\mathcal{M}$ is max-closed, let $\mathbf{a},\mathbf{b} \in \mathcal{M}$. Then $\mathbf{a}=\oplus_{i=1}^s \mathbf{y}_i$ and $\mathbf{b}=\oplus_{i=s+1}^{2s} \mathbf{y}_i$ with $\mathbf{y}_i \in \mathcal{S}$. Therefore, $\mathbf{a}\oplus \mathbf{b}=\oplus_{i=1}^{2s} \mathbf{y}_i$ and $\mathbf{a}\oplus \mathbf{b} \in \mathcal{M}$. Furthermore, we showed that any point of $\mathcal{M}$ is a tropical sum of points from $\mathcal{S}$, and therefore $\mathcal{M}$ is the smallest max-closed set containing $\mathcal{S}$.
 \end{proof}

We immediately obtain the following corollary.

\begin{corollary}\label{cor:mcconvex}
Let $\mathcal{S}\subset \mathbb{R}^s_{\geq 0}$ be a convex cone. The max-closure of $\mathcal{S}$ is a convex cone.
\end{corollary}

For a convex cone $\mathcal{S}\subset \mathbb{R}^s$, let $\mathcal{S}^*$ denote its dual cone with respect to the standard inner product.

\begin{corollary}\label{cor:dualc}
Let $\mathcal{S}\subset \mathbb{R}^s_{\geq 0}$ be a max-closed convex cone. Then 
$$\mathcal{S}^*=\sum_{i=1}^s (\mathcal{S}^*\cap \mathcal{Q}_i),$$
where $\sum$ denotes Minkowski addition. 
\end{corollary}

\begin{proof}
We know from Lemma \ref{lem:mhull} that $\mathcal{S}=\bigcap_{i=1}^s (\mathcal{S}+\mathcal{Q}_i)$. The rest follows from standard results in convexity. The dual cone of an intersection of cones is equal to the Minkowski sum of the dual cones, so we have: $\mathcal{S}^*=\oplus_{i=1}^s (\mathcal{S}+\mathcal{Q}_i)^*$. The dual cone of Minkowski sum of cones is the intersection of dual cones, so we have $\mathcal{S}^*=\oplus_{i=1}^s (\mathcal{S}^*\cap \mathcal{Q}_i^*)$. Finally, the cone $\mathcal{Q}_i$ is self-dual: $\mathcal{Q}_i^*=\mathcal{Q}_i$ for all $i$, and the desired conclusion follows.
\end{proof}

We now prove an important structural result on extreme rays of dual cones to tropicalization of graph profiles.

\begin{corollary}\label{cor:onenegative}
Any extreme ray of the dual cone $\tropNU^*$ is spanned by a vector with at most one negative coordinate.
\end{corollary}
\begin{proof}
Since $\tropNU$ is a max-closed convex cone contained in $\mathbb{R}^s_{\geq 0}$ we can apply Corollary~\ref {cor:dualc}. It follows that an extreme ray of $\tropNU^*$ is also an extreme ray of one of the intersections $\tropNU^*\cap Q_i$ and is therefore spanned by a vector with at most one negative coordinate.
\end{proof}

\begin{definition}
Let $\mathcal{S}\subset \mathbb{R}^s_{\geq 0}$ be a cone. We define the \textit{double hull} of $\mathcal{S}$ to be the smallest max-closed convex cone containing $\mathcal{S}$ and denote it by $\operatorname{dh}(\mathcal{S})$.
\end{definition}
We see from Corollary \ref{cor:mcconvex} that the double hull of $\mathcal{S}$ is the max-closure of the convex hull of $\mathcal{S}$.

\begin{definition}
We call a point $\mathbf{p}\in \mathcal{S}$ \textit{max-extreme} if whenever $\mathbf{p}=\mathbf{x}\oplus \mathbf{y}$ with $\mathbf{x},\mathbf{y} \in \mathcal{S}$, then $\mathbf{p}=\mathbf{x}$ or $\mathbf{p}=\mathbf{y}$. It is easy to see that $\mathbf{p}$ is max-extreme if and only if any point on the ray spanned by $\mathbf{p}$ is also max-extreme in $\mathcal{S}$.

Let $\mathcal{S}\subset \mathbb{R}^s_{\geq 0}$ be a closed convex cone. We say that $\mathbf{p}\in \mathcal{S}$ spans a \textit{doubly extreme ray} of $\mathcal{S}$ if $\mathbf{p}$ spans an extreme ray of $\mathcal{S}$ and $\mathbf{p}$ is max-extreme.
\end{definition}

\begin{theorem}\label{thm:tmink}
Let $\mathcal{S}\subseteq \mathbb{R}^s_{\geq 0}$ be a closed and max-closed cone. Then $\mathcal{S}$ is equal to the max-closure of its max-extreme rays.
\end{theorem}
\begin{proof}
Let $\mathbf{p}\in \mathcal{S}$. Observe that for $i\in\{1,\dots, s\}$, the subset $\mathcal{S}_i$ of $\mathcal{S}$ consisting of points 
that are coordinatewise less that or equal to $\mathbf{p}$ and with $i$-th coordinate equal to $p_i$ is compact. For each $i\in\{1, \dots, s\}$, construct points $\mathbf{y}_i$ as follows: $y_{ii}=p_i$ and $\mathbf{y}_i$ is a lexicographically minimal point of $\mathcal{S}_i$, which means that $y_{i1}\leq x_1$ for all $\mathbf{x}\in \mathcal{S}_i$ and if $y_{i1}=x_1$ then $y_{i2}\leq x_2$, and so on. We claim that $\mathbf{y}_i$ are max-extreme points of $\mathcal{S}$. Suppose that $\mathbf{y}_i=\mathbf{v}\oplus \mathbf{w}$. Then $y_{ii}=v_i$ or $y_{ii}=w_i$, and then we must have $\mathbf{y}_i=\mathbf{v}$ or $\mathbf{y}_i=\mathbf{w}$. Thus $\mathbf{y}_i$ is max-extreme and we have $\mathbf{p}=\oplus_{i=1}^s \mathbf{y}_i$.

\end{proof}

\begin{theorem}\label{thm:dext}
Let $\mathcal{S}\subseteq \mathbb{R}^s_{\geq 0}$ be a closed convex cone that is also max-closed. Then $\mathcal{S}$ is equal to the double hull of its doubly extreme rays.
\end{theorem}
\begin{proof}
It is clear that $\mathcal{S}$ contains the double hull of its doubly extreme rays. To show the other inclusion, by Theorem \ref{thm:tmink}, it suffices to show that any max-extreme point of $\mathcal{S}$ is contained in the conical hull of double rays of $\mathcal{S}$.

Let $\mathbf{p}\neq 0$ be a max-extreme point of $\mathcal{S}$. Then $\mathbf{p}$ lies on the boundary of $\mathcal{S}$, and therefore $\mathbf{p}$ lies in the relative interior of a proper face $F$ of $\mathcal{S}$. We claim that any point $\mathbf{v}\in F$ is max-extreme. Suppose not,  then $\mathbf{v}=\mathbf{x}\oplus \mathbf{y}$, with $\mathbf{x},\mathbf{y}\neq \mathbf{v}$, and we can write $\mathbf{p}=\lambda \mathbf{v} +(1-\lambda) \mathbf{w}$ for some $\mathbf{w}\in F$ and $0<\lambda<1$. Therefore $\mathbf{p}=(\lambda \mathbf{x}+(1-\lambda) \mathbf{w})\oplus (\lambda \mathbf{y}+(1-\lambda) \mathbf{w})$, and since $\mathbf{x},\mathbf{y}\neq \mathbf{v}$ we see that $\mathbf{p}$ is not max-extreme, which is a contradiction.

Therefore, we see that extreme rays of $F$ are doubly extreme rays of $\mathcal{S}$ since they are max-extreme and since they are also extreme rays of $\mathcal{S}$. Thus, any max-extreme point is in the convex hull of doubly extreme rays of $\mathcal{S}$, and we may conclude by Theorem \ref{thm:tmink}.
\end{proof}

\subsection{Some examples of tropicalizations}\label{sec:examplesoftropicalizations}

We illustrate the previous properties with some relatively simple examples. In Section \ref{sec:troppaths}, we find the tropicalization of a more complicated example, namely when $\mathcal{U}$ contains all paths up to a certain length. We note that the tropicalizations in all our examples are rational polyhedral cones, including the general case of chordal series-parallel graphs of Theorem~\ref{thm:chordalseriesparallel}, which leads us to make the following conjecture. Moreover, we compute the tropicalization for an example (matroids) without the Hadamard property where we also observe polyhedrality.

\begin{conjecture}\label{conj:polyhedrality}
Let $\mathcal{U}$ be a finite collection of finite connected simple graphs. Then $\tropNU$ is a rational polyhedral cone.
\end{conjecture}

We first prove the example of even cycles presented in the introduction. Our proof strategy involves the double hull to give an example of how it can be used. 

\begin{theorem}\label{thm:ecycles}
Let $\mathcal{U}=\{C_4, C_6, C_8, \ldots C_{2m}\}$. Then $$\tropNU=\left\{ \begin{array}{lll}
\mathbf{y}\in\mathbb{R}^{m-1} | & y_{2i-2}-2y_{2i}+y_{2i+2} \geq 0 & \forall \, 3 \leq i \leq m-1\\
& -y_{4} + y_{6} \geq 0 &\\
& m \cdot y_{2m-2} - (m-1) \cdot y_{2m} \geq 0 & 
\end{array} \right\}.$$ 

The set $\tropNU$ is the double hull of the following two doubly extreme rays: $\vec{1}$ and $(2,3,4,\dots,m)$. 
\end{theorem}

\begin{proof}
Let $Q_\mathcal{U}$ be the cone on the righthand side of the equation in the theorem. 

\textbf{Claim 1:} We have that $\tropNU \subseteq Q_\mathcal{U}$. 

The inequality $C_{2i} \geq C_{2i-2}$ for any $i\geq 2$ holds since there is a surjective homomorphism from $C_{2i}$ to $C_{2i-2}$. Now, let $A_G$ be the adjacency matrix of a graph $G$ on $n$ vertices. Recall that the entry $(v_1,v_2)$ of $A_G^i$ is the number of walks of length $i$ between the vertices $v_1$ and $v_2$ of $G$, and so, in particular, the entry $(v,v)$ of $A_G^i$ is the number of walks of length $i$ starting and ending at some vertex $v$. This means that $\hom(C_{2i};G)=\textup{tr}(A_G^{2i})$ which in turn is equal to $\lambda_1^{2i} + \ldots + \lambda_n^{2i}$, where $\lambda_1, \ldots, \lambda_n$ are the eigenvalues of $A_G$. Then $C_{2i-2}C_{2i+2}\geq C_{2i}^2$ and $C_{2i-2}^i \geq C_{2i}^{i-1}$ can be viewed as inequalities for $\ell_{2p}$ norms of real vectors. The second inequality is thus simply stating that $||\mathbf{\lambda}||_{2i} \geq ||\mathbf{\lambda}||_{2i-2}$ where $\mathbf{\lambda}=(\lambda_1, \ldots, \lambda_n)$, and the other follows from an application of H\"older's inequality: $||\mathbf{\lambda}_{i-1}||_2 ||\mathbf{\lambda}_{i+1}||_2 \geq ||\mathbf{\lambda}_{i-1} \cdot \mathbf{\lambda}_{i+1}||_1$ where $\mathbf{\lambda}_{u}=(\lambda_1^u, \ldots, \lambda_n^u)$. 

These binomial inequalities for $\NU$ imply the linear inequalities $y_{2i-2}-2y_{2i}+y_{2i+2}\geq 0$, $-y_{2i-2} + y_{2i} \geq 0$ and $ i \cdot y_{2i-2} - (i-1) \cdot y_{2i}\geq 0$ for $\tropNU$. Thus $\tropNU \subseteq Q_\mathcal{U}$.

\textbf{Claim 2:} The rays $(1,1, \ldots, 1)$ and $(2,3,4, \ldots, m)$ are in $\tropNU$.

First note that $(1,1, \ldots, 1)\in \tropNU$ since $\hom(C_{2i}, K_2)=2$ for all $2 \leq i \leq m$. Moreover, note that $(2,3,4,\ldots, m)\in \tropNU$ since $\frac{\log \hom(C_{2i}, K_n)}{\log n} \rightarrow 2i$ as $n\rightarrow \infty$. 

\textbf{Claim 3:} We have that $Q_\mathcal{U} \subseteq \textup{dh}(\textup{cone}((1,1,\ldots, 1), (2,3,4, \ldots, m)))$, and so $Q_\mathcal{U}\subseteq \tropNU$.

Let $\mathcal{S}=\textup{cone}((1,1\ldots, 1), (2,3,4, \ldots, m))$ and let $\mathcal{D}:=\operatorname{dh} (\mathcal{S})$. From Corollary \ref{cor:dualc}, we know that $\mathcal{D}^*= \sum_{i=2}^m \mathcal{S}^* \cap \mathcal{Q}_{2i}$ where $\mathcal{Q}_{2i}$ is the orthant of $\mathbb{R}^{m-1}$ consisting of all points where the coordinate corresponding to $y_{2i}$ is nonpositive, and all other coordinates are nonnegative, and where $\sum$ denotes the Minkowski sum. Therefore, the set of extreme rays of $\mathcal{D}^*$ is contained in the union of extreme rays of $\mathcal{S}^*\cap \mathcal{Q}_{2i}$ for $i\in \{2, 3, \ldots, m\}$. Note that the extreme rays of $\mathcal{D}^*$ correspond to linear inequalities for $\mathcal{D}$.

The cone $\mathcal{S}^*\cap Q_{2i}$ is defined by $m+1$ inequalities, and at least $m-2$ of them must be tight to form an extreme ray. Let's consider various possibilities for which inequalities are tight.

Consider a ray $\mathbf{r}=(r_4, r_6, \ldots, r_{2m})$ of $\mathcal{S}^*\cap Q_{2i}$ with $m-2$ tight inequalities including the two inequalities corresponding to $(1,1,\ldots, 1)$ and $(2,3,\ldots, m)$. Then there must be at least $m-4$ inequalities of $Q_{2i}$ that are tight, which implies that there are at most three coordinates of $\mathbf{r}$ that are non-zero, say $r_{2i}$, $r_{2j}$ and $r_{2k}$ where $i<j<k$. Then we know $r_{2i}+r_{2j}+r_{2k}=0$ and $ir_{2i}+jr_{2j}+kr_{2k}=0$. Without loss of generality, this implies that $r_{2i}=k-j$, $r_{2j}=i-k$, $r_{2k}=j-i$, and $r_{2l}=0$ for all $l\in \{2,3,\ldots, m\}\backslash\{i,j,k\}$. 

If $j=i+1$ and $k=j+1$, then those rays correspond exactly to the inequalities $y_{2j-2}-2y_{2j}+y_{2j+2} \geq 0$ for all $3 \leq j \leq m-1$. The general rays obtained are redundant since $(k-j)y_{2i} + (i-k) y_{2j} + (j-i) y_{2k} \geq 0$ can be written as the following conic combination of the previous case:

\begin{align*}
&\sum_{u=i}^{j-1} (k-j)\cdot(u-i+1)\cdot(y_{2i}-2y_{2i+2}+y_{2i+4} \geq 0)\\
+& \sum_{u=j}^{k-2} (j-i)\cdot (k-1-u)\cdot(y_{2i}-2y_{2i+2}+y_{2i+4} \geq 0).
\end{align*}

Now consider a ray $\mathbf{r}=(r_4, r_6, \ldots, r_{2m})$ of $\mathcal{S}^*\cap Q_{2i}$ with $m-2$ tight inequalities including the inequality corresponding to $(1,1,\ldots, 1)$, but not $(2,3,\ldots, m)$. Then there must be at least $m-3$ inequalities of $Q_{2i}$ that are tight, which implies that there are at most two coordinates of $\mathbf{r}$ that are non-zero, say $r_{2i}$ and $r_{2j}$ where $i<j$. Then we know $r_{2i}+r_{2j}=0$ and $ir_{2i}+jr_{2j}>0$. Without loss of generality, this implies that $r_{2i}=-1$, $r_{2j}=1$, and $r_{2k}=0$ for all $k\in \{2,3,\ldots, m\}\backslash\{i,j\}$. If $i=2$ and $j=3$, then this corresponds to the inequality $-y_4+y_6 \geq 0$. The other rays are redundant since $-y_{2j-2}+y_{2j} \geq 0$  for any $4 \leq j \leq m$ can be written as the following conic combination:

\begin{align*}
\sum_{l=3}^{j-1} & (y_{2l-2}-2y_{2l}+y_{2l+2} \geq 0)\\
+& (-y_4+y_6\geq 0),
\end{align*}
which implies that $-y_{2i} + y_{2j} \geq 0$ for all $i<j$. 

Now consider a ray $\mathbf{r}=(r_4, r_6, \ldots, r_{2m})$ of $\mathcal{S}^*\cap Q_{2i}$ with $m-2$ tight  inequalities including the inequality corresponding to $(2,3,\ldots, m)$, but not $(1,1,\ldots, 1)$. Then there must be at least $m-3$ inequalities of $Q_{2i}$ that are tight, which implies that there are at most two coordinates of $\mathbf{r}$ that are non-zero, say $r_{2i}$ and $r_{2j}$ where $i<j$. Then we know $ir_{2i}+jr_{2j}=0$ and $r_{2i}+r_{2j}>0$. Without loss of generality, this implies that $r_{2i}=j$, $r_{2j}=-i$, and $r_{2k}=0$ for all $k\in \{2,3,\ldots, m\}\backslash\{i,j\}$. If $i=m-1$ and $j=m$, then this corresponds to the inequality $my_{2m-2}-(m-1)y_m \geq 0$. The other rays are redundant since $(i+1)y_{2i}-iy_{2i+2} \geq 0$ for any $2 \leq i \leq m-2$ can be written as the following conic combination:

\begin{align*}
 \sum_{l=i+1}^{m-1} j\cdot &(y_{2l-2}-2y_{2l}+y_{2l+2} \geq 0)\\
+ & (my_{2m-2} - (m-1)y_m \geq 0).
\end{align*}
Note that this implies that $jy_{2i}-iy_{2j} \geq 0$ for any $2 \leq i < j \leq m$ since it can be written as a conic combination of the previous inequalities:

\begin{align*}
\sum_{l=i}^{j-1} \frac{i \cdot j}{l(l+1)} ((l+1)y_{2l} - ly_{2l+2} \geq 0).
\end{align*}

Finally, consider a ray $\mathbf{r}=(r_4, r_6, \ldots, r_{2m})$ of $\mathcal{S}^*\cap Q_{2i}$ tight with $m-2$ inequalities where neither inequalities corresponding to $(2,3,\ldots, m)$ or $(1,1,\ldots, 1)$ are tight. Then there is at most one coordinate of $\mathbf{r}$ that is non-zero, so without loss of generality, $\mathbf{r}=\mathbf{e}_{2i}$ for some $2 \leq i\leq m$. This corresponds to inequalities stating the nonnegativity of each inequalities. Since we've already seen that rays are non-decreasing, the only potentially non-redundant inequality would be $y_4 \geq 0$, however, it is redundant as it can be written as

\begin{align*}
\sum_{j=3}^m j\cdot &(y_{2j-2} -2y_{2j}+y_{2j+2}\geq 0)\\
+ 2 \cdot & (-y_4 + y_6 \geq 0)\\
+ & (my_{2m-2} - (m-1)y_{2m} \geq 0).
\end{align*}

Thus $Q_\mathcal{U} \subseteq \textup{dh}(\textup{cone}((1,1,\ldots, 1), (2,3,4, \ldots, m)))$, and we know $\textup{dh}(\textup{cone}((1,1,\ldots, 1), (2,3,4,\ldots, m))) \subseteq \tropNU$ because $(1,1,\ldots, 1)$ and $(2,3,4,\ldots, m)$ are in $\tropNU$ by the previous claim, so we obtain that $Q_\mathcal{U}\subseteq \tropNU$.
\end{proof}

We now do the same for odd cycles. 

\begin{theorem}
Let $\mathcal{U}=\{C_3, C_5, C_7, \ldots C_{2m+1}\}$. Then $$\tropNU=\left\{ \begin{array}{lll}
\mathbf{y}\in\mathbb{R}^{m} | &y_3 \geq 0& \\
&  -y_{2i-1} + y_{2i+1} \geq 0  & 2 \leq i \leq m
\end{array} \right\}.$$ 
The set $\tropNU$ is the double hull of the following doubly extreme rays: $(\underbrace{0, \ldots, 0}_{i-1 \textup{ times}}, 1, \ldots, 1)$ for $1 \leq i \leq m$. 
\end{theorem}

\begin{proof}
Let $Q_\mathcal{U}$ be the cone on the righthand side of the equation in the theorem. 

\textbf{Claim 1:} We have that $\tropNU \subseteq Q_\mathcal{U}$. 

The inequality $C_{2i-1} \geq C_{2i+1}$ for any $i\geq 2$ holds because there is a surjective homomorphism from $C_{2i+1}$ to $C_{2i-1}$. These binomial inequalities for $\NU$ imply the defining inequalities for $Q_\mathcal{U}$, and so they are valid on $\tropNU$. Moreover, since $\tropNU \subseteq \mathbb{R}_{\geq 0}^m$, we have that $y_3 \geq 0$ is a valid inequality for $\tropNU$. Thus $\tropNU \subseteq Q_\mathcal{U}$. 

\textbf{Claim 2:} The extreme rays of $Q_\mathcal{U}$ are $\mathbf{r}_i=(r_1, \ldots, r_m)$  for $1 \leq i \leq m$ where $$r_j=\left\{\begin{array}{ll}
0 & \textup{if } j < i,\\
1  & \textup{if } j \geq i.
\end{array}\right.$$

Observe that $Q_\mathcal{U}=\{\mathbf{y}\in\mathbb{R}^{m}| A\mathbf{y} \geq 0\}$ where

$$A:=\begin{pmatrix}
1 & 0 & 0 & 0 &  \ldots & 0 & 0 & 0\\
-1 & 1 & 0 & 0 & \ldots & 0 & 0 & 0\\
0 & -1 & 1 & 0 & \ldots & 0 & 0 & 0\\
0 & 0 & -1 & 1 & \ldots & 0 & 0 & 0\\
& & & &\ddots & & & \\
0 & 0 & 0 & 0 &  \ldots & -1 & 1 & 0\\
0 & 0 & 0 & 0 &  \ldots & 0 & -1 & 1\\
\end{pmatrix}$$
where the $i$th row of $A$ is $\mathbf{a}_i$. Since $A \in \mathbb{R}^{m\times m}$, the candidate extreme rays of $Q_\mathcal{U}$ are obtained by setting $m-1$ of the defining constraints $\mathbf{a}_i^\top \mathbf{y} \geq 0$ to equality. Since there are exactly $m$ constraints, it implies there are at most $m$ extreme rays. Observe that $$\mathbf{r}_i=(\underbrace{0, \ldots, 0}_{i-1 \textup{ times}}, 1, \ldots, 1)$$ satisfies all but the $i$th constraint at equality. Since $\mathbf{r}_i\in Q_\mathcal{U}$, it is an extreme ray for each $1 \leq i \leq m$.

\textbf{Claim 3:} The extreme rays of $Q_\mathcal{U}$ are in $\tropNU$, and hence $Q_\mathcal{U}=\tropNU$.

To realize $\mathbf{r}_i$, let $G_n$ be the graph with the disjoint union of $n$ $(2i+1)$-cycles and one 3-cycle. Then as $n \rightarrow \infty$, $\frac{\log \hom(C_{2j+1}; G_n)}{\log n} \rightarrow 1$ if $j\geq i$, and $\frac{\log \hom(C_{2j+1}; G_n)}{\log n} \rightarrow 0$ if $j<i$. 

\textbf{Claim 4:} The rays $\mathbf{r}_i$ are doubly extreme for $1 \leq i \leq m$.

First note that any ray in $\tropNU$ must be non-decreasing and nonnegative. 

Assume that $\mathbf{r}_i$ is not max-extreme for some $i$. Then there exists $\mathbf{t}:=(t_3, t_5, \ldots, t_{2m+1})$ and $\mathbf{t}':=(t'_3, t'_5, \ldots, t_{2m+1})$, both in $\tropNU$, such that $\mathbf{t}\oplus\mathbf{t'}=\mathbf{r}_i$ and neither $\mathbf{t}$ or $\mathbf{t}'$ is $\mathbf{r}_i$. Since $\mathbf{t},\mathbf{t}'\in \tropNU$, we know that $t_{2j+1}, t'_{2j+1} \geq 0$ for every $1 \leq j \leq m$, which implies that $t_{2j+1}=t'_{2j+1}=0$ for every $1 \leq j \leq i-1$. We also know that $t_{2j+1}\oplus t'_{2j+1}=1$ for every $i \leq j \leq m$, so $\max\{t_{2j+1},t'_{2j+1}\} =1$ and $t_{2j+1}, t'_{2j+1} \leq 1$ for every $i \leq j \leq m$. Without loss of generality, assume $t_{2i+1}=1$. Then $t_{2j+1}=1$ for $i \leq j \leq m$ since $t_{2j-1} \leq t_{2j+1}$ for every $2 \leq j \leq m$ because $\mathbf{t}\in \tropNU$, and each entry is bounded above by 1. Thus, $\mathbf{t} = \mathbf{r}_i$, a contradiction.

\end{proof}


We call a tree on $k+1$ vertices where one vertex is connected to all the other $k$ vertices a \emph{star with $k$ branches} and denote it by $S_k$.
In \cite{BRST2}, we computed $\tropGU$ when $\mathcal{U}$ is a collection of stars. As explained in the introduction, $\NU$ is more general than $\GU$, and we see here that $\tropNU$ is indeed slightly more complicated than $\tropGU$. In particular, the inequalities $S_2 \geq S_1$ and $S_{m-1}^m \geq S_{m}^{m-1}$ for $\NU$ do not have any equivalent inequalities for homomorphism densities.

\begin{theorem}
Let $\mathcal{U}=\{S_0, S_1, \ldots, S_m\}$ where $S_i$ is the star graph with $i$ branches. Then $$\tropNU=\left\{ \begin{array}{lll}
\mathbf{y}\in\mathbb{R}^{m+1} | & y_{i-1}-2y_{i}+y_{i+1} \geq 0 & \forall 1 \leq i \leq m-1\\
& -y_{1} + y_{2} \geq 0 &\\
& y_0 + y_{m-1} - y_m \geq 0 & \\
& m \cdot y_{m-1} - (m-1) \cdot y_{m} \geq 0 & 
\end{array} \right\}.$$ 

The set $\tropNU$ is the double hull of the following doubly extreme rays: $\vec{1}$, $(1,0,0,\ldots, 0)$, $(1,1,2,3,\ldots, m)$, and $(1,2,3,4,\ldots,m+1)$. 
\end{theorem}

\begin{proof}
Let $Q_\mathcal{U}$ be the cone on the righthand side of the equation in the theorem. 

\textbf{Claim 1:} We have that $\tropNU \subseteq Q_\mathcal{U}$. 

Let $G$ be a graph on $n$ vertices with degree sequence $\mathbf{d}=(d_1,\dots,d_n)$. Then $\hom(S_i;G)=d_1^i+\dots+d_n^i=||\mathbf{d}||_{i}^i$. For $m=0$, we define $0^0=1$. The above inequalities are known inequalities (e.g., H\"older inequalities) for $\ell_{p}$ norms of real nonnegative vectors applied to the degree sequence $\mathbf{d}$.

\textbf{Claim 2:} The extreme rays of $Q_\mathcal{U}$ are $(1,0,\ldots, 0)$, $(1,2,3,4,\ldots, m+1)$, $\mathbf{r}_i=(r_0, r_1, \ldots, r_m)$  for $2 \leq i \leq m$ where $$r_j=\left\{\begin{array}{ll}
i & \textup{if } j \leq i,\\
j & \textup{if } j \geq i,
\end{array}\right.$$
and $\mathbf{s}_i=(s_0, s_1, \ldots, s_m)$ for $1 \leq i \leq m-1$ where $$s_j=\left\{\begin{array}{ll}
i+j(i-1) & \textup{if } j\leq i,\\
i+i(i-1)+(j-i)i & \textup{if } j\geq i.
\end{array} \right.$$

Observe that $Q_\mathcal{U}=\{\mathbf{y}\in\mathbb{R}^{m+1}| A\mathbf{y} \geq 0\}$ where

$$A:=\begin{pmatrix}
0 & -1 & 1 & 0 & 0 & \ldots & 0 & 0 & 0\\
1 & -2 & 1 & 0 & 0 & \ldots & 0 & 0 & 0\\
0 & 1 & -2 & 1 & 0 & \ldots & 0 & 0 & 0\\
0 & 0 & 1 & -2 & 1 & \ldots & 0 & 0 & 0\\
& & & & & \ddots & & & \\
0 & 0 & 0 & 0 & 0 & \ldots & 1 & -2 & 1\\
1 & 0 & 0 & 0 & 0 & \ldots & 0 & 1 & -1\\
0 & 0 & 0 & 0 & 0 & \ldots & 0 & m & -(m-1)
\end{pmatrix}$$
where the $i$th row of $A$ is $\mathbf{a}_i$. Since $A \in \mathbb{R}^{(m+2)\times(m+1)}$, the extreme ray candidates of $Q_\mathcal{U}$ are obtained by setting $m$ of the defining constraints $\mathbf{a}_i^\top \mathbf{y} \geq 0$ to equality.

Let's first check that the vectors in the claim are indeed rays. Certainly, all of them satisfy all inequalities. Moreover, $(1,0,\ldots, 0)$ is tight with all inequalities except the second and the $(m+1)$th ones, and $(1, 2, 3, \ldots, m+1)$ with all but the first and $(m+2)$th one. Note that $\mathbf{r}_i$ is tight with all inequalities except the $(i+1)$th and $(m+1)$th, and that $\mathbf{s}_i$ is tight with all except for the first and $(i+1)$th ones. We still need to show that no other extreme ray exists.

We first show that there exists no extreme ray such that $\mathbf{a}_1^\top \mathbf{y} = 0$ and $\mathbf{a}_{m+1}^\top \mathbf{y}=0$. Consider a point $\mathbf{z}=(z_0, z_1, \ldots, z_m) \in Q_\mathcal{U}$ such that $z_m=z_{m-1}+z_0$ and such that $z_1=z_2$. From $\mathbf{a}_2^\top\mathbf{y}\geq 0$ and the assumption that $z_1=z_2$, we know that $z_0 \geq z_2$, and from $\mathbf{a}_{m+2}^\top \mathbf{y} \geq 0 $ and the assumption that $z_m=z_{m-1}+z_0$, we get that $z_{m-1} \geq (m-1)z_0$. Thus, the average gap $z_{i+1}-z_{i}$ for $2 \leq i \leq m-2$ is $\frac{z_{m-1}-z_2}{m-3} \geq \frac{(m-1)z_0 - z_0}{m-3}= \frac{m-2}{m-3}z_0 > z_0$. This implies there exists $2 \leq j \leq m-2$ such that $z_{j+1}-z_j > z_0$. Let $j^*$ be the biggest such index. Then $z_{j^*+1}-z_{j^*} > z_0$ and $z_{j^*+2}-z_{j^*+1} \leq z_0$, which implies that $z_{j^*}-2z_{j^*+1}+z_{j^*+2} < 0$, a contradiction to $z\in Q_\mathcal{U}$. Therefore, any extreme ray must be such that $\mathbf{a}_1^\top \mathbf{y} > 0$ or $\mathbf{a}_{m+1}^\top \mathbf{y}>0$ or both. The rays we have listed are the ones where exactly one of those two inequalities holds.

Let's finally check that there is no extreme ray where  $\mathbf{a}_1^\top \mathbf{y} > 0$ and $\mathbf{a}_{m+1}^\top \mathbf{y}> 0$. If there was, say $\mathbf{z}$, then all other inequalities would have to be tight. In particular, this means that $z_i-z_{i-1}=z_2-z_1=:k$ where $k>0$ (because of $\mathbf{a}_1^\top \mathbf{y} > 0$) for every $1 \leq i \leq m$ because of the tight log-convexity inequalities. Moreover, since $z_{m-1}=z_1+(m-2)k$, $z_{m}=z_1+(m-1)k$ and the last inequality is tight, we have that $m(z_1+(m-2))k=(m-1)(z_1+(m-1)k)$ which implies that $k=z_1$, i.e., $z_2=2z_1$. Thus, $z_0+z_{m-1}-z_m= 0 + (m-1)z_1 - mz_1 <0$. So there is no extreme ray where  $\mathbf{a}_1^\top \mathbf{y} > 0$ and $\mathbf{a}_{m+1}^\top \mathbf{y}> 0$.

\textbf{Claim 3:} The extreme rays of $Q_\mathcal{U}$ are in $\tropNU$, and hence $Q_\mathcal{U}=\tropNU$.

Just like for the even cycles case, a single edge and the complete graph on $n$ vertices as $n$ grows yield $(1,1,\ldots, 1)$ and $(1,2,3,4,\ldots, m+1)$ respectively. We also have that $(1,0,0, \ldots, 0)$ is realizable, for example with a graph on $n$ vertices as $n\rightarrow \infty$ containing a single edge.

Moreover, note that $\frac{\log \hom(S_i, S_n)}{\log n} \rightarrow i$ as $n\rightarrow \infty$ for $i\geq 1$, and $\frac{\log \hom(S_0, S_n)}{\log n} \rightarrow 1$, so $(1,1,2,3,4,\ldots, m) \in \tropNU$ as well. Observe that $\mathbf{r}_i=(1,1,2,3, \ldots, m)\oplus i\cdot (1,1,\ldots, 1)$, and $\mathbf{s}_i=i\cdot (1,1,2,3,\ldots, m)\oplus((i-1)\cdot (1,2,3,\ldots, m+1)+(1,1,1, \ldots, 1))$. This implies that $\tropNU$ is the double hull of $\vec{1}$, $(1,0,0,\ldots, 0)$, $(1,1,2,3,\ldots, m)$, and $(1,2,3,4,\ldots,m+1)$. 

\textbf{Claim 4:} The rays $\vec{1}$, $(1,0,0,\ldots, 0)$, $(1,1,2,3,\ldots, m)$, and $(1,2,3,4,\ldots,m+1)$ are doubly extreme.

First note that $y_i \leq y_{i+1}$ for every $1 \leq i \leq m-1$ since it can be written as the following linear combination:

\begin{align*}
\sum_{j=2}^i &(y_{j-1}-2y_j+y_{j+1} \geq 0)\\
+& (-y_1+y_2 \geq 0).
\end{align*} 

So any ray in $\tropNU$ is non-decreasing after the first component. Also, observe that $iy_{i-1}-(i-1)y_i \geq 0$ for every $2 \leq i \leq m$ since it can be written as the following linear combination:

\begin{align*}
\sum_{j=i}^{m-1} j\cdot & (y_{j-1}-2y_{j}+y_{j+1} \geq 0)\\
+ & (my_{m-1} - (m-1)y_m \geq 0).
\end{align*}

Moreover, $y_i \geq 0$ for every $0 \leq i \leq m$ since we know $\tropNU \subset \mathbb{R}^{m+1}_{\geq 0}$.

Now let $\mathbf{r}=(r_0, r_1, \ldots, r_m)$ and $\mathbf{r}'=(r'_0, r'_1, \ldots, r'_m)$, both in $\tropNU$. 

Suppose that $\mathbf{r}\oplus\mathbf{r}'=(1,0,0, \ldots, 0)$. Then since $r_i, r'_i \geq 0$, we have that $r_i=r'_i=0$ for all $1 \leq i \leq m$. Moreover, we know that $r_0 \oplus r'_0 = 1$. Without loss of generality, let $r_0=1$. Then $\mathbf{r}=(1,0,0, \ldots, 0)$, and so $(1,0,0,\ldots, 0)$ is max-extreme.

Suppose that $\mathbf{r}\oplus\mathbf{r}'=\vec{1}$. Without loss of generality, assume $r_1=1$ since we know that $r_1\oplus r'_1 =1$. Then $r_i = 1$ for all $1 \leq i \leq m$ since $\mathbf{r}$ is non-decreasing after $r_1$ because it is in $\tropNU$. Finally, we know that $r_0 \leq 1$ since $r_0\oplus r'_0 = 1$, and $r_0 \geq 1$ since $r_0 - 2r_1 + r_2 = r_0 -1 \geq 0$ since $\mathbf{r}\in \tropNU$. Thus $r_0 = 1$, and so $\mathbf{r}=\vec{1}$, which implies that $\vec{1}$ is max-extreme. 

Suppose that $\mathbf{r}\oplus\mathbf{r}'=(1,1,2,3,\ldots, m)$. Without loss of generality, assume $r_m=m$ since $r_m\oplus r'_m = m$. We know that $i\cdot r_{i-1} - (i-1) r_i \geq 0$ for all $2 \leq i \leq m$ since $\mathbf{r}\in \tropNU$. If there exists $i$ such that this is a strict inequality, let $i^*$ be the greatest such $i$. Then $r_{i^*-1} > i^*-1$ which is a contradiction to $r_{i^*-1} \oplus r'_{i^*-1}= i^*-1$. So this implies that $r_i=i$ for all $1 \leq i \leq m$. Additionally, we have that $r_0 \geq 1$ because $r_0 + r_{m-1} - r_m = r_0-1 \geq 0$ since $\mathbf{r}\in \tropNU$. We also know that $r_0 \leq 1$ since $r_0\oplus r'_0 =1$. Thus $r_0=1$, and $\mathbf{r}=(1,1,2,3,\ldots, m)$, and so $(1,1,2,3,\ldots, m)$ is max-extreme.

Finally, suppose that $\mathbf{r}\oplus \mathbf{r}'=(1,2,3,\ldots, m+1)$. Without loss of generality, assume $r_m=m+1$ since $r_m\oplus r'_m = m+1$. Since $\mathbf{r}\in \tropNU$, we know that $r_0+r_{m-1} - r_m \geq 0$ which implies that $r_0+r_{m-1} \geq m+1$. Since $r_0 \leq 1$ and $r_{m-1} \leq m$ because $r_0\oplus r'_0 = 1$ and $r_{m-1}\oplus r'_{m-1} = m$, we have that $r_0=1$ and $r_{m-1}=m$. Suppose that $r_{j+1}=j+2$ and $r_j=j+1$ for some $1 \leq j \leq m-1$. Then since $r_{j-1}-2r_j+r_{j+1} =r_{j-1}-j \geq 0$ because $\mathbf{r}\in \tropNU$, we have that $r_{j-1} \geq j$. Moreover, since $r_{j-1}\oplus r'_{j-1} = j$, we know that $r_{j-1} \leq j$. Thus, $r_{j-1}=j$, and by induction, we get that $r_i = i+1$ for all $0 \leq i \leq m+1$, and so that $\mathbf{r}=(1,2,3,\ldots, m+1)$. Therefore, $(1,2,3,\ldots, m+1)$ is max-extreme.

\end{proof}
On the other hand, the tropicalization of homomorphism numbers for complete graphs is not more complicated than the one for graph densities presented in \cite{BRST2}. We still go through it as it will be helpful to understand the tropicalization of simplicial complexes afterwards.

\begin{theorem}\label{thm:tropcompletegraphs}
Let $\mathcal{U}=\{K_1, \ldots, K_m\}$ where $K_i$ is a complete graph on $i$ vertices. Then $$\tropNU=\left\{ \begin{array}{lll}
\mathbf{y}\in\mathbb{R}^{m} | & i \cdot y_{i-1} - (i-1) \cdot y_{i} \geq 0 & 2 \leq i \leq m\\
& y_m \geq 0 & 
\end{array} \right\}.$$ 

The set $\tropNU$ is the double hull of the following doubly extreme rays: $\mathbf{r}_i=(1, 2, \ldots, i, 0, \ldots, 0)$ for $1 \leq i \leq m$. 

\end{theorem}

\begin{proof}
Let $Q_\mathcal{U}$ be the cone on the righthand side of the equation in the theorem. 

\textbf{Claim 1:} We have that $\tropNU \subseteq Q_\mathcal{U}$. 

The Kruskal-Katona theorem \cite{Kruskal, Katona} implies that $K_p^q - K_q^p \geq 0$ for any $2\leq p < q$, and it is also clear that $K_1^2 \geq K_2$. These binomial inequalities for $\NU$ imply the defining inequalities for $Q_\mathcal{U}$, and so they are valid on $\tropNU$.  Further, $y_m\geq 0$ is a valid inequality for $\tropNU$ since $\tropNU\subseteq \mathbb{R}_{\geq 0}^m$. Thus $\tropNU \subseteq Q_\mathcal{U}$. 

\textbf{Claim 2:} The extreme rays of $Q_\mathcal{U}$ are $\mathbf{r}_i=(r_1, \ldots, r_m)$  for $1 \leq i \leq m$ where $$r_j=\left\{\begin{array}{ll}
j & \textup{if } j \leq i,\\
0  & \textup{if } j > i.
\end{array}\right.$$

Observe that $Q_\mathcal{U}=\{\mathbf{y}\in\mathbb{R}^{m}| A\mathbf{y} \geq 0\}$ where

$$A:=\begin{pmatrix}
2 & -1 & 0 & 0 & \ldots & 0 & 0 & 0\\
0 & 3 & -2 & 0 & \ldots & 0 & 0 & 0\\
0 & 0 & 4 & -3 & \ldots & 0 & 0 & 0\\
& & & &\ddots & & & \\
0 & 0 & 0 & 0 &  \ldots & m-1 & -(m-2) & 0\\
0 & 0 & 0 & 0 &  \ldots & 0 & m & -(m-1)\\
0 & 0 & 0 & 0 &  \ldots & 0 & 0 & 1\\
\end{pmatrix}$$
where the $i$th row of $A$ is $\mathbf{a}_i$. Since $A \in \mathbb{R}^{m\times m}$, the extreme ray candidates of $Q_\mathcal{U}$ are obtained by setting $m-1$ of the defining constraints $\mathbf{a}_i^\top \mathbf{y} \geq 0$ to equality. Since there are exactly $m$ constraints, it implies there are at most $m$ extreme rays. Observe that $\mathbf{r}_i=(1, 2, \ldots, i, 0, \ldots, 0)$ satisfies all but the $i$th constraint at equality. Since $\mathbf{r}_i\in Q_\mathcal{U}$, it is an extreme ray for each $1 \leq i \leq m$.

\textbf{Claim 3:} The extreme rays of $Q_\mathcal{U}$ are in $\tropNU$, and hence $Q_\mathcal{U}=\tropNU$.

To realize $\mathbf{r}_i$, let $G_n$ be an $i$-partite complete graph where each part has $\frac{n}{i}$ vertices (i.e., a Tur\'an graph) with a disjoint copy of $K_m$.  Then as $n\rightarrow \infty$, $\frac{\log \hom(K_j; G_n)}{\log n} \rightarrow j$ if $j\leq i$ and $\frac{\log \hom(K_j; G_n)}{\log n} \rightarrow 0$ if $j>i$. 

\textbf{Claim 4:} The rays $\mathbf{r}_i$ for $1 \leq i \leq m$ are doubly extreme.

Let $\mathbf{t}=(t_1, \ldots, t_m)$ and $\mathbf{t}'=(t'_1, \ldots, t'_m)$, both in $\tropNU$, and suppose that $\mathbf{t}\oplus\mathbf{t}'=\mathbf{r}_i$ for some $1 \leq i \leq m$. Since $\mathbf{t}$ and $\mathbf{t}'$ are in $\tropNU$, $t_j, t'_j \geq 0$ for every $1 \leq j \leq m$. Thus since $t_j\oplus t'_j=0$ for $i+1 \leq j \leq m$, we have that $t_j=t_j'=0$ for all $i+1 \leq j \leq m$. Now, without loss of generality, assume that $t_i=i$ since $t_i\oplus t'_i=i$. Since $\mathbf{t}\in \tropNU$, we know that $jt_{j-1}-(j-1)t_j \geq 0$ for every $2 \leq j \leq m$. If there exists $2 \leq j \leq i$ such that $jt_{j-1}-(j-1)t_j >0$, let $j^*$ be the greatest such $j$. Then $t_{j^*-1} > j-1$, and so $t_{j*^-1}\oplus t'_{j^*-1} \neq j^*-1$, a contradiction. Thus, $jt_{j-1}-(j-1)t_j = 0$ for every $2 \leq j \leq i$, and $\mathbf{t}=\mathbf{r}_i$, which implies that $\mathbf{r}_i$ is max-extreme.

\end{proof}

Polyhedrality also shows up when tropicalizing counts other combinatorial structures as we show in the following examples about simplicial complexes and matroids.

Let $S_1$ and $S_2$ be simplicial complexes on $[m]$ and $[n]$.

\begin{definition}

\begin{itemize}
\item Let $S_1\otimes S_2$ be a simplicial complex on $[m]\times [n]$ as follows: $F\subset [m]\times [n]$ is a $k$-face of $S_1\otimes S_2$ if and only if $\pi_m(F)$ is a $k$-face of $S_1$ and $\pi_n(F)$ is a $k$-face of $S_2$.

\item For a simplicial complex $S$, define the scaled $f$-vector $\mathbf{f}_S$ as follows: the $k$th entry, denoted as $f_k$, records the number of $k$-faces of $S$ multiplied by $k!$.
\end{itemize}
\end{definition}

\begin{lemma}\label{lem:hadamardsimplicialcomplexes}
We have that $\mathbf{f}_{S\otimes T}=\mathbf{f}_S\cdot \mathbf{f}_T$ where $\cdot$ is the Hadamard product.
\end{lemma}

\begin{proof}
First, observe that the $k$th entry of $\mathbf{f}_S \cdot \mathbf{f}_T$ is $k!k!s_kt_k$ where $s_k$ is the number of $k$-faces in $S$ and $t_k$ is the number of $k$-faces in $T$. So we need to show that the number of $k$-faces in $S\otimes T$ is equal to $k!s_kt_k$. This is true since for each of the $s_kt_k$ pairs of $k$-faces in $S$ and $T$ (with one coming from $S$ and one coming from $T$), we get a $k$-face of $S\otimes T$ by matching the $k$ elements of the $k$-face of S and the $k$ elements of the $k$-face of $T$ that we are considering, and there are $k!$ (perfect) matchings. 
\end{proof}

Fix some $m\in \mathbb{N}$ and let $\mathcal{S}=\{\mathbf{f}_S=(f_0, f_1, \ldots, f_m) \in \mathbb{Z}_{\geq 0}^m \,\, \mid \,\, S \textup{ is a simplicial complex}\}$. Then the tropicalization of $\mathcal{S}$ is the tropicalization of the complete graph profile.

\begin{theorem}
We have that $$\trop(\mathcal{S})=\left\{ \begin{array}{lll}
\mathbf{y}\in\mathbb{R}^{m} | & (i+1) \cdot y_{i-1} - i \cdot y_{i} \geq 0 & 1 \leq i \leq m\\
& y_m \geq 0 & 
\end{array} \right\}.$$ 

The set $\trop(\mathcal{S})$ is the double hull of the following doubly extreme rays: $\mathbf{r}_i=(1, 2, \ldots, i+1, 0, \ldots, 0)$ for $0 \leq i \leq m$. 

\end{theorem}

\begin{proof}
Note that the $H$-description of the tropicalization $\trop (\mathcal{S})$ is the same as the description of the tropicalization complete graph profile in Theorem~\ref{thm:tropcompletegraphs} (with the indexing shifted by 1 to account for the fact that $f_0$ counts the number of vertices which were previously indexed by $1$). The vector of homomorphism numbers from complete graphs into a graph $G$ is the same as the scaled $f$-vector of the clique complex of $G$. It follows that $\trop (\mathcal{S})$ contains the tropicalization of the complete graph profile. Therefore it suffices to verify that the inequalities of the $H$-description correspond to valid binomial inequalities on $\mathcal{S}$. The Kruskal-Katona theorem \cite{Kruskal, Katona} implies that $f_{p-1}^q - f_{q-1}^p \geq 0$ for any $2\leq p < q$ for any simplicial complex, and it is also clear that $f_0^2 \geq f_1$ and $f_m \geq 1$ since we tropicalize only simplicial complexes that have at least one $m$-face.
\end{proof}

We now look at a particular type of simplicial complexes: the independence complex of a matroid $M$, i.e., the collection of subsets of the ground set of $M$ that are independent.  Even though the scaled $f$-vectors of the independence complexes of matroids do not have the Hadamard property, we still show that its tropicalization is polyhedral. The proof differs from previous examples since we do not know a priori that the tropicalization is a closed convex cone.



\begin{theorem}
Let $\mathcal{F}_m \in \mathbb{N}^{m+1}$ consist of all possible first $m+1$ entries of scaled $f$-vectors of matroids. Then $$\trop(\mathcal{F}_m) =\left\{ \begin{array}{lll}
\mathbf{y}\in\mathbb{R}^{m+1} | & 2 y_0 - y_1 \geq 0 & \\
& -y_{i-1}+2y_{i}-y_{i+1} \geq 0 & 1 \leq i \leq m-1\\
& -y_{m-1}+y_m \geq 0 & 
\end{array} \right\}.$$ 

\end{theorem}

\begin{proof}
Let $Q$ be the cone on the righthand side of the equation in the theorem. 

\textbf{Claim 1:} We have that $\trop(\mathcal{F}_m) \subseteq Q$. 

By a version of Mason's Conjecture \cite{masonconjecture}, it is known that the entries of the scaled $f$-vector are a log-concave sequence, i.e., they satisfy $f_{k-1}f_{k+1}\leq f_k^2$, which implies the first $m$ inequalities of $Q$ (since we do not keep track of $f_{-1}$) \cite{anari2018log, MR4172622}. The last inequality comes from the fact that every $(k-1)$-face is contained in a $k$-face (since $f_k\neq 0$ for vectors that we tropicalize), and every $k$-face leads to at most $k$ distinct $(k-1)$-faces. So $\trop(\mathcal{F}_m) \subseteq Q_\mathcal{U}$.

\textbf{Claim 2:} A point is in $Q$ if and only if it has the following shape: $(a_0, a_0+a_1, a_0+a_1+a_2, \ldots, a_0+a_1+\ldots + a_m)$ for some real numbers $a_0\geq a_1\geq  \ldots \geq a_m \geq 0$. 

Certainly, every point $(a_0, a_0+a_1, a_0+a_1+a_2, \ldots, a_0+a_1+\ldots + a_m)$ for any real numbers $a_0\geq a_1\geq  \ldots \geq a_m \geq 0$ satisfies the the defining inequalities of $Q$, and thus such points are in $Q$. 

For the other direction, first note that $y_j \geq y_{j-1}$ is a true inequality for $Q$ for every $1 \leq j \leq m$ since it can be written as the following conic combination of defining inequalities of $Q$:

\begin{align*}
\sum_{k=j}^{m-1} & (-y_{k-1}+2y_{k}-y_{k+1} \geq 0)\\
+& (-y_{m-1} + y_m \geq 0),
\end{align*}
and so, points in $Q$ are non-decreasing, and so certainly any point has the form $(a_0, a_0+a_1, a_0+a_1+a_2, \ldots, a_0+a_1+\ldots + a_m)$ for some real numbers $a_0, a_1,  \ldots, a_m \geq 0$

Since $-y_{i-1}+2y_{i}-y_{i+1}=\sum_{k=0}^{i-1}a_k - 2\sum_{k=0}^{i} a_k + \sum_{k=0}^{i+1}= a_i-a_{i+1}\geq 0$, we have that $a_i \geq a_{i+1}$ for every $1 \leq i \leq m-1$. Finally, since $2y_0-y_1=2a_0-(a_0+a_1)\geq 0$, we also have that $a_0 \geq a_1$. 

\textbf{Claim 3:} We have that $Q \subseteq \trop(\mathcal{F}_m)$, and hence $Q=\trop(\mathcal{F}_m)$.

Consider the point $(a_0, a_0+a_1, a_0+a_1+a_2, \ldots, a_0+a_1+\ldots + a_m)$ for some real numbers $a_0\geq a_1\geq  \ldots \geq a_m \geq 0$. Further, consider the partition matroid $M=(E, \mathcal{I})$ where $E=E_0\cup E_1\cup E_2 \cup \ldots \cup E_m$, $|E_i|=\lceil n^{a_i}\rceil$ and $\mathcal{I}=\left\{S\subseteq E: |S\cap E_i| \leq 1 \ \forall i\in \{0, 1, 2, \ldots, m\}\right\}$. As $n\rightarrow \infty$, the scaled $f$-vector for this partition matroid is such that $$f_i = \sum_{\substack{T\subseteq \{0,1,2,\ldots, m\}:\\ |T|=i+1}} i!  \prod_{k\in T} n^{a_k}$$ since one can take at most one element from each of the $E_k$'s and one needs to pick $i+1$ elements.   Taking the $\log$ of the scaled $f$-vector thus yields a positive multiple of  $(a_0, a_0+a_1, a_0+a_1+a_2, \ldots, a_0+a_1+\ldots + a_m)$ in the limit (since $m$ is finite and $\sum_{i=0}^j a_i \geq \sum_{i\in U} a_i$ for any $U\subseteq \{0,1,2,\ldots, m\}$ and $|U|=j+1$ for any $0\leq j\leq m$), and so each point of $Q$ is in $\trop(\mathcal{F}_m)$. 
\end{proof}
 \begin{remark}
 Note that some of the defining inequalities of the tropicalization of the matroid profile have more than one negative term. Therefore by Corollary~\ref{cor:onenegative} we see that the matroid profile is not max-closed (this is also not hard to show directly). However, we still observe the same polyhedrality phenomenon.
 \end{remark}

\section{Results of Kopparty and Rossman}\label{sec:kr}

In this section, we go over some ideas introduced by Kopparty and Rossman in \cite{koppartyrossman} which we use in our proofs in the next sections to compute the tropicalization of the number profile for paths. The results of Kopparty and Rossman also allow us to provide more evidence for Conjecture \ref{conj:polyhedrality}.

\subsection{A theorem by Kopparty-Rossman and its implication for the polyhedrality conjecture}\label{subsec:koppartyrossman}

The concept of homomorphism domination exponent was introduced in \cite{koppartyrossman}, though the idea behind it had been central to many problems in extremal graph theory for a long time. Let the \emph{homomorphism domination exponent} of a pair of graphs $F_1$ and $F_2$, denoted by $\HDE(F_1;F_2)$, be the maximum value of $c$ such that $\hom(F_1;G) \geq \hom(F_2;G)^c$ for every graph $G$. For example, Sidorenko's conjecture states that $\HDE(P_0^{2|E(H)|-|V(H)|}\cdot H; P_1)=|E(H)|$ for any bipartite graph $H$. The Erd\H{o}s-Simonovits conjecture was equivalent to showing that $\HDE(P_0^{k-t} P_{k}^t; P_t)=k$ for any $k\geq t$ where both $k$ and $t$ are odd. 

In \cite{koppartyrossman}, Kopparty and Rossman showed that $\HDE(F_1; F_2)$ can be found by solving a linear program when $F_1$ is chordal and $F_2$ is series-parallel. Since this is the case when $F_1$ and $F_2$ are unions of paths, this result will be very useful to us in the next sections to compute the tropicalization of the number profile of paths. 
We now recall their results.

Let $\Hom(F_1;F_2)$ be the set of homomorphisms from $F_1$ to $F_2$, and let $\mathcal{P}(F_2)$ be the polytope consisting of normalized $F_2$-polymatroidal functions:
\begin{align*}
\mathcal{P}(F_2)=\big\{p\in \mathbb{R}^{2^{|V(F_2)|}} | \,\, & p(\emptyset) = 0 &&\\
& p(V(F_2)) = 1 && \\
& p(A) \leq p(B) && \forall \  A \subseteq B \subseteq V(F_2)\\
& p(A\cap B) + p(A\cup B) \leq p(A)+p(B) && \forall \ A, B \subseteq V(F_2)\\
& p(A\cap B) + p(A\cup B) = p(A)+p(B) &&\forall \ A, B \subseteq V(F_2) \textup { such that } A\cap B \\
&&& \quad \quad \quad \textup{ separates } A\backslash B \textup{ and }B\backslash A\big\}.
\end{align*}
In the definition above, $A\cap B$ is said to separate $A\backslash B$ and $B\backslash A$ if there is no path in $F_2$ going from a vertex in $A\backslash B$ to a vertex in $B\backslash A$. 

\begin{theorem}[Kopparty-Rossman, 2011]\label{thm:kr11}
Let $F_1$ be a chordal graph and let $F_2$ be a series-parallel graph. Then

$$\HDE(F_1; F_2) = \min_{p\in \mathcal{P}(F_2)} \max_{\varphi\in \Hom(F_1;F_2)}  \sum_{S \subseteq \textup{MaxCliques}(F_1)} -(-1)^{|S|} p(\varphi(\cap S))\\$$
where $\textup{MaxCliques}(F_1)$ is the set of maximal cliques of $F_1$ and $\cap S$ is the intersection of the maximal cliques in $S$. 
\end{theorem}

\noindent To lighten the notation, we let $$c_{\varphi,F_1}(p) = \sum_{S\subseteq \textup{MaxCliques}(F_1)} -(-1)^{|S|} \cdot p(\varphi(\cap S)).$$ 

We combine the Kopparty-Rossman linear program with Corollary \ref{cor:onenegative} to show that $\tropNU$ is a rational polyhedral cone when $\U=\{H_1, \ldots, H_s\}$ contains only graphs that are chordal and series-parallel.

\begin{theorem}\label{thm:chordalseriesparallel}
For any family $\mathcal{U}$ consisting of finitely many chordal series-parallel graphs, $\tropNU$ is a rational polyhedral cone.
\end{theorem}

\begin{proof}
From Corollary \ref{cor:onenegative}, we know that it suffices to understand valid binomial inequalities for $\NU$ of the form $H_1^{\alpha_1} \cdots H_{j-1}^{\alpha_{j-1}} H_{j+1}^{\alpha_{j+1}} \cdots H_s^{\alpha_s}\geq H_j^{\beta}$ for some $j\in [s]$, and $\alpha_i\geq 0$ for each $i\in [s]\backslash\{j\}$. By taking logarithms, these binomial inequalities become linear inequalities of the form $\sum_{i\in [s]\backslash\{j\}}\alpha_i \log(\hom(H_i;G)) - \beta \log(\hom(H_j;G)) \geq 0$ valid on $\tropNU$. Observe that for nonnegative integers $\alpha_i$,  the largest $\beta$ such that the inequality $H_1^{\alpha_1} \cdots H_{j-1}^{\alpha_{j-1}} H_{j+1}^{\alpha_{j+1}} \cdots H_s^{\alpha_s}\geq H_j^{\beta}$ is valid on $\NU$ satisfies $$\beta = \HDE(H_1^{\alpha_1}\cdots H_{j-1}^{\alpha_{j-1}} H_{j+1}^{\alpha_{j+1}} \cdots H_s^{\alpha_s};  H_j).$$

It suffices for us to show that for nonnegative integers $\alpha_i$, the optimal $\beta$ is computed via an extended rational linear program, which also linearly incorporates the integer exponents $\alpha_i$ in the objective. Our desired result follows immediately from this, via approximation of exponents $\alpha_i$ by rational numbers. Without loss of generality we may assume that $j=1$. 

Let $F$ be the disjoint union of $H_2,\dots, H_s$.
Construct the polytope $\bar{\mathcal{P}}(H_1)$ as follows. For every point $p=(p_1, \ldots, p_{2^{|V(H_1)|}}) \in \mathcal{P}(H_1)$ and every $\varphi\in \Hom(F;H_1)$, create a point $(v_1, \ldots, v_{2^{|V(H_1)|}}, c_{\varphi, F}(p))$, and let $\bar{\mathcal{P}}(H_1)$ be the convex hull of these points. Taking the upper hull of $\bar{\mathcal{P}}(H_1)$ and projecting down, we obtain a regular subdivision of $\mathcal{P}(H_1)$ that we call $\mathcal{P}_{\textup{subdiv}}(H_1)$. Denote the regions of $\mathcal{P}_{\textup{subdiv}}(H_1)$ by $\mathcal{P}_{i}(H_1)$ for $i\in [k]$.

Observe that each region $\mathcal{P}_i(H_1)$ is associated to a homomorphism $\varphi_i$ such that, for any point $p\in \mathcal{P}_i(H_1)$, $c_{\varphi_i,F}(p) \geq c_{\varphi,F}(p)$ for all $\varphi\in \Hom(F;H_1)$. Note that when computing $\HDE(F;H_1)$ we are allowed to choose different homomorphisms to $H_1$ on different copies of $H_i$ contained in $F_1:=H_2^{\alpha_2}\cdots H_s^{\alpha_s}$. However, the optimal choice leads to choosing one homomorphism and repeating it $\alpha_i$ times.
Accordingly, for a homomorphism $\varphi: F \rightarrow H_1$ and $\bm{\alpha}=(\alpha_2, \ldots, \alpha_s)\in \mathbb{N}^{s-1}$, let $\bm{\alpha}\cdot \varphi$ be the homomorphism from $F_1$ to $H_1$ such that every copy of the component $H_i$ is sent to $H_1$ in the way determined by $\varphi$. Then

$$\textup{HDE}(F_1;H_1)= \min_{i\in [k]} \min_{p\in \mathcal{P}_i(H_1)} c_{\bm{\alpha}\cdot \varphi_i,F_1}(p).$$ 

This gives us the desired rational linear program and therefore when $\mathcal{U}$ contains only chordal series-parallel graphs, $\tropNU$ is a rational polyhedral cone. 
\end{proof}

 \subsection{Blow-up construction}\label{subsec:blowup}
 
We now present a few properties of a blow-up construction introduced in \cite{koppartyrossman} that gives one direction of Theorem \ref{thm:kr11} ($\leq$) and which will allow us to construct some rays in in the tropicalization of the number profile of paths later on. We explain their ideas only for paths as this is the only case that we need. 
 
 \begin{definition}
Let $P$ be a path of length $2f+1$ where $V(P)=\{0, 1, \ldots, 2f+1\}$ and $E(P)=\{\{i,i+1\}|0 \leq i\leq 2f\}$. Let $p: V(P)+E(P) \rightarrow \mathbb{R}_{\geq 0}$ be a weight function on the vertices and edges of $P$ such that $\max\{p(\{v\}), p(\{w\})\} \leq p(\{v,w\}) \leq p(\{v\})+p(\{w\})$ for every $\{v,w\} \in E(P)$. For every $m\in \mathbb{N}$, create a blow-up of $P$, called $B_m$. Let $V(B_m):=\{(v,a): v\in V(P), a\in [m^{p(\{v\})}]\}$, i.e., for every vertex $v\in V(P)$, construct a stable set with $m^{p(\{v\})}$ vertices. Each vertex $(v,a)$ is adjacent to $m^{p(\{v,v+1\})-p(\{v\})}$ vertices in the stable set corresponding to vertex $v+1$ of $P$ (if $0\leq v \leq 2f$) and to $m^{p(\{v-1, v\})-p(\{v\})}$ vertices in the stable set corresponding to vertex $v-1$ of $P$ (if $1 \leq v \leq 2f+1$). So there are $m^{p(\{v,v+1\})}$ edges between the stable sets corresponding to vertex $v$ and $v+1$ of $P$ for every $0 \leq v \leq 2f$. Note that it might not be possible to construct a graph with exactly those edges for every $m$, but it will always be possible to do so as $m\rightarrow \infty$.
\end{definition}
 
 Just as in \cite{koppartyrossman}, homomorphism numbers from any path with $i$ edges, $P_i$, to $B_m$ can be directly calculated from the weighted path $P$. We repeat this here for completeness. Given a homomorphism $\varphi: P_i \rightarrow P$, define $\Hom_\varphi$ to be the set of homomorphisms from $P_i$ to $B_m$ that send the vertices of $P_i$ to the stable sets of $B_m$ (which correspond to vertices of $P$) according to the homomorphism $\varphi$.
 $$\Hom_\varphi:=\{\vartheta\in \Hom(P_i;B_m) \,\, \mid \,\, \textup{there exists }k_j\in [m^{p(\{\varphi(j)\})}] \textup{ such that }\vartheta(j)=(\varphi(j),k_j) \textup{ for all } 0 \leq j \leq i\}.$$ Further, let $\hom_{\varphi}=|\Hom_\varphi|$. For a homomorphism $\varphi: P_i \rightarrow P$,  we define $p(\varphi)$ as $$p(\varphi):=\sum_{j=0}^{i-1} p(\{\varphi(j), \varphi(j+1)\})-\sum_{j=1}^{i-1} p(\{\varphi(j)\}).$$

\begin{lemma}\label{lem:BmP}
Let $\varphi: P_i \rightarrow P$ be a homomorphism. As $m\rightarrow \infty$, we have that
\begin{enumerate}
\item $\hom_\varphi = m^{p(\varphi)}$, and
\item $\hom(P_i; B_m)=\sum_{\varphi\in \Hom(P_i;P)} \hom_\varphi$.
\end{enumerate}
\end{lemma}

\begin{proof}
We count the number of homomorphisms in $\Hom_\varphi$. Vertex $0$ has $m^{p(\{\varphi(0)\})}$ vertices that it can be sent to. If $p(\{\varphi(0),\varphi(1)\})=p(\{\varphi(0)\})+p(\{\varphi(1)\})$, then there are $m^{p(\{\varphi(1)\})}$ choices for where vertex $1$ gets mapped to since it means that the stable sets corresponding to $\varphi(0)$ and $\varphi(1)$ form a complete bipartite graph in $B_m$. However, if there are less edges, the degree of each vertex in the part corresponding to $0$, and thus the number of choices for where vertex $1$ gets sent to,  will be $m^{p(\{\varphi(1)\})-(p(\{\varphi(0)\})+p(\{\varphi(1)\})-p(\{\varphi(0),\varphi(1)\}))}=m^{-p(\{\varphi(0)\})+p(\{\varphi(0),\varphi(1)\})}$. More generally, the degree of each vertex in the stable set of $B_m$ corresponding to vertex $j$ in $P_i$, and thus the number of choices for where vertex $j+1$ goes, is $m^{-p(\{\varphi(j)\})+p(\{\varphi(j),\varphi(j+1)\})}$. So the total number of homomorphisms in $\Hom_\varphi$ is $$m^{p(\{\varphi(0)\})}\cdot m^{-p(\{\varphi(0)\})+p(\{\varphi(0),\varphi(1)\})} \cdot m^{-p(\{\varphi(1)\})+p(\{\varphi(1),\varphi(2)\})} \cdots m^{-p(\{\varphi(i-1)\})+p(\{\varphi(i-1),\varphi(i)\})}=m^{p(\varphi)}$$

To see that $\hom(P_i; B_m)=\sum_{\varphi\in \Hom(P_i;P)} \hom_\varphi$, it suffices to note that the $\Hom_\varphi$'s partition $\Hom(P_i; B_m)$. 
\end{proof}

\begin{corollary}\label{cor:BuP}
As $m\rightarrow \infty$, $\frac{\log(\hom(P_i;B_m))}{\log m} \rightarrow p(\varphi^*)$ where $\varphi^*=\argmax_{\varphi\in \Hom(P_i; P)}\{p(\varphi)\}$.
\end{corollary}

\begin{proof}
From the previous lemma, we have that $\hom(P_i; B_m)=O(m^{p(\varphi^*)})$, and so $$\lim_{m\rightarrow \infty} \frac{\log(\hom(P_i;B_m))}{\log(m)} = p(\varphi^*).$$
\end{proof}

Therefore to calculate homomorphism numbers from any path $P_i$ to $B_m$ with $m$ large, it suffices to find which weighed homomorphisms are maximal, an observation that we will use in the next sections.


\section{Valid binomial inequalities involving paths}\label{sec:truebinomialinequalitieswithpaths}

The rest of the paper focuses on calculating the tropicalization of the number profile for paths and discussing its consequences. This turns out to be much more intricate than the examples that were given in Section \ref{sec:examplesoftropicalizations}.

In this section, we derive some pure binomial inequalities involving paths, including some that were previously known and others that we prove here. In some cases, we use the results of Kopparty and Rossman introduced in the previous section to do so. These binomial inequalities will be necessary to describe the tropicalization of the number profile for paths in the next section. In particular, in Section~\ref{subsec:generalizederdossimonovits}, we prove a generalization of the Erd\H{o}s-Simonovits conjecture. We proved the original conjecture in \cite{BRerdossimonovits}, unaware of the earlier proof by Sa\u{g}lam in \cite{saglam} using different techniques. 
In Section~\ref{subsec:evenoddeven}, we prove a different type of inequality for paths, which is also necessary for describing the number profile; to the best of our knowledge, this inequality is new. Finally, in \ref{subsec:findknownineqs}, we go over some inequalities involving paths that are simpler and/or that were previously known, and that are necessary to describe the number profile.

\subsection{A generalization of the Erd\H{o}s-Simonovits conjecture}\label{subsec:generalizederdossimonovits}

In \cite{BRerdossimonovits}, we explained how the Erd\H{o}s-Simonovits conjecture, which stated that $P_0^{2v-2w}P_{2v+1}^{2w+1} \geq P_{2w+1}^{2v+1}$ for any $w<v$, is equivalent to proving the validity of $P_0^2 P_{2v+1}^{2v-1} \geq P_{2v-1}^{2v+1}$. Here, we show a more general inequality $P_{2u}^{2v-2w}P_{2v+1}^{2w+1-2u} \geq P_{2w+1}^{2v+1-2u}$ for any $v>w>u$ (see Lemma \ref{lem:biggestgeneralizationES}) by first showing that $P_{2u}^2P_{2v+1}^{2v-1-2u}\geq P_{2v-1}^{2v+1-2u}$ for any $0\leq u<v-1$. These last inequalities are necessary to describe the tropicalization of the number profile of paths; the rest of the inequalities in Lemma \ref{lem:biggestgeneralizationES} are derived from them.
Note that the case when $u=0$ yields the original Erd\H{o}s-Simonovits conjecture.

Let  $V(P_{k})=\{0, 1, 2, \ldots, k\}$, and let $E(P_{k})=\{\{i,i+1\}| i \in \{0, 1, 2, \ldots, k\}\}$. Lemma 2.5 of \cite{koppartyrossman} implies that for any $p\in \mathcal{P}(P_{k})$ we have
$$p(V(P_{k}))=\sum_{\{i, i+1\}\in E(P_{k})}p(\{i,i+1\}) -\sum_{i \in \{1,\ldots,k-1\}}p(\{i\}).$$See  Lemma 2.2 of \cite{BRerdossimonovits} for details.

\begin{theorem}\label{thm:generalizationoferdossimonovits}
We have that $\HDE(P_{2u}^2 P_{2v+1}^{2v-1-2u}; P_{2v-1})=2v+1-2u$.
\end{theorem}

\begin{proof}
We assume that $u\geq 1$. (For the case when $u=0$, see our proof in \cite{BRerdossimonovits}.) We first show that $\HDE(P_{2u}^2 P_{2v+1}^{2v-1-2u}; P_{2v-1})\leq 2v+1-2u.$ For $i \in \{0, 1, 2, \ldots, 2v-1\}$ and $S\subseteq \{0,1,2,\ldots, 2v-1\}$, let $p_i\in \mathbb{R}^{2^{2v}}$ be such that $p_i(S)=1$ if $S$ contains $i$, and $p_i(S)=0$ otherwise. It's easy to check that $p_i \in \mathcal{P}(P_{2v-1})$. Let $p^*$ be the average of the $p_i$'s, i.e., $p^*= \frac{1}{2v} \sum_{i\in \{0,1,2,\ldots,2v-1\}} p_i$. In particular, this means that $p^*(\{i\})= \frac{1}{2v}$ for any $i\in \{0,1,2,\ldots, 2v-1\}$, and $p^*(\{i,i+1\})=\frac{2}{2v}$ for any $i\in \{0,\ldots, 2v-2\}$. Since $p^*$ is a convex combination of the $p_i$'s, we see that $p^*\in \mathcal{P}(P_{2v-1})$. For any homomorphism $\varphi$ from $P_{2u}^2 P_{2v+1}^{2v-1-2u}$ to $P_{2v-1}$, we have:

\begin{align*}
\sum_{S \subseteq \textup{MaxCliques}(P_{2u}^2 P_{2v+1}^{2v-1-2u})} & -(-1)^{|S|} p^*(\varphi(\cap S)) \\
&= 2\cdot(2u\cdot \frac{2}{2v} - (2u-1) \cdot \frac{1}{2v})+(2v-1-2u)\cdot((2v+1)\cdot\frac{2}{2v} -2v\cdot \frac{1}{2v}  \\
&= 2v-2u+1,
\end{align*} which implies that the optimal value of the linear program is at most $2v-2u+1$.

We now show that $\HDE(P_{2u}^2 P_{2v+1}^{2v-1-2u}; P_{2v-1})\geq 2v+1-2u$. For $u+1\leq i \leq 2v-1-u$ let $\varphi_{i}$ be the homomorphism from $P_{2v+1}$ to $P_{2v-1}$ such that $\varphi_{i}(j)=j$ for all $0\leq j\leq i$, and $\varphi_{i}(j)=j-2$ for all $i+1\leq i\leq 2v+1$. In other words, every edge of $P_{2v-1}$ is visited by $P_{2v+1}$ once, except for $\{i-1, i\}$ which is visited three times. Let $\vartheta_0$ be the homomorphism from $P_{2u}$ to $P_{2v-1}$ such that $\vartheta_0(j)=j$ for all $0\leq j \leq u$ and $\vartheta_0(j)=2u-j$ for all $u\leq j \leq 2u$. Symmetrically, let $\vartheta_{2v-1}$ be the homomorphism from $P_{2u}$ to $P_{2v-1}$ such that $\vartheta_{2v-1}(j)=2v-1-j$ for $0 \leq j \leq u$ and $\vartheta_{2v-1}(j)=2v-1-(2u-j)$ for all $u \leq j \leq 2u$. 

Let $\psi$ be the homomorphism from $P_{2u}^2 P_{2v+1}^{2v-1-2u}$ to $P_{2v-1}$ such that one copy of $P_{2u}$ gets sent to $P_{2v-1}$ via $\vartheta_0$ and the other copy of $P_{2u}$ via $\vartheta_{2v-1}$, and the $i$-th copy of $P_{2v+1}$ is mapped to $P_{2v-1}$ via $\varphi_i$ for $u+1\leq i\leq 2v-1-u$. This matches the fact that there are $2v-2u-1$ copies of of $P_{2v+1}$. See Figure \ref{fig:erdossimonovitsgeneralized} for an example when $u=2$ and $v=4$.

\begin{figure}[ht]
\includegraphics[width=0.65\textwidth]{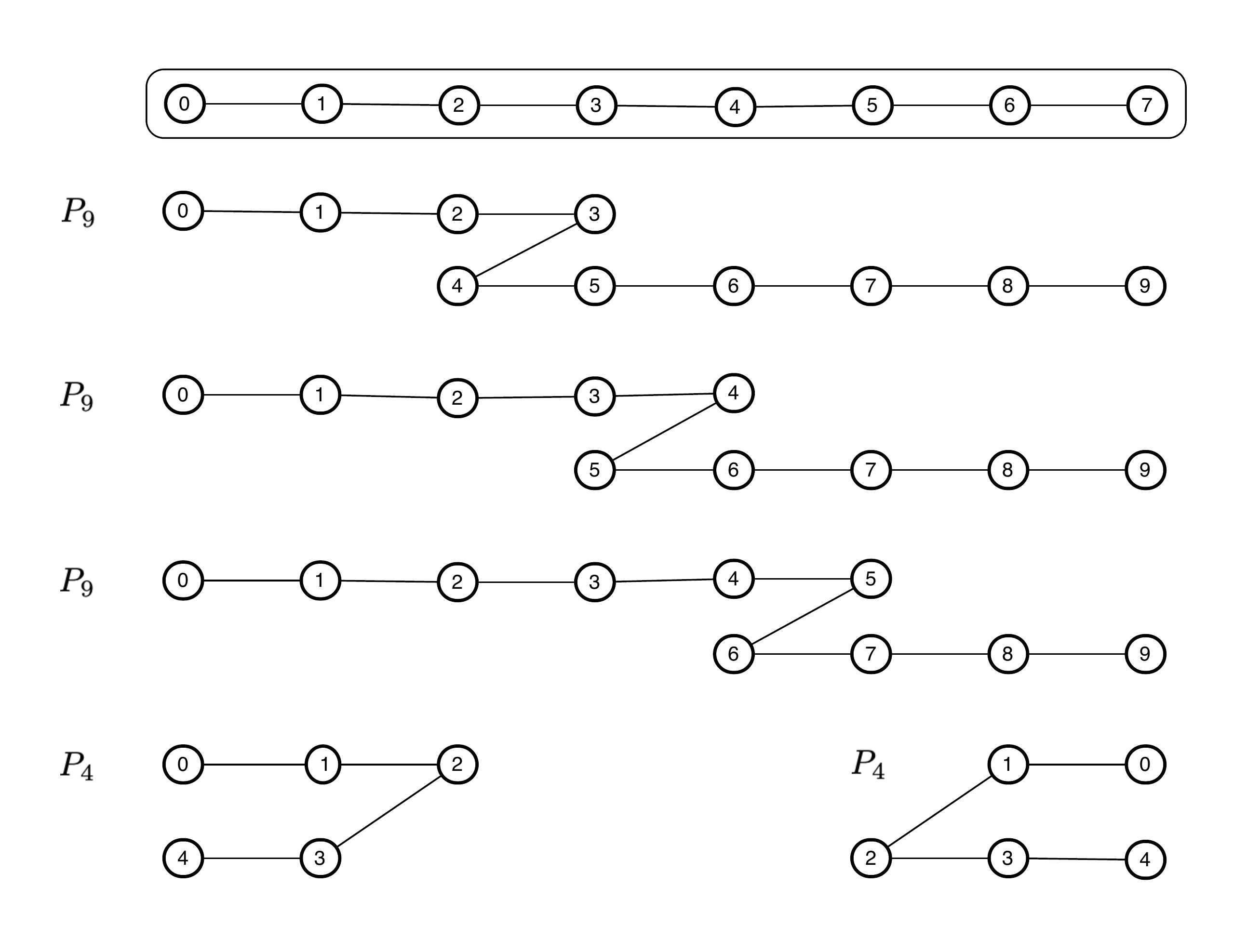}
\caption{\label{fig:erdossimonovitsgeneralized} How $\psi$ sends $P_4^2P_9^3$ to $P_7$}
\end{figure}

Now for any $p \in \mathcal{P}(P_{2v-1})$, we compute $$\sum_{S \subseteq \textup{MaxCliques}(P_{2u}^2 P_{2v+1}^{2v-1-2u})} -(-1)^{|S|} p(\psi(\cap S)).$$ Observe that only sets $S$ of size one or two contribute to the sum above since no three maximal cliques of $P_{2u}^2 P_{2v+1}^{2v-1-2u}$ intersect. Every edge $\{i, i+1\}$ where $u\leq i \leq 2v-1-u$ of $P_{2v-1}$ is covered by an image of an edge of $P_{2v+1}$ via $\psi$ exactly $2v-2u+1$ times. Every edge $\{i, i+1\}$ where $0\leq i< u$ or $2v-1-u<i\leq 2v-2$ of $P_{2v-1}$ is covered by an image of an edge of $P_{2v+1}$ via $\psi$ exactly $2v-2u-1$ times, and by an image of an edge of $P_{2u}$ twice, for a total of $2v-2u+1$. Every vertex $i$ for $u+1\leq i \leq 2v-2-u$ of $P_{2v-1}$ is covered by an image of an inner (non-end) vertex of $P_{2v+1}$ via $\psi$ exactly $2v-2u+1$ times. Vertices $u$ and $2v-1-u$ of $P_{2v-1}$ are covered by an image of an inner vertex of $P_{2v+1}$ via $\psi$ exactly $2v-2u$ times, and by an image of an inner vertex of $P_{2u}$ once, for a total of $2v-2u+1$ times. Finally, vertices $i$ for $1 \leq i \leq u-1$ and $2v-u\leq i \leq 2v-2$ are covered by an image of an inner vertex of $P_{2v+1}$ exactly $2v-2u-1$ times and by an image of an inner vertex of $P_{2u}$ exactly $2$ times via $\psi$, for a total of $2v-2u+1$ times. Note that each inner vertex of some copy of $P_{2v+1}$ (respectively $P_{2u}$) is the intersection of two maximal cliques (i.e., edges) of $P_{2v+1}$ (respectively $P_{2u}$), and thus the coefficient will be negative. Finally, the end vertices of $P_{2v-1}$ are not covered by the image of any inner vertices of $P_{2u}$ or $P_{2v+1}$ via $\psi$, and thus get a coefficient of $0$.  Accordingly, we have  
\begin{align*}
\sum_{S \subseteq \textup{MaxCliques}(P_{2u}^2 P_{2v+1}^{2v-1-2u})} -(-1)^{|S|} &p(\psi(\cap S))\\
&=(2v-2u+1)\left(\sum_{\{i, i+1\}\in E(P_{2v-1})}p(\{i,i+1\}) -\sum_{i \in \{1,\ldots,2v-2\}}p(\{i\})\right)\\
&=(2v-2u+1)p(V(P_{2v-1}))\\
&=2v-2u+1.
\end{align*}

The third line follows from $p(V(P_{2v-1}))=1$ since $p\in \mathcal{P}(P_{2v-1})$. Therefore, for every $p \in \mathcal{P}(P_{2v-1})$, there is a homomorphism that yields $2v-2u+1$, so we see that $\HDE(P_{2u}^2 P_{2v+1}^{2v-1-2u}; P_{2v-1})\geq 2v+1-2u$. This proves that $\HDE(P_{2u}^2 P_{2v+1}^{2v-1-2u}; P_{2v-1})= 2v+1-2u$.
\end{proof}

\subsection{New inequality involving paths}\label{subsec:evenoddeven}

Another inequality that is important for describing the tropicalization, and which we could not find in the literature is the following: $P_{2u-2}^{u+1}P_{2u+1} \geq P_{2u}^{u+1}$. Our proof is similar to the generalized Erd\H{o}s-Simonovits inequalities of the previous subsection.

\begin{theorem}\label{thm:generalizationoferdossimonovits2}
We have that $\HDE(P_{2u-2}^{u+1} P_{2u+1}; P_{2u})=u+1$.
\end{theorem}

\begin{proof}
We first show that  $\HDE(P_{2u-2}^{u+1} P_{2u+1}; P_{2u})\leq u+1$. Let $p^*$ be as in the previous proof. For any homomorphism $\varphi$ from $P_{2u-2}^{u+1} P_{2u+1}$ to $ P_{2u}$,

\begin{align*}
&\sum_{S \subseteq \textup{MaxCliques}(P_{2u-2}^{u+1} P_{2u+1})}  -(-1)^{|S|} p^*(\varphi(\cap S)) \\
&\quad \quad \quad\quad = (u+1)((2u-2)\cdot\frac{2}{2u+1}-(2u-3)\cdot\frac{1}{2u+1})+1\cdot((2u+1)\cdot\frac{2}{2u+1}-2u\cdot\frac{1}{2u+1})  \\
&\quad \quad \quad\quad = u+1,
\end{align*} which implies that the optimal value of the linear program is at most $u+1$.

We now need to show that $\HDE(P_{2u-2}^{u+1} P_{2u+1}; P_{2u})\geq u+1$. Actually, we instead show that $$\HDE(P_{2u-2}^{2u+2} P_{2u+1}^2; P_{2u})\geq 2u+2$$ which is equivalent. Let $\vartheta_0$ be the homomorphism from $P_{2u-2}$ to $P_{2u}$ such that $\vartheta_0(j)=j$ for all $0\leq j \leq u-1$ and $\vartheta_0(j)=2u-2-j$ for all $u-1\leq j \leq 2u-2$. Symmetrically, let $\vartheta_{2u}$ be the homomorphism from $P_{2u-2}$ to $P_{2u}$ such that $\vartheta_{2u}(j)=2u-j$ for $0 \leq j \leq u-1$ and $\vartheta_{2u}(j)=2u-(2u-2-j)$ for all $u-1 \leq j \leq 2u-2$. Furthermore, let $\varphi_{u-1}$ be the homomorphism from $P_{2u+1}$ to $P_{2u}$ that sends the vertices of $P_{2u+1}$ to $u-1$ and $u$ in an alternating fashion. Similarly, let $\varphi_{u}$ be the homomorphism from $P_{2u+1}$ to $P_{2u}$ that sends the vertices of $P_{2u+1}$ to $u$ and $u+1$ in an alternating fashion.  Let $\psi$ be the homomorphism from $P_{2u-2}^{u+1} P_{2u+1}$ to $P_{2u}$ such that $u+1$ copies of $P_{2u-2}$ get sent to $P_{2u}$ via $\vartheta_0$ and the other $u+1$ copies of $P_{2u}$ via $\vartheta_{2u}$, and one copy of $P_{2u+1}$ gets sent to $P_{2u}$ via $\varphi_{u-1}$, and the other copy via $\varphi_u$. See Figure \ref{fig:newinequality} for an example when $u=3$.

\begin{figure}[ht]
\includegraphics[width=0.5\textwidth]{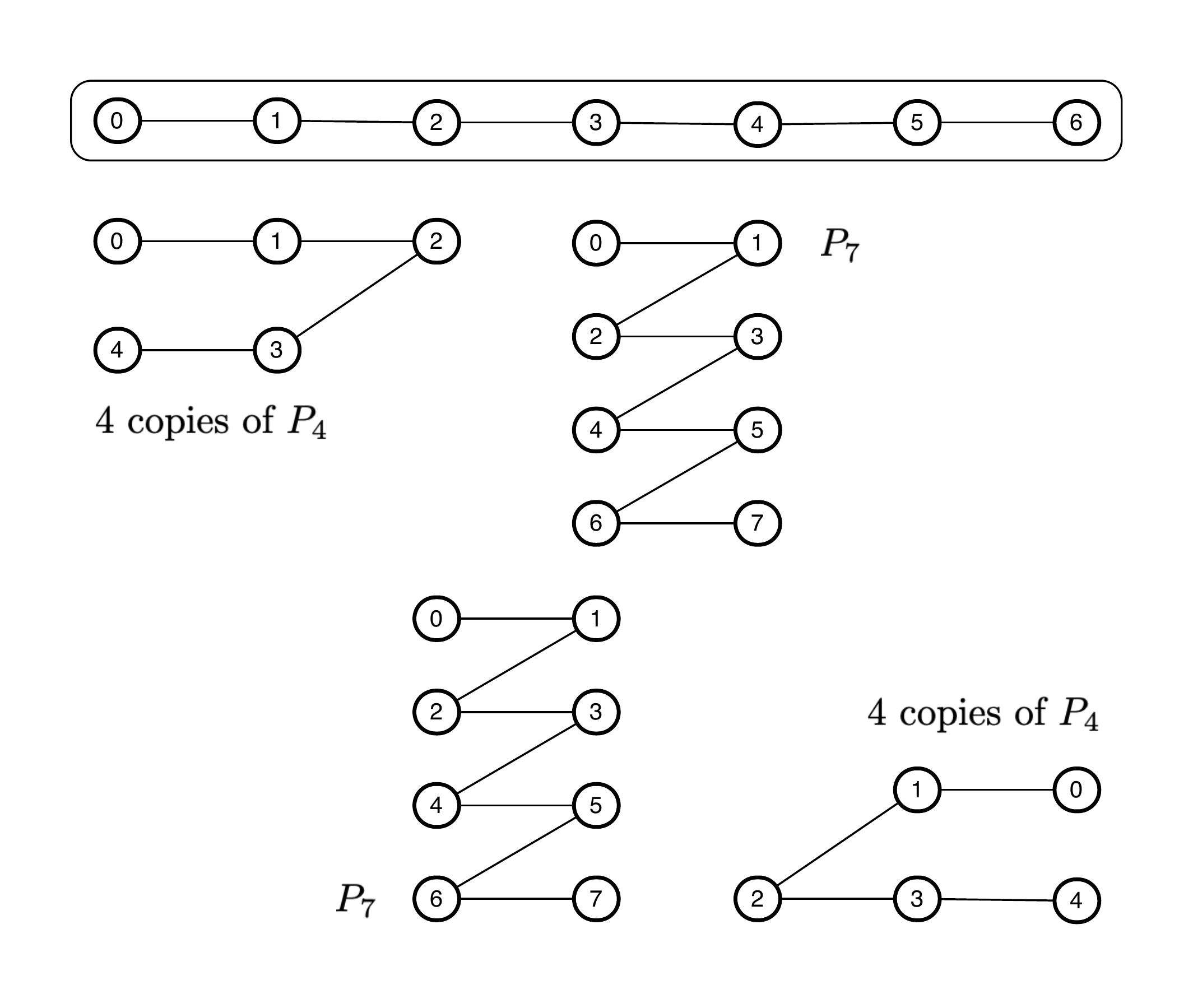}
\caption{\label{fig:newinequality} How $\psi$ sends $P_4^8 P_7^2$ to $P_6$}
\end{figure}

Now for any $p \in \mathcal{P}(P_{2u})$, we compute $$\sum_{S \subseteq \textup{MaxCliques}(P_{2u-2}^{u+1} P_{2u+1})} -(-1)^{|S|} p(\psi(\cap S)).$$ Observe that only sets $S$ of size one or two contribute in the above sum since no three maximal cliques of $P_{2u}^2 P_{2v+1}^{2v-1-2u}$ intersect. Every edge $\{i, i+1\}$ where $0 \leq i \leq u-2$ or $u+1 \leq 2u-1$ of $P_{2u}$ is covered by an image of an edge of $P_{2u-2}$ via $\psi$ exactly $2$ times for each of $u+1$ copies of $P_{2u-2}$ for a total of $2(u+1)$. The edges $\{u-1,u\}$ and $\{u,u+1\}$ of $P_{2u}$ are each covered by an image of an edge of $P_{2u+1}$ via $\psi$ exactly $2u+1$ times. Every vertex $i$ for $1\leq i \leq  u-2$ of $P_{2u}$ is covered by an image of an inner (non-end) vertex of $P_{2u-2}$ via $\psi$ exactly $2$ times for each of $u+1$ copies of $P_{2u-2}$ for a total of $2(u+1)$. Vertices $u-1$ and $u+1$ of $P_{2u}$ are covered by an image of an inner vertex of $P_{2u-2}$ via $\psi$ exactly one time for each of $u+1$ copies, and by an image of an inner vertex of $P_{2u+1}$ $u$ times, for a total of $2u+1$. Finally, vertex $u$ is covered by an image of an inner vertex of $P_{2u+1}$ exactly $u$ times for each of two copies, for a total of $2u$ times. Note that each inner vertex of some copy of $P_{2u-2}$ (respectively $P_{2u+1}$) is the intersection of two maximal cliques (i.e., edges) of $P_{2u-2}$ (respectively $P_{2u+1}$), and thus the coefficient will be negative. Finally, the end vertices of $P_{2u}$ are not covered by the image of any inner vertices of $P_{2u-2}$ or $P_{2u+1}$ via $\psi$, and thus get a coefficient of $0$.  Accordingly, we have  
\begin{align*}
&\sum_{S \subseteq \textup{MaxCliques}(P_{2u}^2 P_{2v+1}^{2v-1-2u})} -(-1)^{|S|} p(\psi(\cap S))\\
=&2(u+1)\left(\sum_{\{i, i+1\}\in E(P_{2u})}p(\{i,i+1\}) -\sum_{i \in \{1,\ldots,2v-2\}}p(\{i\})\right) \\
&- p(\{u-1,u\}) - p(\{u,u+1\}) + p(\{u-1\}) + 2p(\{u\}) + p(\{u+1\})\\
=&2(u+1)p(V(P_{2u}))- p(\{u-1,u\}) - p(\{u,u+1\}) + p(\{u-1\}) + 2p(\{u\}) + p(\{u+1\})\\
=&2(u+1)- p(\{u-1,u\}) - p(\{u,u+1\}) + p(\{u-1\}) + 2p(\{u\}) + p(\{u+1\})\\
\geq &2(u+1).
\end{align*}

The third line follows from $p(V(P_{2u}))=1$ since $p\in \mathcal{P}(P_{2u})$. Similarly, the fourth line follows from $p\in \mathcal{P}(P_{2u})$ because we know that $p(\{u-1\})+p(\{u\})\geq p(\{u-1,u\})$ and $p(\{u\})+p(\{u+1\})\geq p(\{u,u+1\})$. Therefore, for every $p \in \mathcal{P}(P_{2u})$, there is an homomorphism that yields at least $2(u+1)$, so we see that $\HDE(P_{2u-2}^{2u+2} P_{2u+1}^2; P_{2u})\geq 2(u+1)$. This proves that $\HDE(P_{2u-2}^{u+1} P_{2u+1}; P_{2u})=u+1$.
\end{proof}

\subsection{Other important inequalities for paths}\label{subsec:findknownineqs}

We now present other inequalities that are important for understanding the tropicalization of the number profile for paths.
In \cite{DressGutman}, Dress and Gutman showed that $P_{2a}P_{2b} \geq P_{a+b}^2$. One way to prove this is to observe that it can be obtained from the $2\times 2$ minor of the moment matrix with rows and columns indexed by a path of length $a$ and $b$, each with an end vertex labeled 1. 

One type of binomial inequality that follows from the definition of homomorphism numbers is $P_{u} \leq P_{u+1}$ for $u\geq 1$. Indeed, for any homomorphism $\varphi$ from $P_u$ to some graph $G$ (assuming there exists at least one), there exists a homomorphism $\varphi':P_{u+1}\rightarrow G$ be such that $\varphi'(j)=\varphi(j)$ for all $0\leq j \leq u$ and $\varphi'(u+1)=\varphi(k-1)$. 

The final type of inequalities we need are inclusion inequalities $P_{2u+1}P_{2v+1} \geq P_{2u+2v+3}$. These follow more or less from the definition of homomorphism numbers: as stated in Lemma 2.2(e) from \cite{koppartyrossman}, if there exists a surjective homomorphism from $F$ onto $G$, then $F\geq G$. Here, let $\psi:P_{2u+1}P_{2v+1} \rightarrow P_{2u+2v+3}$ be the homomorphism such that $\psi(i)=i$ for all $i \in V(P_{2u+1})$, and $\psi(i)=i+2u+2$ for all $i\in V(P_{2v+1})$. This covers all the vertices of $P_{2u+2v+3}$, so $P_{2u+1}P_{2v+1} \geq P_{2u+2v+3}$. 

If we want to only describe binomial inequalities in homomorphism numbers of paths that are extremal for homomorphisms into graphs with no isolated vertices, the inequalities we have discussed would suffice. However, since we would like to also consider target graphs with isolated vertices, we need the following variation for binomial inequalities that contain $P_0$ on the larger side of the inequality. If $P_0^a P_v^b \geq P_w^c$ is a valid inequality for some $a,b,c\geq 0$, then $P_1^a P_v^b \geq P_w^c$ is also valid. Indeed, this is clear for any graph $G$ such that $\hom(P_1;G) \geq \hom(P_0;G)$: in that case, $P_1^a P_v^b \geq P_0^a P_v^b \geq P_w^c$ is valid on such graphs. Now consider some graph $G$ such that $\hom(P_1;G) < \hom(P_0;G)$. Then $G$ must contain some isolated vertices. Let $G=G'\cup S$ where $S$ consists of all the isolated vertices of $G$. Then note that $\hom(P_i;G)=\hom(P_i;G')$ for all $i\geq 1$, and that $\hom(P_1;G')\geq \hom(P_0;G')$. Moreover, we know that $P_0^a P_v^b \geq P_w^c$ is a valid inequality for any graph, including $G'$, so $\hom(P_0;G')^a \hom(P_v;G')^b \geq \hom(P_w;G')^c$. Therefore, we have that 
\begin{align*}\hom(P_1;G)^a \hom(P_v;G)^b&= \hom(P_1;G')^a \hom(P_v;G')^b\\
&\geq \hom(P_0;G')^a \hom(P_v;G')^b \\ 
&\geq \hom(P_w;G')^c\\
&=\hom(P_w;G)^c
\end{align*} as desired, so $P_1^a P_v^b \geq P_0^a P_v^b \geq P_w^c$ is valid for all graphs.

\section{Tropicalization of path profile}\label{sec:troppaths}
Let $\mathcal{U}_{2n+1}=\{P_0, P_1, \ldots, P_{2n+1}\}$. In this section, we describe $\troptn$. We show that the pure binomial inequalities discussed in the previous section can be used to characterize $\troptn$. However, unlike in previous examples, we need to consider binomial inequalities involving larger paths.
 Indeed, we need to consider binomial inequalities in paths of length up to $4n+3$. Let $y_i:=\log P_i$ for $0 \leq i \leq 4n+3$. 
 Consider the cone $C$ in variables $y_0, \ldots, y_{4n+3}$ whose inequalities come from taking the log of some of the inequalities mentioned in the previous section:
\begin{align}
C=\big\{\mathbf{y}=&(y_0, y_1, y_2, \ldots, y_{4n+3}) \in \mathbb{R}^{4n+4} | && \nonumber\\
&y_{2u}-2y_{2u+1}+y_{2u+2} \geq 0, &&0 \leq u \leq 2n, \\
& y_{2u} - 2y_{2u+2} + y_{2u+4} \geq 0, && 0 \leq u \leq 2n-1, \\
& -y_{2u} + y_{2u+1} \geq 0, && 1\leq u \leq 2n+1,  \\
& y_{2u+1}+y_{2v+1}-y_{2u+2v+3} \geq 0, && 0 \leq u \leq v \textup{ such that } 2u+2v+3\leq 4n+3,\\
& 2y_{2u}-(2v+1-2u)y_{2v-1}+(2v-1-2u)y_{2v+1} \geq 0, && 0 \leq u<v-1, v\leq 2n+1,   \\
& (u+1)y_{2u-2}-(u+1)y_{2u}+y_{2u+1} \geq 0, &&1 \leq u\leq 2n+1, \\
& y_1-2y_2+y_4\geq 0, &&\\
& 2y_1-(2v+1)y_{2v-1}+(2v-1)y_{2v+1} \geq 0,  && 2 \leq v\leq 2n+1,\\
& -y_1+y_2 \geq 0 \big\}.
\end{align}

Let $\projC=\{(r_0, r_1, \ldots, r_{2n+1})\,\,|\,\,(r_0, r_1, \ldots, r_{4n+3})\in C\}$. Our overarching goal is to show that $\projC=\troptn$. 
Therefore, $C$ gives us \emph{a lifted representation} of $\troptn$. The inequality description of $\troptn$ only using the inequalities in paths of length at most $2n+1$ seems to be more complicated.

Using results of Section \ref{sec:truebinomialinequalitieswithpaths}, we quickly show that $\troptn$ is contained in $\projC$:

\begin{theorem} 
We have that $ \trop (\mathcal{N}_{\mathcal{U}_{2n+1}})  \subseteq \textup{proj}_{2n+1}(C)$.
 \end{theorem}
 \begin{proof}

We show that every inequality of $C$ corresponds to the logarithm of a valid binomial inequality for $\mathcal{N}_{\mathcal{U}_{4n+3}}$. 
Inequalities
\begin{itemize}
\item (5.1) come from the Dress-Gutman inequalities when $(a,b)=(u,u+1)$,
\item (5.2) come from the Dress-Gutman inequalities when $(a,b)=(u,u+2)$,
\item (5.3) come from $P_{u} \leq P_{u+1}$ for $u\geq 1$,
\item (5.4) come from the inclusion inequalities discussed in \ref{subsec:findknownineqs},
\item (5.5) come from the generalized Erd\H{o}s-Simonovits inequalities we certified in \ref{subsec:generalizederdossimonovits},
\item (5.6) come from new inequalities we certified in \ref{subsec:evenoddeven},
\end{itemize}

Inequalities (5.7), (5.8), and (5.9) are the variations of inequalities containing $P_0$ on the lefthand side discussed in Section~\ref{subsec:findknownineqs}. Indeed, inequality (5.7) is the variation for $P_0P_4\geq P_2^2$ from (5.2), inequalities (5.8) are the variation for Erd\H{o}s-Simonovits inequalities of (5.5) with $u=0$, and inequality (5.9) is the variation for $P_0P_2 \geq P_1^2$ of (5.1). We thus proved that $\trop (\mathcal{N}_{\mathcal{U}_{2n+1}}) \subseteq \projC$. 
\end{proof}

It remains to show that $\trop (\mathcal{N}_{\mathcal{U}_{2n+1}}) \supseteq \projC$. In Section \ref{subsec:moreaboutC}, we prove that certain inequalities are valid on $\projC$, which will help proving the results in further sections. Then in Section \ref{subsec:construction}, we construct a family of rays $\mathcal{R}$ that lies in $\trop (\mathcal{N}_{\mathcal{U}_{2n+1}}) $. Finally, in Section \ref{subsec:doublehullofconstruction}, we show that $\projC$ is contained in the double hull of $\mathcal{R}$, and thus $\trop (\mathcal{N}_{\mathcal{U}_{2n+1}}) = \projC$.
\subsection{Valid inequalities on $\projC$}\label{subsec:moreaboutC}

The following inequalities will be useful later on. Note that though all of these inequalities are in variables $y_0, y_1, \ldots, y_{2n+1}$, they are built from the defining inequalities of $C$ that involve variables up to $y_{4n+3}$. 
Note that each new (linear) inequality we prove to be valid on $\projC$ corresponds to a valid binomial inequality for $\mathcal{N}_{\mathcal{U}_{2n+1}}$ given that we have already shown that $\troptn \subseteq \projC$. 

\begin{lemma}\label{lem:forbprimenonneg}
Let $u, v\in \mathbb{N}$ such that $u \leq \frac{v}{2}$ and $v\leq 2n+1$. Then $$vy_{2u}-2uy_v \geq 0$$ is a valid inequality for $C$ and $\projC$, and equivalently $$P_{2u}^v \geq P_{v}^{2u}$$ is a valid inequality in graph homomorphism numbers.
\end{lemma}

\begin{proof}
First note that if $v$ is even, then it is sufficient to do it for the case when $2u=v-2$. Indeed, that case implies that for some $u',v'$ such that $2u'<v'-2$, we have
$$\frac{y_{2u'}}{y_{v'}}=\frac{y_{2u'}}{y_{2u'+2}}\cdot \frac{y_{2u'+2}}{y_{2u'+4}}\cdots \frac{y_{v'-2}}{y_v'} \geq \frac{u'}{u'+1}\cdot \frac{u'+1}{u'+2} \cdots \frac{\frac{v'-2}{2}}{\frac{v'}{2}} = \frac{u'}{\frac{v'}{2}}$$ as desired. We thus now show that when $v$ is even and $2u=v-2$, the inequality $vy_{v-2}-(v-2)y_v\geq 0$ holds on $C$  since one can easily check that it can be written as the following conic combination of defining inequalities of $C$:

\begin{align*}
2&(y_{v-2}-2y_{v-1}+y_v \geq 0)\\
+\sum_{w=\frac{v}{2}-1}^{v-3} (2v-4-2w)&(y_{2w}-2y_{2w+2}+y_{2w+4} \geq 0)\\
+2&(-y_{2v-2}+y_{2v-1} \geq 0)\\
+2&(2y_{v-1}-y_{2v-1} \geq 0).
\end{align*}

Indeed, note that these are all valid inequalities for $C$, namely inequalities of type (5.1), (5.2), (5.3) and (5.4), since no variable past $y_{4n+3}$ is used (in fact, the largest subscript of a variable involved in the conic combination is $4n-1$ which appears in the third and fourth row if $v=2n$).

Similarly, if $v$ is odd, it is sufficient to do the case when $2u=v-1$. Indeed that case together with the even case imply that for some $u', v'$ such that $2u'<v'-1$, we have $$\frac{y_{2u'}}{y_{v'}}=\frac{y_{2u'}}{y_{v'-1}}\cdot \frac{y_{v'-1}}{y_{v'}} \geq \frac{2u'}{v'-1}\cdot\frac{v'-1}{v'}= \frac{2u'}{v'}$$ as desired. We thus now show that when $v$ is odd and $2u=v-1$, the inequality $vy_{v-1} - (v-1)y_{v} \geq 0$ holds on $C$ since it can be written as the following conic combination of defining inequalities of $C$:

\begin{align*}
(\frac{v+1}{2})&(y_{v-1}-2y_v+y_{v+1} \geq 0)\\
+\sum_{w=\frac{v-1}{2}}^{v-2}(v-1-w)&(y_{2w}-2y_{2w+2}+y_{2w+4} \geq 0)\\
+&(-y_{2v}+y_{2v+1} \geq 0)\\
+&(2y_v - y_{2v+1} \geq 0).
\end{align*}

Indeed, note that these are still the same valid inequalities for $C$ as in the even case, and again no variable past $y_{4n+3}$ is used; here the largest subscript of a variable involved in the conic combination is actually $4n+3$, which appears in the third and fourth rows if $v=2n+1$. 

Moreover, note that this implies that $vy_{2u}-2uy_v \geq 0$ is also a valid inequality for $\projC$ for every $u \leq \frac{v}{2}$ and $v\leq 2n+1$ since such inequalities do not involve variables past $y_{2n+1}$. 
\end{proof}

\begin{lemma}\label{lem:forupperboundd0}
Let $u, v\in \mathbb{N}$ such that $u \leq \lfloor\frac{v-1}{2}\rfloor$ and $v\leq 2n+1$. Then $$(v-2u-1)y_{2u} - (v-2u)y_{2u+1}+y_v \geq 0$$ is a valid inequality on $C$ and $\projC$, and equivalently, $$P_{2u}^{v-2u-1} P_v \geq P_{2u+1}^{v-2u}$$ is a valid inequality in graph homomorphism numbers.
\end{lemma}

\begin{proof}
When $v$ is odd, one can see $(v-2u-1)y_{2u} - (v-2u)y_{2u+1}+y_v \geq 0$ is valid on $C$ since it is the following conic combination of defining inequalities of $C$:

\begin{align*}
\sum_{w=u+1}^{\frac{v-1}{2}} \frac{v-2u}{(2w+1-2u)(2w-1-2u)}(2y_{2u}-(2w+1-2u)y_{2w-1}+(2w-1-2u)y_{2w+1} \geq 0).
\end{align*}

Indeed, these are inequalities of type (5.5) for $C$ and no variable past $y_{4n+3}$ is used; that variable occurs when $v=2n+1$ and $w=\frac{v-1}{2}$. Moreover, note that
\begin{align*}
\sum_{w=u+1}^{\frac{v-1}{2}} \frac{1}{(2w+1-2u)(2w-1-2u)} & = \sum_{w=1}^{\frac{v-2u-1}{2}} \left( \frac{w}{2w-1}-\frac{w+1}{2(w+1)-1}\right)\\
& = 1 - \frac{v-2u+1}{2(v-2u)} = \frac{v-2u-1}{2(v-2u)}
\end{align*}
by canceling like terms. Therefore, the coefficient for $y_{2u}$ is $2(v-2u)\frac{v-2u-1}{2(v-2u)}=v-2u-1$.

Similarly, when $v$ is even, we have the following conic combination

\begin{align*}
\sum_{w=u}^{\frac{v}{2}-2} (\frac{v}{2}-w-1)&(y_{2w}-2y_{2w+2}+y_{2w+4} \geq 0)\\
+(\frac{v}{2}-u)&(y_{2u}-2y_{2u+1}+y_{2u+2} \geq 0)
\end{align*}
yielding $(v-2u-1)y_{2u} - (v-2u)y_{2u+1}+y_v \geq 0$ as desired. Here, note that the inequalities in the conic combination are of type (5.2) and (5.1) for $C$, and the $y_i$'s involved are such that $i\leq 4n < 4n+3$.

Moreover, note that this implies that $(v-2u-1)y_{2u} - (v-2u)y_{2u+1}+y_v \geq 0$ is also a valid inequality for $\projC$ for every $u \leq \lfloor\frac{v-1}{2}\rfloor$ and $v\leq 2n+1$ since such inequalities do not involve variables past $y_{2n+1}$. 
\end{proof}

\begin{lemma}\label{lem:forlowerboundd0}
Let $t, u, v\in \mathbb{N}$ such that $t\leq \lfloor \frac{v}{2} \rfloor$, $u\leq \lfloor\frac{v-1}{2}\rfloor$ and $v\leq 2n+1$. Then $$(v-2t)y_{2u+1}+2(u+1)y_{2t}-2(u+1)y_v \geq 0$$ is a valid inequality for $C$ and $\projC$, and equivalently, $$P_{2u+1}^{v-2t} P_{2t}^{2(u+1)} \geq P_v^{2(u+1)}$$ is a valid inequality in graph homomorphism numbers.
\end{lemma}

\begin{proof}
When $v$ is even, $(v-2t)y_{2u+1}+2(u+1)y_{2t}-2(u+1)y_v \geq 0$ can be written as the following conic combination of defining inequalities of $C$ of type (5.6), (5.2) and (5.4):

\begin{align*}
\frac{v-2t}{\bar{w}+1} & ((\frac{v}{2}+1)y_{\bar{v}-2} - (\frac{v}{2}+1)y_{\bar{v}} + y_{\bar{v}+1} \geq 0)\\
+ \sum_{w=t}^{\frac{v}{2}-2} 2(u+1)(w-t+1)&(y_{2w}-2y_{2w+2}+y_{2w+4} \geq 0)\\
+ \sum_{w=\frac{v}{2}-1}^{\frac{\bar{v}}{2}-2} 2(u+1)&(\frac{v}{2}-t)(y_{2w}-2y_{2w+2}+y_{2w+4} \geq 0)\\
+\sum_{w=1}^{\bar{w}} \frac{v-2t}{\bar{w}+1} & (y_{2u+1} + y_{2wu+2w-1} - y_{2u+2w+2uw+1} \geq 0)
\end{align*}
where $\bar{w}:=\min\{w\in \mathbb{N}| 2u+2w+2uw \geq v\}$ and $\bar{v}:= 2u+2\bar{w}+2u\bar{w}$. To see that these are indeed all valid inequalities for $C$, we need to show that there is no $y_i$ involved such that $i>4n+3$. First observe that $\bar{w}=\lceil\frac{v-2u}{2(1+u)}\rceil$. Then a quick check reveals that $\bar{v}\leq 2v$. The most worrisome variable in the conic combination is $y_{2u+2w+2uw+1}$ when $w=\bar{w}$, which is equal to $y_{\bar{v}+1}$ (which also appears in the first row), and given that $\bar{v}\leq 2v \leq 4n$, the inequalities involve no variable past $y_{4n+3}$. 

Now, to see that the conic combination yields the desired inequality, note that odd coordinates only appear in the last row except for $y_{\bar{v}+1}$ which appears also in the first row. Looking at $y_{2u+1}$, it appears twice in the bottom row when $w=1$, and once for every $w\in \{2, \ldots, \bar{w}\}$, and each time with a coefficient of $\frac{v-2t}{\bar{w}+1}$. So the coefficient for $y_{2u+1}$ is $(\bar{w}+1)\cdot \frac{v-2t}{\bar{w}+1}= v-2t$ as desired. Now note that $2(w+1)u+2(w+1)-1=2u+2w+2uw+1$, so other odd coordinates in the last row cancel themselves out for subsequent $w$'s, except for $y_{2u+2\bar{w}+2u\bar{w}+1}$ which only appears with a coefficient of $-\frac{v-2t}{\bar{w}+1}$ when $w = \bar{w}$. However, as explained before, since $2u+2\bar{w}+2u\bar{w}+1=\bar{v}+1$, it also appears in the first row with a coefficient of $\frac{v-2t}{\bar{w}+1}$ which thus means that that the coefficient for $y_{2u+2\bar{w}+2u\bar{w}+1}$ in the total is zero as desired.

It's easy to check that the coefficients of most even coordinates cancel out. The only slightly complicated cases are $y_{v-2}$ that appears in the second row both when $w=\frac{v}{2}-3$ and $w=\frac{v}{2}-2$ as well as in the third row when $w=\frac{v}{2}-1$, $y_{\bar{v}-2}$ that appears in the first row as well as in the third row when $w=\frac{\bar{v}}{2}-3$ and $w=\frac{\bar{v}}{2}-2$, and $y_{\bar{v}}$ which appears in the first row as well as in the third row when $w=\frac{\bar{v}}{2}-2$. Finally, $y_{2t}$ only appears in the second row when $w=t$ and yields a coefficient of $2(u+1)$ as desired, $y_{v}$ appears in the second row when $w=\frac{v}{2}-2$ and in the third row when $w$ is either $\frac{v}{2}-1$ or $\frac{v}{2}$, for a total coefficient of $-2(u+1)$ also as desired.

When $v$ is odd, we do the same thing in two different cases. In the first case, we assume that $0 \leq u \leq t-1$:

\begin{align*}
(v-2t)&(y_{2u+1}+y_{v-2u-2}-y_v \geq 0)\\
+\sum_{w=\frac{v-2u-1}{2}}^{\frac{v-1}{2}} \frac{v-2t}{(2w+1-2t)(2w-1-2t)}&(2y_{2t}-(2w+1-2t)y_{2w-1}+(2w-1-2t)y_{2w+1} \geq 0).
\end{align*}

Here it is easy to see that no variables involved are past $y_{4n+3}$, and this is thus a conic combination of inequalities of type (5.4) and (5.5) of $C$. In the second case, we assume that $t \leq u$, and have a conic combination of inequalities of type (5.4) and (5.5) again:

\begin{align*}
(v-2t)&(2y_{2u+1}-y_{4u+3} \geq 0)\\
+ \sum_{w=u+1}^{\frac{v-1}{2}} \frac{(v-2t)(2u-2t+1)}{(2w+1-2t)(2w-1-2t)}&(2y_{2t} - (2w+1-2t)y_{2w-1} + (2w-1-2t)y_{2w+1} \geq 0)\\
+\sum_{w=\frac{v+1}{2}}^{2u+1} \frac{(4u+3-2t)(v-2t)}{(2w+1-2t)(2w-1-2t)}&(2y_{2t}-(2w+1-2t)y_{2w-1} + (2w-1-2t)y_{2w+1} \geq 0).
\end{align*}

Here again it is easy to see that no variables involved are past $y_{4n+3}$.

Finally, note that this implies that $(v-2t)y_{2u+1}+2(u+1)y_{2t}-2(u+1)y_v \geq 0$ is also a valid inequality for $\projC$ for every $t\leq \lfloor \frac{v}{2} \rfloor$, $u\leq \lfloor\frac{v-1}{2}\rfloor$ and $v\leq 2n+1$ since such inequalities do not involve variables past $y_{2n+1}$. 

\end{proof}

%
%
%
%
 
 \begin{lemma}\label{lem:lowerboundd0leven}
 Let $0 \leq u<v \leq n$. Then $$(2v+1)y_{2u}-(2u+1)y_{2v} \geq 0$$ is a valid inequality on $C$ and $\projC$, and equivalently, $$P_{2u}^{2v+1} \geq P_{2v}^{2u+1}$$ is a valid inequality in graph homomorphism numbers.
 \end{lemma}
 
 \begin{proof}
 We first show the case when $u=v-1$, i.e., we show that $(2v+1)y_{2v-2}-(2v-1)y_{2v} \geq 0$. This inequality holds because it can be written as the following conic combination of inequalities of types (5.4), (5.1), (5.2), and (5.6) for $C$:
 \begin{align*}
& 1\cdot (2y_{2v-1} - y_{4v-1} \geq 0)\\
+ & 1 \cdot (y_{2v-2}-2y_{2v-1}+y_{2v} \geq 0)\\
+&\sum_{w=v-1}^{2v-3} 2v(y_{2w}-2y_{2w+2}+y_{2w+4} \geq 0)\\
+& 1 \cdot (2vy_{4v-4} - 2v y_{4v-2} + 1 y_{4v-1} \geq 0).
 \end{align*}
 
 Observe that it is easy to see that no variables involved are past $y_{4n+3}$. Now note that this implies that for a general $u<v$, we have that
 
 \begin{align*}
 \frac{y_{2u}}{y_{2v}}&=\frac{y_{2u}}{y_{2u+2}}\cdot \frac{y_{2u+2}}{y_{2u+4}}\cdots \frac{y_{2v-4}}{y_{2v-2}}\cdot \frac{y_{2v-2}}{y_{2v}}\\
 &\geq \frac{2u+1}{2u+3}\cdot \frac{2u+3}{2u+5}\cdots \frac{2i-3}{2i-1}\cdot \frac{2i-1}{2i+1}\\
 & = \frac{2u+1}{2v+1}
 \end{align*}
 as desired.
 
 Finally, note that this implies that $(2v+1)y_{2u}-(2u+1)y_{2v} \geq 0$ is also a valid inequality for $\projC$ for every $0 \leq u<v \leq n$  since such inequalities do not involve variables past $y_{2n+1}$. 
 \end{proof}
 
 \begin{lemma}\label{lem:biggestgeneralizationES}
 Let $0\leq u<v$, $l\geq 0$, $v+l\leq n$. Then $$2(l+1) y_{2u} - (2v+2l+1-2u)y_{2v-1} + (2v-1-2u) y_{2(v+l)+1}\geq 0$$ is a valid inequality for $C$ and $\projC$, and equivalently, $$P_{2u}^{2(l+1)} P_{2(v+l)+1}^{2v-1-2u} \geq P_{2v-1}^{2v+2l+1-2u}$$ is a valid inequality in graph homomorphism numbers. 
 \end{lemma}
 
 \begin{proof}
 This broader generalization of the Erd\H{o}s-Simonovits inequalities can be obtained as a conic combination from the first generalization of those same inequalities, i.e., inequalities of type (5.5) for $C$:
 
 \begin{align*}
 \sum_{w=v}^{v+l} \frac{(2v-1-2u)(2v+2l+1-2u)}{(2w+1-2u)(2w-1-2u)}\left(2y_{2u} - (2w+1-2u)y_{2w-1}+(2w-1-2u)y_{2w+1} \geq 0 \right).
 \end{align*}
 
Indeed, it's easy to check that no variable past $y_{4n+3}$ is used in this conic combination. To see that we get the desired inequality, note that the coefficient for $y_{2u}$ is
 
 \begin{align*}
 &2(2v-1-2u)(2v+2l+1-2u)\sum_{z=v-u}^{v+l-u} \frac{1}{(2z-1)(2z+1)} \\
 =& 2(2v-1-2u)(2v+2l+1-2u)\left( \sum_{z=1}^{v+l-u} \frac{1}{(2z-1)(2z+1)}- \sum_{z=1}^{v-u-1} \frac{1}{(2z-1)(2z+1)}\right)\\
 =& 2(2v-1-2u)(2v+2l+1-2u) \left( \frac{v+l-u}{2v+2l-2u+1} - \frac{v-u-1}{2v-2u-1} \right)\\
 = & 2(l+1). 
 \end{align*}
 
 The other coefficients are simple to check.  Finally, note that this implies that $2(l+1) y_{2u} - (2v+2l+1-2u)y_{2v-1} + (2v-1-2u) y_{2(v+l)+1}\geq 0$ is also a valid inequality for $\projC$ for every $0\leq u<v$, $l\geq 0$, and $v+l\leq n$ since such inequalities do not involve variables past $y_{2n+1}$. 
 \end{proof}
 
 
 \subsection{Construction of some rays in $\troptn$}\label{subsec:construction} 

We use the blow-up construction introduced in Section \ref{subsec:blowup} to construct some rays of $\troptn$ using the function $p$ defined below.
Recall from Section \ref{subsec:blowup} that calculating the number of homomorphisms from some path $P_i$ to the blow-up comes down to finding which homomorphisms weighed with $p$ are maximal within the original graph.

\begin{definition}\label{def:blowupgraph}
Let $P$ be a path of length $2f+1$ where $f\leq n$ and with vertex set $V(P)=\{0, 1, 2, \ldots, 2f+1\}$, and edge set $E(P)=\{\{0,1\}, \{1, 2\}, \ldots, \{2f, 2f+1\}\}$. Let $b\geq s\geq 0$ and $\mathbf{d}=(d_0, d_1, \ldots, d_f,0,\ldots, 0)\in \mathbb{R}^{n+1}$ be such that 
\begin{enumerate}
\item $d_v \geq 0$ for all $v\geq 0$,
\item $d_0=b-s$,
\item $d_0 + \ldots + d_u \geq d_{v-u} + \ldots + d_v$ for all $0\leq u< v$,
\item $d_u=d_{f-u}$ for all $u<t$,
\item $d_u=0$ for all $t+1\leq u \leq f-t-1$,
\item $d_t\geq d_{f-t} \geq 0$,
\item $2\sum_{v=1}^f d_v \leq s$,
\item $s \geq d_0$ if $t=0$,
\end{enumerate}
where $t$ is the largest index such that $t\leq \frac{f}{2}$ and $d_t >0$. Define $p:=p_{b, s, \mathbf{d}}:V(P)\cup E(P) \rightarrow \mathbb{R}$ to be such that

 \begin{itemize}
\item  $p(\{0\})=b$,
\item $p(\{1\})=b-s=d_0$,
\item $p(\{2u+1\})=d_0+2\sum_{v=1}^u d_v$ for all $1\leq u \leq \lfloor \frac{f-1}{2} \rfloor$,
\item $p(\{2u\})=s-d_0-2\sum_{v=1}^{u-1} d_v$ for all $1 \leq u\leq \lfloor\frac{f}{2}\rfloor$,
\item $p(\{2f+1-u\})=p(u)$ for all $0\leq u \leq f$,
\item $p(\{2u, 2u+1\})=s+d_u$ for all $0\leq u \leq \lfloor \frac{f-1}{2}\rfloor$,
\item $p(\{2u-1, 2u\})=s$ for all $1\leq u \leq \lfloor \frac{f-1}{2} \rfloor $,
\item $p(\{f,f+1\})=s+d_{\frac{f}{2}}-d_t+d_{f-t}$ if $f$ is even,
\item $p(\{f,f+1\})=p(f)+p(f+1)-d_t+d_{f-t}$ if $f$ is odd,
\item $p(\{2f-u,2f+1-u\})=p(\{u,u+1\})$ for all $0 \leq u \leq f-1$.
\end{itemize}

Note that all weights are nonnegative, which is straightforward to check in most cases. The most ambiguous case is $p(\{2u\})$.  We know that $2\sum_{v=1}^f d_v \leq s$, and if $t>0$, then $d_0=d_f$, so $d_0 + 2\sum_{v=1}^{u-1} d_v \leq 2\sum_{v=1}^f d_v \leq s$, and thus $p(\{2u\})=s-d_0-2\sum_{v=1}^{u-1} d_v \geq 0$ for $1 \leq u\leq \lfloor\frac{f}{2}\rfloor$. If $t=0$, then $p(\{2u\}) \geq 0$ as well since $s \geq d_0$. 
\end{definition}

\begin{example}
Let $n \geq f=6$, $b=34$, $s=30$, $\mathbf{d}=(4,3,3,0,1,3,4, 0, \ldots, 0)\in \mathbb{R}^{n+1}$. Then the weights $p$ on the vertices and edges of $P_{13}$ are given in Figure \ref{fig:exampleconstruction1}.
\begin{figure}[ht]
\includegraphics[width=0.95\textwidth]{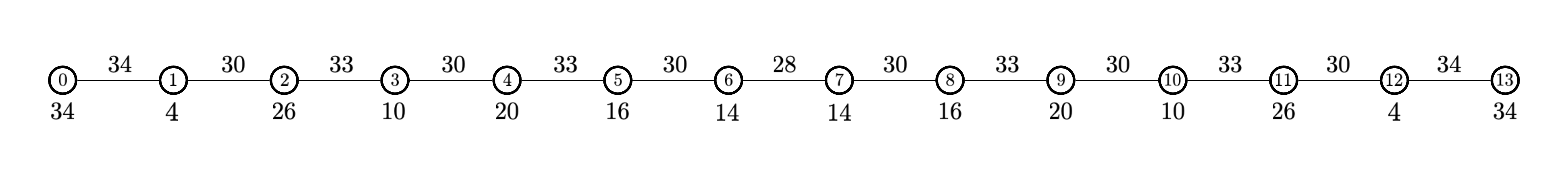}
\caption{\label{fig:exampleconstruction1} Weights $p$ when $f=6$, $b=34$, $s=30$, $\mathbf{d}=(4,3,3,0,1,3,4, 0, \ldots, 0)$.}
\end{figure}

We show below how this weighted path yields the ray in $\mathbb{R}^{2n+2}$ $$(34,34,64,67,94,100,124,130,154,161,184,194,214,228,244,258,274,288,\ldots, 30n, 30n+14).$$ Note that the ray settles into a $(30k, 30k+14)$-double-arithmetic progression starting at entries $2f$ and $2f+1$, i.e., at 214 and 228.
\end{example}

Since they are nonnegative, the weights $p$ on vertices and edges give us a way to create a blow-up graph $B_m$ for $P_{2f+1}$ for which the the number of homomorphisms from $P_i$ to the blow-up graph is $O(m^{b+\frac{i}{2} s})$ if $i$ is even and $O(m^{\lceil \frac{i}{2} \rceil s+\sum_{v=0}^{\lfloor\frac{i}{2} \rfloor} d_v})$ if $i$ is odd.

 \begin{theorem}\label{thm:construction}
As $m\rightarrow \infty$, the ray $(\log(\hom(P_0;B_m)), \log(\hom(P_1;B_m)), \ldots, \log(\hom (P_{2n+1};B_m)))$ converges to $\mathbf{r}=(r_0,\dots,r_{2n+1})$ where $r_{i}=b+\frac{i}{2}s$ if $i$ is even and $r_i=\frac{i+1}{2}s+\sum_{j=0}^{\frac{i-1}{2}} d_v$ if $i$ is odd for $0\leq i \leq 2n+1$. 
 \end{theorem}
 
 \begin{proof}
Let $P$ again denote the path of length $2f+1$ and fix some $i$ between $0$ and $2n+1$. Recall from the proof of Lemma \ref{lem:BmP} that we can think of $p(\varphi)$ for some homomorphism $\varphi: P_i \rightarrow P$ as \begin{align*}p(\{\varphi(0)\}) &+ (-p(\{\varphi(0)\})+p(\{\varphi(0),\varphi(1)\} )) + (-p(\{\varphi(1)\})+p(\{\varphi(1),\varphi(2)\})) \\
&+ \ldots + (-p(\{\varphi(i-1)\})+p(\{\varphi(i-1),\varphi(i)\})).\end{align*} We call $-p(\{\varphi(j)\})+p(\{\varphi(j), \varphi(j+1)\})$ the weight associated to the step going from $\varphi(j)$ to $\varphi(j+1)$. 

First observe that the weight of going from left to right and from right to left (not necessarily one right after the other) over any edge besides $\{f,f+1\}$ is $s$. Moreover, the weight of going from left to right and from right to left over $\{f,f+1\}$ is at most $s$. Indeed, if $f$ is even and $0<t<f-t$, we get $2d_0+4\sum_{v=1}^t d_v-2d_t+2d_{f-t}=2d_0 + 2\sum_{v=1}^t d_v + 2\sum_{v=f-t}^{f-1} d_v = 2\sum_{v=1}^f d_v \leq s$ by assumption. If $t=0$, we get $2d_f=2\sum_{v=1}^f d_f \leq s$ by assumption. If $f$ is even and $t=f-t$, we get $2d_0+4\sum_{v=1}^{\frac{f-2}{2}} d_v+d_{\frac{f}{2}}=2\sum_{v=1}^f d_v - d_{\frac{f}{2}} \leq s$ by assumption. Finally, if $f$ is odd, $2d_0+4\sum_{v=1}^{\frac{f-1}{2}}d_v - 2d_t+2d_{f-t}=2\sum_{v=1}^f d_v -2d_t+ 2d_{f-t} \leq s$ by assumption.

Now consider the edges of $P$ that are covered an odd number of times by $\varphi(P_i)$. Look at the subgraph formed by these edges, and observe that it must be connected and is thus a path, say $P_j$, with vertices $w, w+1, \ldots, w+j$ and edges $\{w, w+1\}, \ldots, \{w+j-1, w+j\}$. Note that $i$ and $j$ have the same parity since we are forgetting about edges that are covered an even number of times by $P_i$.  Without loss of generality because of symmetry, assume that $w\leq f$. Observe that $\varphi(0)\in \{w, w+j\}$, and we can also assume without loss of generality that $\varphi(0)=w$. Thus $$p(\varphi(P_i))\leq  \sum_{k=0}^{j-1} p(\{w+k,w+k+1\}) - \sum_{k=1}^{j-1} p(\{w+k\}) +\frac{i-j}{2}\cdot s.$$ 

When $j$ is odd, $P_j$ has an even number of vertices, and in particular, an even number of internal vertices. We partition the internal vertices in groups of two consecutive vertices with the edge in between. That is, we observe that

\begin{align*}
\sum_{k=0}^{j-1} p(\{w+k,w+k+1\}) &- \sum_{k=1}^{j-1} p(\{w+k\})\\
= & \sum_{k\in \{1,3,\ldots, j-2\}} (-p(\{w+k\})-p(\{w+k+1\})+p(\{w+k,w+k+1\}))\\
& + \sum_{k\in \{0, 2, \ldots, j-1\}} p(\{w+k, w+k+1\}).
\end{align*}

Let's first consider the case when $w$ is even. Then $-p(\{w+k\})-p(\{w+k+1\})+p(\{w+k,w+k+1\})=0$ for all $k\in \{1,3,\ldots, j-2\}$ except if $f$ is odd, where $-p(\{f\})-p(\{f+1\})+p(\{f,f+1\})$ contributes $-d_t+d_{f-t}$. Moreover, $p(\{w+k,w+k+1\})=s+d_{\frac{w+k}{2}}$ for all $k\in \{0, 2, \ldots, j-1\}$ except when $w+k=2(f-t)$, when the contribution is $s+d_t$, and when $f$ is even, when $-p(\{f\})-p(\{f+1\})+p(\{f,f+1\})$ contributes $s-d_t+d_{f-t}+d_\frac{f}{2}$. Note that $d_{\frac{f}{2}}=0$ if $t<f-t$. Therefore, $P_j$ contributes $\frac{j+1}{2}s + \sum_{v=\frac{w}{2}}^{\frac{w+j-1}{2}}d_v$ if $P_j$ ends either at $f$ or before or at $2(f-t)+1$ or after. Otherwise, $P_j$ contributes $\frac{j+1}{2}s + \sum_{v=\frac{w}{2}}^{\frac{w+j-1}{2}}d_v - d_t+d_{f-t} + d_{\frac{f}{2}}$. 

Let's now consider the case when $w$ is odd. Then $-p(\{w+k\})-p(\{w+k+1\})+p(\{w+k,w+k+1\})=-d_u$ for all $k\in \{1,3,\ldots, j-2\}$ except when $w+k=2(f-t)$, when the contribution is $-d_t$, and when $f$ is even and the contribution of $-p(\{f\})-p(\{f+1\})+p(\{f,f+1\})$ is $-s+2\sum_{v=1}^f d_v - d_{\frac{f}{2}} + d_t - d_{f-t}\leq d_t-d_{f-t}-d_{\frac{f}{2}}$. Furthermore, $p(\{w+k, w+k+1\})=s$ for all  $k \in \{0, 2, \ldots, j-1\}$ except when $f$ is odd and the contribution of $p(\{f, f+1\})$ is $2\sum_{v=1}^f d_v - d_{f-t}+d_t \leq s- d_{f-t}+d_t$. So it is clear that the contribution of $P_j$ is always bigger when $w$ is even.

Thus, when $i$ and $j$ are odd, the maximum $p(\varphi(P_i))$ can be is at most $\frac{i+1}{2}s + \sum_{v=0}^{\frac{i-1}{2}} d_v$ because of property (3) of $\mathbf{d}$ in Definition \ref{def:blowupgraph}. Note that this is tight. Indeed, if $i\leq 2t+1$, we can take $\varphi: P_i \rightarrow P$ such that $\varphi(l)=l$ for all $0\leq l \leq i$. If $2t+3 \leq i \leq 2(f-t)-1$, let $\varphi$ be such that $\varphi(l)=l$ for all $0\leq l \leq 2t+1$, $\varphi(l)=2t$ for all $2t+2\leq l \leq i$ even, and $\varphi(l)=2t+1$ for all $2t+1\leq l \leq i$ odd. If $2(f-t)+1 \leq l \leq 2f+1$, let $\varphi(l)=l$ for all $0 \leq l \leq i$. Finally, if $i\geq 2f+3$, let $\varphi$ be such that $\varphi(l)=l$ for all $0\leq l \leq 2f+1$, $\varphi(l)=2f$ for all $2f+2\leq l \leq i$ even, and $\varphi(l)=2f+1$ for all $2f+1\leq l \leq i$ odd. 

We now consider the case when $j$ and $i$ are even (where $P_j$ is still the subgraph of $P_i$ whose edges are covered an odd number of times). Here, instead of partitioning the internal vertices in groups of consecutive pairs, we partition the edges in a similar way. Consider the weight associated to going forward by two steps from $j$ to $j+1$ to $j+2$ that do not go over the edge $\{f, f+1\}$:

\begin{displaymath}
\begin{array}{|l|c|}
\hline
 & -p(\{j\})+p(\{j,j+1\})\\
& -p(\{j+1\})+p(\{j+1,j+2\})\\
\hline
j=2u \textup{ and } 2u, 2u+1, 2u+2\leq f & s-d_u\\
j=2u-1 \textup{ and } 2u-1, 2u, 2u+1\leq f & s+d_u\\
j=2u \textup{ and } 2u, 2u+1, 2u+2 \geq f & s-d_{f-u}\\
j=2u-1 \textup { and } 2u-1, 2u, 2u+1 \geq f & s+d_{f-u}\\
\hline
\end{array}
\end{displaymath}

Similarly, look at the weight associated to going forward by two steps from $j$ to $j+1$ to $j+2$ that does go over the edge $\{f, f+1\}$:

\begin{displaymath}
\begin{array}{|l|c|}
\hline
 & -p(\{j\})+p(\{j,j+1\})\\
& -p(\{j+1\})+p(\{j+1,j+2\})\\
\hline
j=f \textup{ and } f \textup{ is even}  & 2d_0+4\sum_{v=1}^{\frac{f-2}{2}} d_v +d_{\frac{f}{2}} -d_t +d_{f-t}\\
j=f-1 \textup{ and } f \textup{ is even} & s+d_{\frac{f}{2}}-d_t+d_{f-t}\\
j=f \textup{ and } f \textup{ is odd} & s+d_{\frac{f-1}{2}} - d_t+d_{f-t}\\
j=f-1 \textup{ and } f \textup{ is odd} & 2d_0+4\sum_{v=1}^{\frac{f-3}{2}} d_v + 3d_{\frac{f-1}{2}} - d_t + d_{f-t}\\
\hline
\end{array}
\end{displaymath}

Note that $2d_0+4\sum_{v=1}^{\frac{f-2}{2}} d_v +d_{\frac{f}{2}} -d_t +d_{f-t}=2\sum_{v=1}^f d_v - d_{\frac{f}{2}}+ d_t -d_{f-t}\leq s - d_{\frac{f}{2}}+d_t-d_{f-t}$. 

We will show that the homomorphism $\varphi^*$ such that $\varphi^*(k)=0$ if $k$ is even, and $\varphi^*(k)=1$ if $k$ is odd, is maximal. Note that $p(\varphi^*(P_i))=b+\frac{i}{2}s$. 

If $w$ is even, note that any two steps starting on an even vertex adds at most $s$ if none of these steps go over $\{f, f+1\}$. In such cases $p(\varphi(P_i))\leq \varphi(0) + \frac{j}{2}\cdot s +\frac{i-j}{2}\cdot s=p(\{\varphi(0)\}) + \frac{i}{2} \cdot s \leq b + \frac{i}{s}$ since $p(\{0\})\geq p(\{2u\})$ for all $u$. If we go over $\{f, f+1\}$ with $f$ even, the contribution of that double step, $2\sum_{v=1}^f d_v - d_{\frac{f}{2}}+ d_t -d_{f-t}$ can be more than $s$, so we need to be a bit more careful. Adding up, we get

\begin{align*}
s-d_0&-2\sum_{v=1}^{u-1} d_v + \frac{j-2}{2} s - \sum_{v=u}^{\frac{f-2}{2}}d_v - \sum_{v=\frac{f+2}{2}}^{\frac{2u+j-2}{2}} d_{f-v} + 2\sum_{v=1}^f d_v - d_{\frac{f}{2}} + d_t - d_{f-t} + \frac{i-j}{2}s\\
&\leq \frac{i}{2}s + 2\sum_{v=1}^f d_v + d_t -d_0\\
&\leq \frac{i}{2}s + b.
\end{align*}
where the second line follows from the fact that all $d_v$'s are nonnegative, and the third line because $d_t\leq d_0$, $2\sum_{v=1}^f d_v \leq s$, and $b=d_0+s$. Similarly, if we go over $\{f, f+1\}$ with $f$ odd, the contribution of that particular double step can be more than $s$, so we need to look at the total contribution a bit more carefully. Again, adding up, we get

\begin{align*}
\frac{j-2}{2}s &- \sum_{v=u}^{\frac{f-3}{2}} d_v - \sum_{v=\frac{f+1}{2}}^{\frac{2u+j-2}{2}} d_{f-v} + 2d_0 + 4\sum_{v=1}^{\frac{f-3}{2}} d_v + 3d_{\frac{f-1}{2}} - d_t + d_{f-t} + s-d_0 - 2\sum_{v=1}^{u-1} d_v + \frac{i-j}{2}s\\
&= \frac{i}{2}s + \underbrace{2\sum_{v=1}^f d_v}_{\leq s} \underbrace{-d_0 + d_t}_{\leq 0} \underbrace{-\sum_{v=1}^{\frac{f-3}{2}} d_v - \sum_{v=\frac{f+1}{2}}^{\frac{2u+j+2}{2}} d_v -\sum_{v=1}^{u-1} d_v - 2d_{\frac{f+1}{2}} - d_{f-t}}_{\leq 0} \underbrace{+ d_{\frac{f-1}{2}}}_{\leq d_0}\\
&\leq \frac{i}{2}s + s + d_0 \\
& \leq \frac{i}{2}s + b
\end{align*}

If $w$ is odd, then $p(\{\varphi(0)\})= d_0 + 2\sum_{v=1}^{\frac{w-1}{2}} d_v\leq \frac{s}{2}$. If $P_j$ does not contain $\{f, f+1\}$, $\sum_{k=0}^{j-1} p(\{w+k,w+k+1\}) - \sum_{k=1}^{j-1} p(\{w+k\})= \frac{j}{2}\cdot s + \sum_{v=\frac{w+1}{2}}^{\frac{w-1+2j}{2}} d_v$. So \begin{align*}
p(\varphi(P_i)) & \leq d_0 + 2\sum_{v=1}^{\frac{w-1+2j}{2}} d_v + \frac{j}{2} \cdot s+\frac{i-j}{2}\cdot s \\
& \leq d_0 + s + \frac{i}{2}\cdot s \leq b + \frac{i}{2}\cdot s.
\end{align*} 

If $P_j$ contains $\{f, f+1\}$, note that the two steps containing $\{f, f+1\}$ that start on an odd vertex all contribute at most $s$ except if $f$ is even and $t=\frac{f}{2}$ or if $f$ is odd and $t=\frac{f-1}{2}$. In both of those cases, we add $s+d_{\lfloor \frac{f}{2}\rfloor}$. Note that $d_{\lfloor \frac{f}{2}\rfloor}$ does not appear in the weight of any other double step. Therefore, we have that $p(\varphi(P_i)) \leq d_0 + 2\sum_{v=1}^{f} d_v + \frac{j}{2} \cdot s+\frac{i-j}{2}\cdot s \leq d_0+ s + \frac{i}{2}\cdot s \leq b + \frac{i}{2}s$.

So in both cases, we have that $p(\varphi(P_i)) \leq b + \frac{i}{2}s$ for any homomorphism $\varphi:P_i\rightarrow P$ when $i$ is even, and we've already shown that this maximum is attainable through $\varphi^*$.
 \end{proof}

\begin{definition}\label{def:raysproperties}
Consider the family of rays $\bigcup_{0\leq s \leq b} \mathcal{R}_{s,b}$ where $(r_0, r_1, \ldots, r_{2n+1})=:\mathbf{r}\in \mathcal{R}_{s,b}$ if
\begin{enumerate}
\item $r_{2i}=is+b$ for all $0\leq i \leq n$ where $b\geq s \geq 0$, and 
\item $r_{2i+1}=(i+1)s + \sum_{v=0}^i d_v$ for all $0 \leq i \leq n$ where $\mathbf{d}:=(d_0, d_1, \ldots, d_f, 0, \ldots, 0)\in \mathbb{R}^{n+1}$ is such that
\begin{enumerate}
\item $d_v \geq 0$ for all $0\leq v \leq n$,
\item $d_0=b-s$,
\item $d_0 + \ldots + d_u \geq d_{v-u} + \ldots + d_v$ for any $0\leq u< v \leq n$,
\item $d_u=d_{f-u}$ for all $0 \leq u<t$,
\item $d_u=0$ for all $t+1\leq u \leq f-t-1$,
\item $d_t\geq d_{f-t} \geq 0$,
\item $2\sum_{v=1}^f d_v \leq s$,
\end{enumerate}
where $t$ is the largest index such that $t\leq \frac{f}{2}$ and $d_t>0$. Note that here we do not require that $s \geq d_0$ if $t=0$ as before.
\end{enumerate}
\end{definition}

\begin{corollary}
The rays $ \mathcal{R}_{s,b}$ with $0\leq s \leq b$ are contained in $\troptn$, i.e.,
$$\bigcup_{0\leq s \leq b} \mathcal{R}_{s,b} \subseteq \troptn.$$ 
\end{corollary} 

\begin{proof}
If $t>0$ or $t=0$ and $s \geq d_0$, this follows directly from the blow-up construction in Theorem \ref{thm:construction}. The only case that needs more thought is if $t=0$ and $s< d_0$. So consider a ray $\mathbf{r}\in \mathcal{R}_{s,b}$ associated to $\mathbf{d}=(d_0, 0, \ldots, 0, d_f)$ where $2d_f \leq s < d_0$.

First note that $(1,1, \ldots, 1)\in \troptn$ since $\hom(P_i;P_1)=2$ for all $i \geq 0$. Let $\mathbf{r}'\in \mathcal{R}_{s,b-d_0+d_f}$ be associated with $\mathbf{d}'=(d_f, 0, \ldots, 0, d_f)$. Note that $b-d_0+d_f \geq s$ since $d_f \geq 0$, and $d'_0=d_f\leq s$, so we know that $\mathbf{r}'\in \troptn$. Now observe that $\mathbf{r'}+(d_0-d_f)\cdot (1,1,\ldots, 1) = \mathbf{r}$. Since $\mathbf{r}$ is a conic combination of rays in $\troptn$, $\mathbf{r}$ is also in $\troptn$.
\end{proof}

We now build a bigger family of rays $\mathcal{\bar{R}}_{s,b}$ such that $\bigcup_{0\leq s \leq b} \mathcal{R}_{s,b} \subset \bigcup_{0\leq s \leq b} \mathcal{\bar{R}}_{s,b} \subseteq \trop (\mathcal{N}_{\mathcal{U}_{2n+1}})$ where we remove the ``almost symmetric'' requirement on $\mathbf{d}$, i.e., where we remove the previous properties (d), (e) and (f) of Definition \ref{def:raysproperties}.

 \begin{definition}\label{def:1?}

Consider the family of rays $\bigcup_{0\leq s \leq b} \mathcal{\bar{R}}_{s,b}$ where $(r_0, r_1, \ldots, r_{2n+1} )=:\mathbf{r}\in \mathcal{\bar{R}}_{s,b}$ if

\begin{enumerate}
\item $r_{2i}=is+b$ for all $0\leq i\leq n$ where $b\geq s \geq 0$, and 
\item $r_{2i+1}=(i+1)s + \sum_{v=0}^i d_v$ for all $0 \leq i \leq n$ where $\mathbf{d}:=(d_0, d_1, \ldots, d_f, 0, \ldots, 0) \in \mathbb{R}^{n+1}$ is such that
\begin{enumerate}
\item $d_v \geq 0$ for all $0 \leq v \leq n$,
\item $d_0=b-s$,
\item $d_0 + \ldots + d_u \geq d_{v-u} + \ldots + d_v$ for any $0\leq u< v \leq n$,
\item $2\sum_{v=1}^f d_v \leq s$,
\end{enumerate}
\end{enumerate}
\end{definition}

We show that these new rays are in the max closure of the rays in $\mathcal{R}_{s,b}$. We first give an example to give some insight of why this is true.

\begin{example}
Suppose $2n+1=13$ and consider the ray $\mathbf{r} \in \mathcal{\bar{R}}_{30,34}\backslash \mathcal{R}_{30,34}$ associated to $\mathbf{d}=(4,3,2,1,2,4,0)$, i.e., $\mathbf{r}=(34,34,64,67,94,99, 124, 130, 154, 162, 184, 196, 214, 226)$. Now look at the rays $\mathbf{r}_0, \ldots, \mathbf{r}_5\in \mathcal{R}_{30,34}$ associated to $\mathbf{d}_0=(4,0,0,0,0,0,0)$, $\mathbf{d}_1=(4,3,0,0,0,0,0)$, $\mathbf{d}_2=(4,1,4,0,0,0,0)$, $\mathbf{d}_3=(4, 1, 1, 4,0,0,0)$, $\mathbf{d}_4=(4, 3,0,1,4,0,0)$, $\mathbf{d}_5=(4,3,2,0,3,4,0)$, namely \begin{align*}
&(34, \underline{34}, 64, 64, 94, 94, 124, 124, 154, 154, 184, 184, 214, 214)\\
&(34, 34, 64, \underline{67}, 94, 97, 124, 127, 154, 157, 184, 187, 214, 217)\\
&(34, 34, 64, 65, 94, \underline{99}, 124, 129, 154, 159, 184, 189, 214, 219)\\
&(34, 34, 64, 65, 94, 96, 124, \underline{130}, 154, 160, 184, 190, 214, 220)\\
&(34, 34, 64, 67, 94, 97, 124, 128, 154, \underline{162}, 184, 192, 214, 222)\\
&(34, 34, 64, 67, 94, 99, 124, 129, 154, 162, 184, \underline{196}, 214, \underline{226}). \end{align*} Note that the sum of the first $j<i$ entries in $\mathbf{d}_i$ is at most the sum of the first $j$ entries in $\mathbf{d}$, and that the sum of all entries in $\mathbf{d}_i$ is equal to the sum of the first $i$ entries of $\mathbf{d}$. Thus $r_{2i+1}$ is equal to the $(2i+1)$st component of $\mathbf{r}_i$, and it is greater or equal to the $i$th component of $\mathbf{r}_j$ for any $j$. Moreover, the even entries of $\mathbf{r}$ and $\mathbf{r}_i$ are the same for all $i$. Thus $\mathbf{r}$ is in the max closure of the $\mathbf{r}_i$'s.
\end{example}

More generally, given a ray $\mathbf{r}\in \mathcal{\bar{R}}_{s,b}$ associated to some vector $\mathbf{d}$, we first show that turning the last non-zero coordinate of $\mathbf{d}$ into zero yields a ray that is also in $\mathcal{\bar{R}}_{s,b}$. That will allow us to proceed recursively. We then only need to show that there exists a ray $\mathbf{r}'\in \mathcal{R}_{s,b}$ where $r'_i \leq r_i$ for all $0\leq i\leq 2f-1$, and $r'_i=r_i$ for all $2f\leq i \leq 2n+1$.

 \begin{lemma}\label{lem:subsequence}
Consider $\mathbf{r}\in \mathcal{\bar{R}}_{s,b}$ associated with some sequence $\mathbf{d}=(d_0, d_1, \ldots, d_f, 0, \ldots, 0)\in \mathbb{R}^{n+1}$. Then there is a ray $\mathbf{r}'\in \mathcal{\bar{R}}_{s,b}$ associated to $\mathbf{d}'$ where $d'_v = d_v$ for all $v\in \{0, 1, 2, \ldots, n\} \backslash \{f\}$ and $d'_f=0$. In particular, $r'_v=r_v$ for all $v\leq 2f$, and $r'_v \leq r_v$ for all $2f+1 \leq v \leq 2n+1$.
 \end{lemma}
 
 \begin{proof}
 Certainly, properties (a), (b), and (d) in Definition \ref{def:1?} still hold for $\mathbf{d}'$. We only need to show that $\sum_{w=0}^u d'_w \geq \sum_{w=v-u}^{v} d'_w$ for all $0 \leq u <v \leq n$. Without loss of generality, we can assume that $u<v-u$. Certainly, we know that $\sum_{w=0}^u d_w \geq \sum_{w=v-u}^{v} d_w$ since $\mathbf{r}\in \mathcal{\bar{R}}_{s,b}$. If $d_f$ is not involved at all in either sum, then those two expressions remain unchanged, and we are done. If it is involved in the righthand side, then $$\sum_{w=0}^u d'_w=\sum_{w=0}^u d_w \geq \sum_{w=v-u}^{v} d_w \geq \sum_{w=v-u}^{v} d_w - d_f = \sum_{w=v-u}^{v} d'_w.$$
 
Finally, if $d_f$ is involved in the lefthand side, then the righthand side is zero since $u<v-u$ and $d_f$ was the last non-zero entry of $\mathbf{d}$, and since $\sum_{w=0}^u d'_w\geq 0$, we are done. 
 \end{proof}

  \begin{lemma}\label{lem:domination}
For any ray $\mathbf{r}\in \mathcal{\bar{R}}_{s,b}$ associated to some sequence $\mathbf{d}=(d_0, d_1, \ldots, d_f, 0, \ldots, 0) \in \mathbb{R}^{n+1}$, there exists a ray $\mathbf{r}'\in \mathcal{R}_{s,b}$ such that $r_i \geq r'_i$ for all $0 \leq i \leq 2f-1$ and $r_i = r'_i$ for all $2f \leq i \leq 2n+1$.
 \end{lemma}

\begin{proof}
Let $T:=\sum_{w=0}^f d_w$ and let $t$ be the smallest index such that $\sum_{w=0}^t d_w \geq \frac{T}{2}$. Note that $t\leq \frac{f}{2}$ since we know that $\sum_{w=0}^{\lfloor\frac{f}{2}\rfloor} d_w \geq \sum_{w=\lceil \frac{f}{2} \rfloor}^{f} d_w$. 

Let $\mathbf{d'}\in \mathbb{R}^{n+1}$ be such that $d'_w=d_w$ for $0 \leq w\leq t$, $d'_w=d_{f-w}$ for $f-t+1 \leq w \leq f$, $d'_{f-t}=T-2\sum_{w=0}^{t-1} d_w - d_t$ (which is at most $d_t$ by the definition of $t$), and $d'_w=0$ otherwise. Note that properties (a), (b), (d), (e), (f), (g) from Definition \ref{def:raysproperties} hold for $\mathbf{d}'$.

Clearly, $\sum_{w=0}^f d_w=\sum_{w=0}^f d'_w=T$. We now show that $\sum_{w=0}^u d_w \geq \sum_{w=0}^u d'_w$. If $u\leq f-t-1$, then this follows from the the definition of $\mathbf{d}'$. Let's now consider the case when $u\geq f-t$. Certainly we know that $\sum_{w=0}^{f-u-1} d_w \geq \sum_{w=u+1}^f d_w$. Moreover, we have that $\sum_{w=u+1}^f d'_w =\sum_{w=0}^{f-u-1} d_w$, so we have

\begin{align*}
\sum_{w=u+1}^f d'_w &\geq \sum_{w=u+1}^f d_w\\
T-\sum_{w=u+1}^f d_w &\geq T-\sum_{w=u+1}^f d'_w\\
\sum_{w=0}^f d_w-\sum_{w=u+1}^f d_w&\geq \sum_{w=0}^f d'_w-\sum_{w=u+1}^f d'_w\\
\sum_{w=0}^u d_w &\geq \sum_{w=0}^u d'_w
\end{align*} as desired.

We now show that property (c) from Definition \ref{def:raysproperties} also holds for $\mathbf{d}'$, that is, that $\sum_{w=0}^u d'_w \geq \sum_{w=v-u}^v d'_w$ for all $0 \leq u < v \leq 2n+1$. Without loss of generality, assume that $u< v-u$. Moreover, observe that if we do not have that $u-v \leq t < f-t \leq v$, then the righthand side will correspond to adding up a subsequence of vector $\mathbf{d}$, and so the inequality must hold since the number of non-zero $d_w$'s in the lefthand side is greater or equal to the number of non-zero $d_w$'s in the righthand side.

So assume that $u-v \leq t < f-t \leq v$. This implies that $\sum_{w=0}^u d'_w=\sum_{w=0}^u d_w$ since $u<t$. Moreover,
\begin{align*}
\sum_{w=v-u}^v d'_w&=\sum_{w=0}^v d'_w - \sum_{w=0}^{v-u-1} d'_w\\
& \leq \sum_{w=0}^v d_w - \sum_{w=0}^{v-u-1} d_w\\
&= \sum_{w=v-u}^{v} d_w\\
&\leq \sum_{w=0}^{u} d_w\\
& = \sum_{w=0}^u d'_w
\end{align*}
where the second line follows from the fact that we have already shown that $\sum_{w=0}^v d'_w \leq \sum_{w=0}^v d_w$ and $\sum_{w=0}^{v-u-1} d_w=\sum_{w=0}^{v-u-1} d'_w$ since $v-u-1<t$.

So the ray $\mathbf{r}'$ associated to $\mathbf{d}'$ is in $\mathcal{R}_{s,b}$, and $r_i \geq r'_i$ for all $0 \leq i \leq 2f-1$ and $r_i = r'_i$ for all $2f \leq i \leq 2n+1$.
\end{proof}

 \begin{theorem}
The rays $\mathcal{\bar{R}}_{s,b}$ with $0\leq s \leq b$ are contained in $\troptn$:  $$\bigcup_{0\leq s \leq b} \mathcal{\bar{R}}_{s,b} \subseteq \troptn.$$
  \end{theorem}
 
 \begin{proof}
Consider any ray $\mathbf{r}\in \mathcal{\bar{R}}_{s,b}$. From Lemma \ref{lem:domination}, there exists a ray $\mathbf{r}'\in \mathcal{R}_{s,b}$ such that $r_v \geq r'_v$ for all $0\leq v \leq 2f-1$ and $r_v=r'_v$ for all $2f \leq v \leq 2n+1$. From Lemma \ref{lem:subsequence}, there is a ray $\mathbf{r}^*\in \mathcal{\bar{R}}_{s,b}$ with sequence $\mathbf{d}^*=(d_0, d_1, \ldots, d_{f-1}, 0, 0, 0, \ldots)$ such that $r^*_v=r_v$ for all $0\leq v\leq 2f$ and $r^*_v\leq r_v$ for all $2f+1 \leq v \leq 2n+1$. Now again from Lemma \ref{lem:domination}, we can find a ray $\mathbf{r}''\in \mathcal{R}_{s,b}$ such that $r^*_v \geq r''_v$ for all $0 \leq v \leq 2f-3$ and $r^*_v=r''_v$ for all $2f-2\leq v \leq 2n+1$. Note that this implies that $r''_v=r_v$ for $v\in \{2f-2, 2f-1\}$. Keeping doing this recursively, we obtain a set of vectors in $\mathcal{R}_{s,b}$ whose max closure is exactly $\mathbf{r}$. Since we know $\troptn$ is closed under max closure, and that $\bigcup_{0\leq s \leq b} \mathcal{R}_{s,b} \subseteq \troptn$, we also have that $\bigcup_{0\leq s \leq b} \mathcal{\bar{R}}_{s,b} \subseteq \troptn$.
  \end{proof}

 We now build an even bigger family of rays in $\troptn$ by relaxing the criterion that $b-s=d_0$.

 \begin{definition}\label{def:?}

Consider the family of rays $\bigcup \mathcal{R}^*_{s,b}$ where the union is over pairs $s, b\geq 0$ such that $b-s\leq d_0 \leq \frac{2b-s}{2}$ where $(r_0, r_1, \ldots, r_{2n+1} )=:\mathbf{r}\in \mathcal{R}^*_{s,b}$ if

\begin{enumerate}
\item $r_{2i}=is+b$ for all $0 \leq i \leq n$ for some $b,s\geq 0$, and
\item $r_{2i+1}=(i+1)s + \sum_{v=0}^i d_v$ for all $0 \leq i \leq n$ where $\mathbf{d}:=(d_0, d_1, \ldots, d_f, 0, \ldots, 0)\in \mathbb{R}^{n+1}$ is such that
\begin{enumerate}
\item $d_v\geq 0$ for all $v\geq 0$
\item $d_0 + \ldots + d_u \geq d_{v-u} + \ldots + d_v$ for any $0\leq u\leq v \leq n$
\item $2\sum_{v=1}^f d_v\leq s$
\end{enumerate}
\end{enumerate}
\end{definition}

 Let $$\mathcal{R}:=(1,0,0, \ldots, 0) \cup \bigcup_{\substack{s,b\geq 0:\\ b-s\leq d_0 \leq \frac{2b-s}{2}}} \mathcal{R}^*_{s,b}.$$

\begin{theorem}
The family of rays $\mathcal{R}$ is contained in $\troptn$: $$\mathcal{R} \subseteq \troptn.$$
  \end{theorem}

 \begin{proof}
First note that $(1,0,0,\ldots, 0)$ is in $\troptn$ since it can be realized by a graph on $m$ vertices with exactly one edge as $m\rightarrow \infty$. Now consider a ray $\mathbf{r}\in \mathcal{R}^*_{s,b}$ for some $b,s\geq 0$ associated to $\mathbf{d}$ where $b-s\leq d_0 \leq \frac{2b-s}{2}$. Then $\mathbf{r}$ can be written as the conic combination of $d_0+s-b$ times $(1,2,3,4,\ldots, 2n+2)$ and one times the ray $\mathbf{r}'$, where $\mathbf{r}'\in \mathcal{\bar{R}}_{2b-s-2d_0, 2b-s-d_0}$ with $\mathbf{d}':=\mathbf{d}$. Note that $d'_0=d_0=2b-s-d_0-(2b-s-2d_0)$ as desired. Thus, any ray $\mathbf{r}\in \mathcal{R}^*_{s,b}$ is in the double hull of $\bigcup_{0\leq k \leq b} \mathcal{\bar{R}}_{s,b}$, and so $\mathcal{R}\subseteq \troptn$.

\end{proof}

\subsection{Double hull of $\mathcal{R}$ is $\projC$}\label{subsec:doublehullofconstruction}

For any ray spanned by $\mathbf{r}=(r_0, r_1, \ldots, r_{2n+1}) \in \projC$, we construct rays $\mathbf{r}'_l\in \mathcal{R}$ for all $0 \leq l \leq 2n+1$ such that every coordinate of $\mathbf{r}'_l$ is at most the corresponding coordinate of $\mathbf{r}$ and the $l$-th coordinate of $\mathbf{r}'_l$ is equal to $r_l$.  Therefore we have $\mathbf{r}=\oplus_{l=1}^{2n+1} \mathbf{r}'_l$ and $\projC \subseteq \trop (\mathcal{N}_{\mathcal{U}_{2n+1}})$.
We split the construction into in two parts: first when $l$ is odd, and then when $l$ is even.

\begin{definition}[$\mathbf{r}'_l$ when $l$ is odd]\label{def:rlodd}
Given $\mathbf{r}=(r_0, r_1, \ldots, r_{2n+1}) \in \projC$ and $0 \leq l \leq 2n+1$ such that $l=2i+1$, let $\mathbf{r}'_{l}=(r'_{0}, r'_{1}, \ldots, r'_{2n+1})$ be as follows.

Let $$s'=\max_{0\leq j\leq i} \left\{ \frac{2(r_{2i+1}-r_{2j})}{2i+1-2j} \right\}$$ and  $$j' = \argmax_{0\leq j\leq i}\left\{ \frac{2(r_{2i+1}-r_{2j})}{2i+1-2j} \right\}.$$ We now give some geometric intuition behind the above definitions. We can think of $\mathbf{r}=(r_0,\dots,r_{2n+1}) \in \projC$ as a function from $\{0,1,\dots, 2n+1\}$ to $\mathbb{R}$ sending $i$ to $r_i$. Consider the graph of this function in $\mathbb{R}^2$.
If we look at the slopes of the lines between pairs of points in the graph, we see that $s'$ is the maximal slope involving the point $(2i+1,r_{2i+1})$ and a point $(2j,r_{2j})$ with $2j<2i+1$, and $j'$ is the index where the maximal slope occurs. We are going to use the slope $s'$ when building our new sequence $\mathbf{r}'_{l}=(r'_0, r'_1, \ldots, r'_{2n+1})$.

Let $\mathbf{d}'=(d'_0, d'_1, \ldots, d'_n)\in \mathbb{R}^{n+1}$ where $$d'_0= \min_{0\leq u \leq i} \{r_{2u+1}-(u+1)s'\}$$ 
and $$d'_v = \min_{v\leq u \leq i} \{r_{2u+1}-(u+1)s'\}-\sum_{w=0}^{v-1}d'_w$$ for all $0<v\leq i$. Let $d'_v=0$ for all $i<v\leq n$. Here the entries $d'_v$ for $v\geq 1$ are giving us the difference of the deviation of one odd entry from the arithmetic sequence defined by even entries compared to the previous odd entry. 

Let $r'_{2v} = r_{2j'} + (v-j')s'$ for all $0\leq v\leq n$ so that all even points are on the same line with slope $s'$. Note that $s'$ is the slope between even entries, so the step size is $2$, i.e., $s'$ is double the slope we would use for all entries. Let $b':=r'_0$ and $f':=i$. Let $$r'_{2v+1}=(v+1)s'+\sum_{w=0}^{v} d'_w$$ for any $0\leq v \leq n$.
\end{definition}

\begin{example}
Let $2n+1=21$, $l=13$, and $$\mathbf{r}=(14,14,24,24,34,35,44,45,55,55,66,66,77,79,88,89,99,99,110,110,121,121).$$

We have that $s'=\max\{10, 10, 10, 10, 9.6, 8.\overline{6}, 4\}=10$ and $j'$ can be chosen to be either 0, 1, 2 or 3. Let's now compute a vector recording $r_{2u+1}-(u+1)s'$ for each $0 \leq u\leq i$: $(4, 4, 5, 5, 5, 6, 9)$. So $\mathbf{d}'=(4,0,1,0,0,1,3,0,0,0,0)$. Finally, this allows us to compute $$\mathbf{r}'_{13}=(14, 14, 24, 24, 34, 35, 44, 45, 54, 55, 64, 66, 74, 79, 84, 89, 94, 99, 104, 109, 114, 119).$$

Observe that $\mathbf{r}'_{13}$ is such that $r'_{13}=r_{13}$ and $r'_v\leq r_v$ for all $0 \leq v \leq 21$, and such that the even entries of $\mathbf{r}'$ form an arithmetic sequence. 
\end{example}

\begin{example}
Let $2n+1=17$, $l=15$, and $$\mathbf{r}=(16, 16, 24, 25, 32, 32, 40, 40, 49, 49, 58, 58, 67, 67, 76, 79, 85, 88).$$

We have $s'=\max\{8.4, 8.\overline{461538}, 8.\overline{54}, 8.\overline{6},8.\overline{571428}, 8.4, 8, 6\}=8.\overline{6}$, and $j'=3$. Computing the vector recording $r_{2u+1}-(u+1)s'$ for each $0\leq u \leq i$, we get $(7.\overline{3}, 7.\overline{6}, 6, 5.\overline{3}, 5.\overline{6}, 6, 6.\overline{3}, 9.\overline{6})$, thus yielding $\mathbf{d}'=(5.\overline{3}, 0, 0, 0, 0.\overline{3}, 0.\overline{3}, 0.\overline{3}, 3.\overline{3},0)$. We thus obtain $$\mathbf{r}'_{15}=(14, 14, 22.\overline{6}, 22.\overline{6}, 31.\overline{3}, 31.\overline{3}, 40, 40, 48.\overline{6}, 49, 57.\overline{3}, 58, 66, 67, 74.\overline{6}, 79, 83.\overline{3}, 87.\overline{6}).$$
Observe again that $\mathbf{r}'_{15}$ is such that $r'_{15}=r_{15}$ and $r'_v\leq r_v$ for all $0 \leq v \leq 17$, and such that the even entries of $\mathbf{r}'$ form an arithmetic sequence. 
\end{example}

We first show through the next three lemmas a few important properties of $\mathbf{r}'_l$:
\begin{itemize}
\item another way to compute $d'_v$ for some $0 <v \leq i$ is $\min_{v\leq u \leq i} \{r_{2u+1}-(u+1)s'\} - \min_{v-1\leq u \leq i} \{r_{2u+1}-(u+1)s'\}$,
\item if $d'_{v+1}>0$, then $r'_{2v+1}=r_{2v+1}$, i.e., showing that certain components of $\mathbf{r}$ and $\mathbf{r}'_l$ must be equal, and
\item $\mathbf{r}'_l$ and $\mathbf{r}$ are equal in the $l$th coordinate, and $\mathbf{r}_l$ is upper bounded by $\mathbf{r}$ in all coordinates. 
\end{itemize}

\begin{lemma}\label{lem:otherwaytocompdv}
We have that $d'_v = \min_{v\leq u \leq i} \{r_{2u+1}-(u+1)s'\} - \min_{v-1\leq u \leq i} \{r_{2u+1}-(u+1)s'\}$ for all $0 < v \leq i$. 
\end{lemma}

\begin{proof}
From Definition \ref{def:rlodd}, we know that $d'_v = \min_{v\leq u \leq i} \{r_{2u+1}-(u+1)s'\}-\sum_{w=0}^{v-1}d'_w$ for all $0<v\leq i$. This implies that $\sum_{w=0}^v d'_w = \min _{v\leq u \leq i} \{r_{2u+1}-(u+1)s'\}$ for all $0<v\leq i$ as well as when $v=0$ by definition of $d'_0$. We thus also have that $\sum_{w=0}^{v-1} d'_w = \min_{v-1 \leq u \leq i} \{r_{2u+1}-(u+1)s'\}$ for all $0<v\leq i$. So $d'_v = \sum_{w=0}^v d'_w - \sum_{w=0}^{v-1} d'_w = \min _{v\leq u \leq i} \{r_{2u+1}-(u+1)s'\} -  \min_{v-1 \leq u \leq i} \{r_{2u+1}-(u+1)s'\}$ for all $0<v\leq i$ as desired. 
\end{proof}

\begin{lemma}\label{lem:somecoordsequalinrlodd}
If $d'_{v+1}>0$, then $r'_{2v+1}=r_{2v+1}$. 
\end{lemma}

\begin{proof}
From Lemma \ref{lem:otherwaytocompdv}, we know that $d'_{v+1}=\min_{v+1\leq u \leq i} \{r_{2u+1}-(u+1)s'\} - \min_{v\leq u \leq i} \{r_{2u+1}-(u+1)s'\}.$ So if $d'_{v+1}>0$, this implies that $\min_{v+1\leq u \leq i} \{r_{2u+1}-(u+1)s'\}>  \min_{v\leq u \leq i} \{r_{2u+1}-(u+1)s'\}$, and so $\min_{v\leq u \leq i} \{r_{2u+1}-(u+1)s'\}= r_{2v+1}-(v+1)s'$. By the construction of $d'_v$ in Definition \ref{def:rlodd}, we thus know that $d'_v = r_{2v+1}-(v+1)s' - \sum_{w=0}^{v-1} d'_w$. Furthermore, by the construction of $r'_{2v+1}$ in Definition \ref{def:rlodd}, we have that
\begin{align*}
r'_{2v+1} & = (v+1)s'+\sum_{w=0}^v d'_w =(v+1)s' + \sum_{w=0}^{v-1} d'_w + d'_v\\
& = (v+1)s' + \sum_{w=0}^{v-1} d'_w +r_{2v+1}-(v+1)s' - \sum_{w=0}^{v-1} d'_w = r_{2v+1}
\end{align*}
as desired.
\end{proof}

\begin{lemma}\label{lem:rupperboundprimeodd}
Given $\mathbf{r}=(r_0, r_1, \ldots, r_{2n+1}) \in \projC$ and $\mathbf{r}'_{l}=(r'_{0}, r'_{1}, \ldots, r'_{2n+1})$ for some $l=2i+1$, we have $r'_v\leq r_v$ for all $0\leq v \leq 2n+1$ and $r_l=r'_l$.
\end{lemma}

\begin{proof}

\textbf{Claim 1:}  We have that $r'_{2v} \leq r_{2v}$ for all $0\leq v \leq j'$.

Since $\mathbf{r}\in \projC$, by inequalities (5.2) for $C$, we know that $r_{2u+4}-r_{2u+2} \geq r_{2u+2}-r_{2u}$ for any $0\leq u \leq n-2$. In particular, this means that $r_{2j'}-r_{2j'-2} \geq r_{2u}-r_{2u-2}$ for all $0\leq u \leq j'$. Now observe that by definition of $s'$ and $j'$, $s'=\frac{2(r_{2i+1} - r_{2j'})}{2i+1-2j'} \geq \frac{2(r_{2i+1}-r_{2j'-2})}{2i+1-(2j'-2)}$ which is equivalent of $(2i+1-(2j'-2))(r_{2i+1}- r_{2j'}) \geq (2i+1-2j')(r_{2i+1}-r_{2j'-2})$, which in turn is equivalent to $\frac{2(r_{2i+1}-r_{2j'})}{2i+1-2j'} \geq r_{2j'}-r_{2j'-2}$.  So we get that $s' \geq r_{2j'}-r_{2j'-2}$ which implies that $s' \geq r_{2u}-r_{2u-2}$ for all $0\leq u\leq j'$. Together with the fact that $r'_{2v}=r'_{2j'}-(j'-v)s'=r_{2j'}-(j'-v)s'$ by definition, this implies that $r'_{2v} \leq r_{2v}$ for all $0\leq v\leq j'$.\\

\textbf{Claim 2:} We have that $r'_{2v} \leq r_{2v}$ for all $j'\leq v \leq n$.

We first show that $s'\leq r_{2j'+2}-r_{2j'}$. Let's first assume that $i> j'$. Then $$s'=\frac{2(r_{2i+1}-r_{2j'})}{2i+1-2j'} \geq \frac{2(r_{2i+1}-r_{2j'+2})}{2i+1-(2j'+2)}$$ by the definition of $s'$ and $j'$, which is equivalent to $$\frac{2(r_{2i+1}-r_{2j'})}{2i+1-2j'} \leq r_{2j'+2} - r_{2j'},$$ as desired.  Since $\mathbf{r}\in \projC$, by inequalities (5.2) for $C$, we get that $s'\leq r_{2j'+2} - r_{2j'} \leq r_{2u+2}-r_{2u}$ for all $ j' \leq u \leq n$. Together with the fact that $r'_{2v}=r'_{2j'}+(v-j')s'=r_{2j}+(v-j')s'$ by definition, this implies that $r'_{2v} \leq r_{2v}$ for all $ v\geq j'$.

Finally we consider the case when $j'=i$. Then $s'=2(r_{2i+1}-r_{2i})$. Therefore, by inequalities (5.1) for $C$, we know that $s'\leq r_{2i+2}-r_{2i}$, and by the same argument as before, we have that $r'_{2v} \leq r_{2v}$ for all $ v\geq j'$. \\

\textbf{Claim 3:} We have that $r'_{2v+1} \leq r_{2v+1}$ for all $i\leq v\leq n$.

First note that, by Definition \ref{def:rlodd}, 

\begin{align*}
r'_{2v+1}&=(v+1)s'+\sum_{w=0}^v d'_w=(v+1)s' + \sum_{w=0}^i d'_w\\
& = (i+1)s' + \sum_{w=0}^i d'_w + (v-i)s'=r'_{2i+1}+(v-i)s'=r_{2i+1}+(v-i)s'.
\end{align*}

Thus showing that $r'_{2v+1} \leq r_{2v+1}$ is equivalent to showing that $$(2i+1-2j')r_{2v+1}-(2v+1-2j')r_{2i+1}+2(v-i)r_{2j'}\geq 0.$$ From our biggest generalization of the Erd\H{o}s-Simonovits inequalities in Lemma \ref{lem:biggestgeneralizationES}, we know that this inequality holds since $\mathbf{r}\in \projC$.

\smallskip

\textbf{Claim 4:} We have that $r'_{2v+1} \leq r_{2v+1}$ for all $0\leq v < i$.

This holds by the construction of $\mathbf{d}'$ and $r'_{2v+1}$ in Definition \ref{def:rlodd}. Indeed, when $v=0$, we have

$$r'_1=s'+\min_{0\leq u <i} \{r_{2u+1}-(u+1)s'\} \leq s' + r_1 -s' = r_1.$$

For any $1\leq v < i$, we have 
\begin{align*}
r'_{2v+1} &= (v+1)s' + \sum_{w=0}^v d'_w\\
& = (v+1)s' + \sum_{w=0}^{v-1} d'_w + \min_{v\leq u <i} \{r_{2u+1}-(u+1)s'\} - \sum_{w=0}^{v-1} d'_w\\
& \leq (v+1)s' + r_{2v+1} - (v+1)s' = r_{2v+1}.
\end{align*}

\textbf{Claim 5:} We have that $r'_{2i+1}=r_{2i+1}$. 

Note that by the construction of $d'_i$ in Definition \ref{def:rlodd}, we have that $d'_i = \min_{i \leq u \leq i}\{r_{2u+1} - (u+1)s'\} - \sum_{w=0}^{i-1} d'_w$, which implies that $\sum_{w=0}^i d'_w=r_{2i+1}-(i+1)s'.$ Thus, by the construction of $r'_{2v+1}$ in Definition \ref{def:rlodd}, we have that $r'_{2i+1}=(i+1)s'+r_{2i+1}-(i+1)s'=r_{2i+1}$.

\end{proof}

We now show that the points $\mathbf{r}'_l$ are in $\mathcal{R}$ when $l$ is odd, and thus in $\trop (\mathcal{N}_{\mathcal{U}_{2n+1}})$. 

\begin{theorem}
Let $\mathbf{r}\in \projC$ such that $r_1\geq r_0$ and fix $0\leq l \leq 2n+1$ such that $l=2i+1$. Then $\mathbf{r}'_l$ is in $\mathcal{R}$.
\end{theorem}

\begin{proof}
We show that all properties of rays in $\mathcal{R}$ hold for $\mathbf{r}'_l$.\\

\textbf{Claim 1:} We have that $s'\geq 0$.

By the construction of $s'$ in Definition \ref{def:rlodd}, $s'=\max_{0\leq j \leq i} \left\{ \frac{2(r_{2i+1}-r_{2j})}{2i+1-2j} \right\}$. The denominator is trivially nonnegative, and we claim the numerator is also nonnegative. Indeed, since $\mathbf{r}\in \projC$, we know that $r_{2j+1} \geq r_{2j}$ for all $1\leq j \leq n$ from inequalities (5.3) for $C$. Moreover, combining inequalities (5.1) and (5.3) for $C$ yields that $r_{2j+2}\geq r_{2j+1}$ for all $0\leq u \leq n-1$. Finally, we assumed that $r_1 \geq r_0$. Thus $r_{2i+1}-r_{2j}\geq 0$ for any $0 
\leq j \leq i$, and so $s'\geq 0$.\\

\textbf{Claim 2:} We have that $b'\geq 0$.

By the construction of $b'$ in Definition \ref{def:rlodd}, we need to show that $r_{2j'}-j's'\geq 0$. By the definition of $s'$, this is equivalent to showing that $(2i+1)r_{2j'}-2j'r_{2i+1} \geq 0$, which we know holds by Lemma \ref{lem:forbprimenonneg}.\\

\textbf{Claim 3:} We have that $d'_0 \geq b'-s'$.

By the constructions of $d'_0$ and $b'$ in Definition \ref{def:rlodd}, we want to show that $\min_{0\leq u \leq i}\{r_{2u+1}-(u+1)s'\} \geq r_{2j'}-j's' - s'$ which is equivalent to showing that $r_{2u+1} - (u+1)s' \geq r_{2j'} - (j'+1)s'$ for all $0 \leq u \leq i$. By the definition of $s'$, this is equivalent to showing that $(2i+1-2j)r_{2u+1} + 2(j'-u)r_{2i+1} - (l-2u)r_{2j'} \geq 0$ for all $0 \leq u \leq i$. Note that, from the definition of $s'$ and $j'$, we already know that $\frac{2(r_{2i+1}-r_{2j'})}{2i+1-2j'} \geq \frac{2(r_{2i+1}-r_{2u})}{2i+1-2u}$ for all $0 \leq u \leq i$, which is equivalent to $(2i+1-2j')r_{2u}+2(j'-u)r_{2i+1}-(2i+1-2u)r_{2j'} \geq 0$. Since $\mathbf{r}\in \projC$, $r_{2u+1} \geq r_{2u}$ for every $0 \leq u \leq i$ by inequalities of type (5.3) for $C$, and thus the result holds.\\

\textbf{Claim 4:} We have that $d'_0 \leq \frac{2b'-s'}{2}$.

By the construction of $d'_0$ in Definition \ref{def:rlodd}, we want to show that $\min_{0\leq u \leq i}\{r_{2u+1}-(u+1)s'\}\leq \frac{2b'-s'}{2}$. It is thus sufficient to find any $u$ such that $0 \leq u \leq i$ and $r_{2u+1}-(u+1)s'\leq \frac{2b'-s'}{2}$. Set $u=j'$; we want to show that $2r_{2j'+1}-2(j'+1)s'\leq 2b'-s'$. Then plugging in the expressions for $s'$ and $b'$ in Definition \ref{def:rlodd}, this is equivalent to showing that $r_{2i+1}+(2i-2j')r_{2j'}-(2i+1-2j')r_{2j'+1} \geq 0$, which follows from Lemma \ref{lem:forupperboundd0}.\\

\textbf{Claim 5:} We have that $d'_v \geq 0$ for all $0 \leq v \leq n$.

By Lemma \ref{lem:otherwaytocompdv} and by the construction of $d'_0$ in Definition \ref{def:rlodd}, it suffices to show that $r_{2u+1}-(u+1)s' \geq 0$ for all $0 \leq u \leq i$, which, by the definition of $s'$, is equivalent to showing that $(2i+1-2j')r_{2u+1}+2(u+1)r_{2j'} - 2(u+1) r_{2i+1} \geq 0$ for all $0 \leq u \leq i$. This holds by Lemma \ref{lem:forlowerboundd0}.\\

\textbf{Claim 6:} We have that $\sum_{w=0}^u d'_w \geq \sum_{w=v-u}^v d'_w$ for all $0 \leq u < v \leq n$.

First note that showing $\sum_{w=0}^u d'_w \geq \sum_{w=v-u}^v d'_w$ is equivalent to showing $(u+1)s' + \sum_{w=0}^u d'_w + (v-u)s' + \sum_{w=0}^{v-u-1} d'_w \geq (v+1)s' + \sum_{w=0}^v d'_w$.  This in turn is equivalent to showing that $r'_{2u+1}+r'_{2(v-u-1)+1} \geq r'_{2v+1}$. Renaming indices, we can show instead that $r'_{2u+1}+r'_{2v+1}\geq r'_{2(u+v+1)+1}$.

For any $0\leq u<i$, let $\bar{u}$ be the smallest index such that $\bar{u}\geq u$ and $d'_{\bar{u}+1}>0$. Then observe that 

\begin{align*}
r'_{2\bar{u}+1} &= (\bar{u}+1)s' + \sum_{w=0}^{\bar{u}} d'_w\\
&= (u+1)s' + (\bar{u}-u)s' + \sum_{w=0}^u d'_w + \sum_{w=u+1}^{\bar{u}} d'_w\\
& = r'_{2u+1} + (\bar{u}-u)s'\\
& = r_{2u+1} + (\bar{u}-u)s'
\end{align*}
where the third line follows from the fact that $\sum_{w=u+1}^{\bar{u}} d'_w=0$ by definition of $\bar{u}$, and where the last line follows from the observation in Lemma \ref{lem:somecoordsequalinrlodd} that since $d'_{\bar{u}+1}>0$, we know that $r'_{2\bar{u}+1} = r_{2\bar{u}+1}$ . Thus

\begin{align*}
r'_{2u+1}+r'_{2v+1}&= r_{2\bar{u}+1}-(\bar{u}-u)s' + r_{2\bar{v}+1}-(\bar{v}-v)s'\\
& \geq r_{2(\bar{u}+\bar{v}+1)+1} - (\bar{u}-u)s'-(\bar{v}-v)s'\\
& \geq r'_{2(\bar{u}+\bar{v}+1)+1}- (\bar{u}-u)s'-(\bar{v}-v)s'\\
& \geq r'_{2(u+v+1)+1}
\end{align*}
where the first line holds by the definition of $\bar{u}$ and $\bar{v}$, the second because $\mathbf{r} \in \projC$ and inequalities of type (5.4) for $C$, the third because $r'_w \leq r_w$ for every $w\geq 0$ by Lemma \ref{lem:rupperboundprimeodd}, and the last one because $r'_{2w+1}-r'_{2w-1} \geq s'$ for every $w$ in our construction.\\

\textbf{Claim 7:} We have that $2\sum_{v=1}^{f'} d'_v \leq s'$.

Note that by the construction of $r'_{2v+1}$ in Definition \ref{def:rlodd}, we have that $$\sum_{w=0}^i d'_w=r_{2i+1}-(i+1)s'.$$ We have already shown that $r'_{2i+1}=r_{2i+1}$ in Lemma \ref{lem:rupperboundprimeodd}.  Moreover, by the definition of $r'_{2v}$ in Definition \ref{def:rlodd}, we know that $r'_{2i}=r_{2j'}+(i-j')s'$. Thus
$$r'_{2i+1}-r'_{2i}-\frac{s'}{2} = r_{2i+1} - r_{2j'} - \frac{(i-j)\cdot 2 \cdot(r_{2i+1}-r_{2j'})}{2i+1-2j'} - \frac{r_{2i+1}-r_{2j'}}{2i+1-2j'}=0.$$
In other words, $\frac{s'}{2}=r'_{2i+1}-r'_{2i}$. We also know that $r'_{2i+1}-r'_{2i} = (i+1)s'+\sum_{w=0}^{i} d'_w - (b'+is')$. So 

\begin{align*}
\frac{s'}{2} &= -b'+\sum_{w=0}^{i} d'_w+s'\\
b'-\frac{s'}{2} -d'_0 &= \sum_{w=1}^{i} d'_w\\
\frac{s'}{2} &\geq \sum_{w=1}^{i} d'_w
\end{align*}
where the last line follows from the fact that $b' \leq d'_0+s'$ from Claim 3.
\end{proof}

We now give a similar definition and prove similar results for even $l$.

\begin{definition}[$\mathbf{r}_l$ when $l$ is even]\label{def:rleven}
Given $\mathbf{r}=(r_0, r_1, \ldots, r_{2n+1}) \in \projC$ and $0 \leq l \leq 2n+1$ such that $l=2i$, let $\mathbf{r}'_{l}=(r'_{0}, r'_{1}, \ldots, r'_{2n+1})$ be as follows.

Let $$s'=\max_{0\leq j\leq i} \left\{ \frac{r_{2i}-r_{2j}}{i-j} \right\}$$ and  $$j' = \argmax_{0\leq j\leq i}\left\{ \frac{r_{2i}-r_{2j}}{i-j} \right\}.$$ Here $s'$ is the maximal slope involving the point $l$ and a point $2j$ before it, and $j'$ is the index where the maximal slope occurs. We are going to use the slope $s'$ when building our alternative sequence $\mathbf{r}'_{l}=(r'_0, r'_1, \ldots, r'_{2n+1})$.

Let $r'_{2v} = r_{2j'} + (v-j')s'$ for all $0\leq v\leq n$ so that all even points are on the same line with slope $s'$. Let $b':=r'_0$ 

Let $\mathbf{d}'=(d'_0, 0, \ldots, 0) \in \mathbb{R}^n+1$ where $d'_0= b'-s'$ if $b'>s'$ and $d'_0=0$ is $b'\leq s'$. Let $$r'_{2v+1}=(v+1)s'+d'_0$$ for any $0\leq v \leq n$.
\end{definition}

\begin{example}
Let $2n+1=21$, $l=14$, and $$\mathbf{r}=(14,14,24,24,34,35,44,45,55,55,66,66,77,79,88,89,99,99,110,110,121,121).$$

Then $s'=10$, $b'=14$, $\mathbf{d}'=(4,0,0,0, \ldots, 0)$, and $$\mathbf{r}'_{14}=(11, 11, 22, 22, 33, 33, 44, 44, 55, 55, 66, 66, 77, 77, 88, 88, 99, 99, 110, 110, 121, 121).$$

\end{example}

\begin{example}
Let $2n+1=15$, $l=12$, and $$\mathbf{r}=(11, 11, 19, 20, 27, 30, 36, 40, 45, 50, 55, 60, 65, 70, 75, 80).$$

Then $s'=10$, $b'=5$, $\mathbf{d}'=(0,0,\ldots, 0)$, and $$\mathbf{r}'_{12}=(5,10,15,20,25,30,35,40,45,50,55,60,65,70,75, 80).$$ 
\end{example}

\begin{lemma}
Given $\mathbf{r}=(r_0, r_1, \ldots, r_{2n+1}) \in \projC$ and $\mathbf{r}'_{l}=(r'_{0}, r'_{1}, \ldots, r'_{2n+1})$ for some even $l=2i$, $r'_v\leq r_v$ for all $0\leq v \leq 2n+1$ and $r_l=r'_l$.
\end{lemma}

\begin{proof}
\textbf{Claim 1:} We have that $r'_{2i}=r_{2i}$.

By Definition \ref{def:rleven}, we have
\begin{align*}
r'_{2i}-r_{2i}&=r_{2j'}+(i-j')s'-r_{2i}\\
&=r_{2j'}+(i-j')\frac{(r_{2i}-r_{2j'})}{i-j'} - r_{2i}\\
&=0.
\end{align*}

\textbf{Claim 2:} We have that $r'_{2v} \leq r_{2v}$ for all $0\leq v \leq j'$.

Since $\mathbf{r}\in \projC$, by inequalities (5.2) for $C$, we know that $r_{2j'}-r_{2j'-2} \geq r_{2u}-r_{2u-2}$ for all $0\leq u \leq j'$. Now observe that by the construction of $s'$ and $j'$ in Definition \ref{def:rleven}, $s'=\frac{r_{2i} - r_{2j'}}{i-j'} \geq \frac{r_{2i}-r_{2j'-2}}{i-(j'-1)}$ which is equivalent to $(i-(j'-1))(r_{2i} - r_{2j'}) \geq (i-j')(r_{2i}-r_{2j'-2})$, which in turn is equivalent to $\frac{r_{2i}-r_{2j'}}{i-j'} \geq r_{2j'}-r_{2j'-2}$.  So we get that $s' \geq r_{2j'}-r_{2j'-2}$ which implies that $s' \geq r_{2u}-r_{2u-2}$ for all $u\leq j'$. Together with the fact that $r'_{2v}=r'_{2j'}-(j'-v)s'=r_{2j'}-(j'-v)s'$ by definition, this implies that $r'_{2v} \leq r_{2v}$ for all $0\leq v\leq j'$.\\

\textbf{Claim 3:}  We have that $r'_{2v} \leq r_{2v}$ for all $j'\leq v \leq n$.

We first show that $s'\leq r_{2j'+2}-r_{2j'}$. Let's first assume that $j'<i$. Then $$\frac{r_{2i}-r_{2j'}}{i-j'} \geq \frac{r_{2i}-r_{2j'+2}}{i-(j'+1)}$$ by the construction of $s'$ and $j'$ in Definition \ref{def:rleven} which is equivalent to $$\frac{r_{2i}-r_{2j'}}{i-j'} \leq r_{2j'+2} - r_{2j'},$$ as desired.   Thus $s'\leq r_{2u+2}-r_{2u}$ for all $u\geq j'$. Together with the fact that $r'_{2v}=r'_{2j'}+(v-j')s'=r_{2j}+(v-j')s'$ by the construction of $r'_{2v}$ in Definition \ref{def:rleven}, this implies that $r'_{2v} \leq r_{2v}$ for all $ v\geq j'$.

Finally we consider the case when $j'=i$. Then $s'=0$, and so $r'_{2v}=r_{2i}$ for all $v\geq j'$. Since $\mathbf{r}\in \projC$, by combining inequalities of types (5.1) and (5.3) for $C$, we know $r_{2v} \geq r_{2i}$ for every $v\geq j'$, and so we again have that $r'_{2v} \leq r_{2v}$.\\

\textbf{Claim 4:} We have that $r'_{2v+1} \leq r_{2v+1}$ for all $0\leq v\leq n$.

By Definition \ref{def:rleven}, if $b'>s'$, then $\mathbf{d}'=(b'-s', 0, \ldots, 0)$ and $r'_{2v+1}=(v+1)s'+b'-s'=b'+vs'=r'_{2v}$ for every $0\leq v \leq n$. We've already shown that $r'_{2v} \leq r_{2v}$. Moreover, since $\mathbf{r}\in \projC$, by inequalities of type (5.3), $r_{2v}\leq r_{2v+1}$. Therefore, the result holds.

If $b'\leq s'$, then $\mathbf{d}'=(0, 0, \ldots, 0)$ and $r'_{2v+1}=(v+1)s'$ for all $0\leq v \leq n$. So we want to show that  $$(v+1)\frac{(r_{2i}-r_{2j'})}{i-j'} \leq r_{2v+1},$$ which is equivalent to showing that
$$(v+1) r_{2j'} - (v+1) r_{2i} + (i-j')r_{2v+1}\geq 0.$$

The case when $0\leq v \leq i-1$ is covered by Lemma \ref{lem:forlowerboundd0}. It remains to show that this inequality also holds when $i \leq v \leq n$. It does since it can be written as the following conic combination of inequalities (5.5) and (5.2) for $C$ which hold since $\mathbf{r}\in\projC$:

\begin{align*}
&(i-j')((v+1) r_{2v-2} - (v+1)r_{2v} + r_{2v+1} \geq 0)\\
+&\sum_{w=j'}^{\min\{i-1, v-2\}} (w-j'+1)(v+1)(r_{2w}-2r_{2w+2}+r_{2w+4} \geq 0)\\
+ & \sum_{w=i}^{v-2}(i-j')(v+1)(r_{2w}-2r_{2w+2}+r_{2w+4} \geq 0).
\end{align*}

Note that if $v\in \{i, i+1\}$, then the last row disappears. By collecting terms, one can check that $r_{2j'}$ appears $v+1$ times, $r_{2i'}$ appears $-(v+1)$ times, and $r_{2v+1}$ appears $i-j'$ times. 

\end{proof}

\begin{theorem}
Let $\mathbf{r}\in \projC$ such that $r_1\geq r_0$ and fix $0\leq l \leq 2n+1$ such that $l=2i$. Then $\mathbf{r}'_l$ is in $\mathcal{R}$.
\end{theorem}

\begin{proof}
We show that all properties of rays in $\mathcal{R}$ hold for $\mathbf{r}'_l$.\\

\textbf{Claim 1:} We have that $s'\geq 0$.

By the construction of $s'$ in Definition \ref{def:rleven}, $s'=\max_{0\leq j\leq i} \left\{ \frac{r_{2i}-r_{2j}}{i-j} \right\}$. Since $\mathbf{r}\in \projC$, we know that $r_{2i} \geq r_{2j}$ for all $0\leq j \leq i$ since from type (3) inequalities, $r_{2u+1}\geq r_{2u}$ for all $1 \leq u\leq n$, and combining type (1) and type (3) inequalities yields that $r_{2u+2}\geq r_{2u+1}$ for all $1\leq u \leq n-1$. Finally, we assumed that $r_1 \geq r_0$.\\

\textbf{Claim 2:} We have that $b'\geq 0$.

By the construction of $b'$ in Definition \ref{def:rleven}, we need to show that $r_{2j'}-j's'\geq 0$. By the definition of $s'$, this is equivalent to showing that $ir_{2j'}-j'r_{2i} \geq 0$, which we know holds by Lemma \ref{lem:forbprimenonneg}.\\

\textbf{Claim 3:} We have that $d'_0 \geq b'-s'$.

If $b'>s'$, then $d'_0=b'-s'$ and the inequality thus holds. If $b'\leq s'$, then $d'_0=0\geq b'-s'$ as desired.\\

\textbf{Claim 4:} We have that $d'_0 \leq \frac{2b'-s'}{2}$.

If $b'>s'$, then $d'_0=b'-s'$, and since $s'\geq 0$, we have that $d'_0 \leq \frac{2b'-s'}{2}$. If $b'\leq s'$, then $d'_0=0$ and so, by the construction of $b'$ in Definition \ref{def:rleven},  we need to show that $r_{2j'}-j's' - \frac{s'}{2}\geq 0$ which, by the construction for $s'$, is equivalent to showing that $$(2i+1)r_{2j'} - (2j'+1)r_{2i} \geq 0,$$ which we know holds from Lemma \ref{lem:lowerboundd0leven}.\\

\textbf{Claim 5: }We have that $d'_v \geq 0$ for all $0 \leq v \leq n$.

This is clear from the construction $\mathbf{d}'$ in Definition \ref{def:rleven}.\\

\textbf{Claim 6:} We have that $\sum_{v=0}^u d'_v \geq \sum_{v=w-u}^w d'_v$.

Again, this is clear by the construction of $\mathbf{d}'$ in Definition \ref{def:rleven}.\\

\textbf{Claim 7:} We have that $2\sum_{v=1}^{f'} d'_v \leq s'$.

Note that the lefthand side is 0 by our construction of $\mathbf{d}'$ in Definition \ref{def:rleven}, and we've already shown that $s'\geq 0$ in Claim 1.\\
\end{proof}

\begin{theorem}\label{thm:rnotdecreasingintropn}
Let $\mathbf{r}\in \projC$ such that $r_1 \geq r_0$. Then $\mathbf{r}$ is in the max closure of rays $\mathbf{r}'_l$ for $0 \leq l \leq 2n+1$. Moreover, this implies that $\mathbf{r} \in \troptn$. 
\end{theorem}

\begin{proof}
This follows from that fact that $r_l$ is equal to the $l$th component of $\mathbf{r}'_l$, and that all the other components of the latter are smaller or equal to the corresponding components of $\mathbf{r}$. Thus, taking the tropical sum of all $\mathbf{r}'_l$ yields exactly $\mathbf{r}$, i.e., $\mathbf{r}= \oplus_{l=0}^{2n+1} \mathbf{r}'_l$. Since we have shown that each $\mathbf{r}'_l \in \mathcal{R}$ and since we know that the double hull of $\mathcal{R}$ is in $\troptn$, it follows that $\mathbf{r}\in \troptn$.
\end{proof}

\begin{lemma}\label{lem:rdecreasingintropn}
Let $\mathbf{r}=(r_0, r_1, \ldots, r_{2n+1}) \in \projC$ such that $r_1<r_0$. Then $\mathbf{r}\in \troptn$
\end{lemma}

\begin{proof}
First observe that the ray $\mathbf{r}^*=(r^*_0, r^*_1, r^*_2, r^*_3, \ldots, r^*_{2n+1}):=(r_1, r_1, r_2, r_3, \ldots, r_{2n+1}) \in \projC$ as well. Indeed, any inequality involving $y_0$ is still valid: $r^*_0-2r^*_1+r^*_2=-r_1+r_2\geq 0$ by (9), $r^*_0-2r^*_2+r^*_4=r_1-2r_2+r_4$ by (7), and $2r^*_{0}-(2v+1)r^*_{2v-1}+(2v-1)r^*_{2v+1}= 2r_1-(2v+1)r_{2v-1}+(2v-1)r_{2v+1}\geq 0$ for all $2\leq u \leq n$ by (8). All other inequalities did not change, and so they are still valid as well given that $\mathbf{r}$ was in $\projC$. By Theorem \ref{thm:rnotdecreasingintropn}, $\mathbf{r}^* \in  \troptn$

Now observe that $(r_0, r_1, r_1, r_1, \ldots, r_1)\in \troptn$ as it can be obtained as a conic combination of $(1,0,0, \ldots, 0)$ and $(1,1,1, \ldots, 1)$ which are both in $\troptn$, namely $$(r_0, r_1, r_1, r_1, \ldots, r_1)=(r_0-r_1)\cdot(1,0,0, \ldots, 0)+r_1\cdot(1,1,1, \ldots, 1).$$ 

Finally, observe that $\mathbf{r}=(r_1, r_1, r_2, r_3, \ldots, r_{2n+1})\oplus (r_0, r_1, r_1, r_1, \ldots, r_1)$. Indeed, since $\mathbf{r}\in \projC$, inequality (5.9) for $C$ tells us that $r_1\leq r_2$, combining inequalities of types (5.1) and (5.3) yields $-r_{2u+1} +r_{2u+2}\geq 0$ for all $1\leq u \leq n-1$, and we also have $-r_{2u}+r_{2u+1}\geq 0$ for all $1\leq u \leq n$ from inequalities of type (5.3). Therefore, $r_1\leq r_v$ for $1\leq v \leq 2n+1$. Since $\mathbf{r}$ is in the max closure of two rays that are in $\troptn$, $\mathbf{r} \in \troptn$.
\end{proof}

Thus we have shown that $ \projC= \troptn$.


\section{Applications}\label{sec:applications}

\subsection{HDE of two paths}\label{subsec:HDEPmPn}

Recall that $\HDE(F_1;F_2)$ denotes the maximum value of $c$ such that $$\hom(F_1;G) \geq \hom(F_2;G)^c$$ for every graph $G$. In \cite{koppartyrossman}, Kopparty and Rossman show that $\HDE(P_v, P_w) = 1$ when $v\geq w$ and $\HDE(P_v,P_w)=\frac{v+1}{w+1}$ when $v\leq w$ and $v$ is even. The case when $v\leq w$ and $v$ is odd is open in general, though in the same paper, Kopparty and Rossman showed that $\HDE(P_1, P_w)=\frac{1}{\lceil \frac{w+1}{2} \rceil}$ and that
$$\HDE(P_3, P_{4u+i-1}) = \left\{ \begin{array}{ll}
\frac{1}{u} & \textup{if } i=0,\\
\frac{2}{2u+1} & \textup{if } i=1,\\
\frac{4u+1}{4u^2+3u+1} & \textup{if } i=2,\\
\frac{1}{u+1} & \textup{if } i=3.
\end{array}\right.$$

Since $\hom(P_v;G) \geq \hom(P_w;G)^c$ is a binomial inequality, we know it has to be implied by the defining inequalities of the tropicalization of the path profile. Using results from the previous sections, we prove the following theorem which computes $\HDE(P_v, P_w)$ when $v\leq w$ and $v$ is odd, thus resolving the problem of finding a closed expression for $\HDE(P_v, P_w)$ for all $v$ and $w$. This implies Theorem \ref{thm:iHDE} from the introduction.

\begin{theorem}\label{thm:HDEf}
We have that $$\HDE(P_v, P_w) = \frac{v+1}{w+2},$$ when $v$ is odd and $w$ is even with $v\leq w$, and 

$$\HDE(P_v, P_w) =\frac{k(v+1)-v}{kw+2k-w-1},$$ when $v$ and $w$ are both odd, $v\leq w$, and where $k$ is the smallest integer such that $k(v+1)\geq w+1$.
\end{theorem}
Note that unlike the statement of Kopparty-Rossman for $P_3$ and $P_n$ there are in fact only two cases, depending on the parity of $w$ and their result can be reformulated in this way. We split the proof into two cases, based on the parity of $w$ and start with even $w$ first.

\begin{lemma}
We have that $P_{2i+1}^{j+1} \geq P_{2j}^{i+1}$ when $2i+1 \leq 2j$. In particular, this implies that $\HDE(P_{2i+1}, P_{2j}) \geq \frac{i+1}{j+1}$ when $2i+1 \leq 2j$.
\end{lemma}

\begin{proof}
This is equivalent to showing that $(j+1)y_{2i+1} - (i+1)y_{2j} \geq 0$ where $y_v=\log(P_v)$.  

We first note that $$(i+1)y_{2i}+y_{2i+1} - (i+1)y_{2i+2} \geq 0$$ since it can be written as the following conic combination of inequalities for $C$ of type (5.2), (5.6), and (5.4):

\begin{align*}
\sum_{v=i}^{2i-1} (i+1)&(y_{2v}-2y_{2v+2} + y_{2v+4} \geq 0)\\
+ \frac{1}{2} & ((2i+2)y_{4i} - (2i+2) y_{4i+2} +y_{4i+3} \geq 0)\\
+ \frac{1}{2} & (2y_{2i+1} - y_{4i+3} \geq 0).
\end{align*}
Moreover, note that $$y_{2i+1} + (i+1)y_{2l-2} - (i+1) y_{2l} \geq 0$$ for any $l \geq i+2$ since it can be written as the following linear conic combination of inequalities of type (5.2) and a general inclusion inequality described in Section \ref{subsec:findknownineqs}:

\begin{align*}
\sum_{v=l-1}^{i+l-2} (i+l-v-1)&(y_{2v}-y_{2v+2}+y_{2v+4} \geq 0)\\
+ & (y_{2i+1}+y_{2l-2}-y_{2i+2l} \geq 0).
\end{align*}

Taking a conic combination of the previous two types of inequalities

\begin{align*}
&((i+1)y_{2i}+y_{2i+1} - (i+1)y_{2i+2} \geq 0)\\
+ \sum_{l= i+2 }^{j} &(y_{2i+1} + (i+1)y_{2l-2} - (i+1) y_{2l} \geq 0),
\end{align*}
we obtain $$(i+1)y_{2i}+(j-i) y_{2i+1} - (i+1) y_{2j} \geq 0.$$

Finally, since $y_{2i+1} \geq y_{2i}$ by inequalities of type (5.3), this last inequality implies that $(j+1) y_{2i+1} \geq (i+1) y_{2j}$.
\end{proof}

\begin{lemma}
We have that $\HDE(P_{2i+1}, P_{2j}) \leq \frac{i+1}{j+1}$ when $2i+1 \leq 2j$
\end{lemma}

\begin{proof}
Recall the blow-up graph introduced in Section \ref{subsec:blowup}. For every $m\in \mathbb{N}$, we create a blow-up of $P_{2j}$ called $B_m$. We let $p(\{v, v+1\})=\frac{1}{j+1}$ for $0 \leq v \leq 2j-1$ and $p(\{v\})=\frac{1}{j+1}$ if $0 \leq v \leq 2j$ is even, and $p(\{v\})=0$ if $1 \leq v \leq 2j-1$ is odd. As $m\rightarrow \infty$, we will get $\hom(P_{2i+1}, B_m)\rightarrow (2i+1)\frac{1}{j+1}-i\frac{1}{j+1}=\frac{i+1}{j+1}$ since any homomorphism from $P_{2i+1}$ to the weighed version of $P_{2j}$ gives exactly $\frac{i+1}{j+1}$. Similarly, as $m\rightarrow \infty$, $\hom(P_{2j}, B_m) \rightarrow (2j)\frac{1}{j+1} - (j-1) \frac{1}{j+1}=\frac{j+1}{j+1}$ since the maximum of any homomorphism from $P_{2j}$ to the weighed version of $P_{2j}$ is $\frac{j+1}{j+1}$ (here taking the maximum over homomorphisms is not extraneous, as a homomorphism that sends endpoints of $P_{2j}$ to an even vertex of $P_{2j}$ gives only $\frac{j}{j+1}$).

So since $\hom(P_{2i+1}, B_m)\rightarrow \frac{i+1}{j+1}$ and $\hom(P_{2j}, B_m)\rightarrow 1$ as $m\rightarrow \infty$, $\HDE(P_{2i+1}, P_{2j}) \leq \frac{i+1}{j+1}$.
\end{proof}

This finishes the proof of the even $w$ case of Theorem~\ref{thm:HDEf} and we now deal with the case of odd $w$.

\begin{lemma}
We have that $P_v^{kw+2k-w-1} \geq P_w^{k(v+1)-v}$ when $v$ and $w$ are both odd, $v\leq w$, and where $k$ is the smallest integer such that $k(v+1)\geq w+1$.
\end{lemma}

\begin{proof}
The above inequality following immediately from combining three inequalities we have already seen:
$$P_{v}^k\geq P_{k(v+1)-1},  \,\,\,\, P_v\geq P_{v-1},\,\,\,\, P_{v-1}^{k(v+1)-w-1} P^{w-v+1}_{k(v+1)-1}\geq P_w^{k(v+1)-v}.$$
The first two inequalities are just inclusion inequalities. Although the last inequality looks complicated, it is just our generalization of Erd\H{o}s-Simonovits (since $v-1$ is even) presented in Lemma \ref{lem:biggestgeneralizationES}.
\end{proof}

\begin{lemma}
We have that $\HDE(P_{v}, P_{w}) \leq \frac{k(v+1)-v}{kw+2k-w-1}$ when $v$ and $w$ are both odd, $v\leq w$, and where $k$ is the smallest integer such that $k(v+1)\geq w+1$.
\end{lemma}

\begin{proof}
Let $s=2(k-1)$, $b=2k-1$, and $\mathbf{d}=(d_0, d_1, \ldots)$ where 
$$d_u=\left\{\begin{array}{ll}
1 & \textup{if } u\equiv 0 \ \textup{mod} \frac{v+1}{2} \textup{ and } u\leq(k-1)\frac{(v+1)}{2}\\
0 & \textup{otherwise.}
\end{array}\right. $$
Note that the ray $\mathbf{r}$ built from this $b, s, \mathbf{d}$ is in $\mathcal{R}_{s,b}$ and is thus realizable. Observe that $2\cdot (k-1) \frac{(v+1)}{2} +1 \leq w$ so long as $l>0$, which we know is true since if $l=0$, then $w=(k-1)(v+1)$, and since $(v+1)$ is even and thus $w$ is even, we reach a contradiction.

We have that $$r_v = \frac{v+1}{2}s + \sum_{u=0}^{v-1}{2} d_u=\frac{v+1}{2}\cdot (2k-2) +1 = k(v+1)-v,$$ and $$r_w=\frac{w+1}{2}s+\sum_{u=0}^{\frac{w-1}{2}} d_u = (w+1)(k-1) + k = kw+2k-w-1.$$ So we know there exists a blow-up graph $B_m$ for which, as $m\rightarrow\infty$, $\log(\hom(P_v, B_m))\rightarrow k(v+1)-v$ and $\log(\hom(P_w, B_m))\rightarrow kw+2k-w-1$, so  $\HDE(P_{v}, P_{w}) \leq \frac{k(v+1)-v}{kw+2k-w-1}$. 
\end{proof}

This finishes the proof of Theorem~\ref{thm:HDEf}.

\subsection{Previously known inequalities}\label{subsec:recoveroldineqs}

Binomial inequalities between numbers of homomorphisms of paths have been studied for a long time. Erd\H{o}s and Simonovits proved in \cite{erdossimonovits} that $$P_0^{k-1}P_k \geq P_1^k.$$ Lagarias, Mazo, Shepp and McKay showed in \cite{Lagarias} that $$P_0P_{2a+2b} \geq P_{2a+b}P_b.$$ In \cite{DressGutman}, Dress and Gutman showed that $$P_{2a}P_{2b} \geq P_{a+b}^2.$$ In \cite{Hemmecke}, Hemmecke, Kosub, Mayr, T\"{a}ubig and Weihmann generalize all previous listed inequalities by certifying the following two types of inequalities: $$P_{2a}P_{2(a+b+c)} \geq P_{2a+c}P_{2(a+b)+c}$$ and $$P_{2l+pk}P_{2l}^{k-1}\geq P_{2l+p}^k.$$ Our results show that all of these inequalities (and any other true binomial inequality involving paths) are implied by the binomial inequalities defining $C$. We show here how the two most general inequalities can be recovered.

As was seen in the previous section, it is more convenient to check the validity of a binomial inequality for $\NU$ by checking the validity of the corresponding linear inequality for $\tropNU$. There, we simply need to show that this linear inequality can be written as a conic combination of the defining inequalities for $\tropNU$. If there is no such conic combination, then the original pure binomial inequality is not valid for $\NU$. Finding such a conic combination can be done via a linear program. Indeed, suppose one wants to check whether $\prod_{i\in I_1} P_i^{\alpha_i} \geq \prod_{i\in I_2} P_i^{\beta_i}$ is a valid inequality, where $I_1, I_2 \subseteq \{0, 1, 2, \ldots, 2n+1\}$ for some $n\in \mathbb{N}$. This is equivalent to checking that $\sum_{i\in  I_1} \alpha_i y_i - \sum_{i\in I_2} \beta_i y_i \geq 0$ on $\tropn$. Thus one can simply minimize $\sum_{i\in  I_1} \alpha_i y_i - \sum_{i\in I_2} \beta_i y_i$ over the cone $C$ given at the beginning of Section \ref{sec:troppaths}. If the optimal value is 0, then the inequality is valid, and the dual solution gives the conic combination of inequalities of $C$ that yields $\sum_{i\in  I_1} \alpha_i y_i - \sum_{i\in I_2} \beta_i y_i \geq 0$. Otherwise, the inequality is not valid. 

We now show that $P_{2a}P_{2(a+b+c)} \geq P_{2a+c}P_{2(a+b)+c}$ can indeed be recovered in that way. 
Indeed, the conic combination

\begin{align*}
&\sum_{i=a}^{a+\lfloor \frac{c}{2} \rfloor -1} (i+1-a)\cdot (y_{2i}-2y_{2i+2}+y_{2i+4} \geq 0)\\
\bigg[+ & \frac{1}{2}\cdot (y_{2a+c-1}-2y_{2a+c}+y_{2a+c+1}\geq 0)\bigg] \\
+ & \sum_{i=a+\lfloor \frac{c}{2} \rfloor}^{a+b+\lfloor \frac{c}{2} \rfloor-2} \frac{c}{2} \cdot (y_{2i}-2y_{2i+2}+y_{2i+4}\geq 0)\\
\bigg[+ & \frac{1}{2}\cdot (y_{2(a+b)+c-1}-2y_{2(a+b)+c} + y_{2(a+b)+c+1} \geq 0)\bigg]\\
+ & \sum_{i=a+b+\lfloor \frac{c}{2} \rfloor}^{a+b+c-2} (a+b+2c'-1-i)\cdot (y_{2i}-2y_{2i+2}+y_{2i+4} \geq 0),\\
\end{align*}
where the lines in brackets are present only if $c$ is odd, yields $y_{2a} - y_{2a+c} - y_{2(a+b)+c}+ y_{2(a+b+c)}\geq 0$ as desired. So $P_{2a}P_{2(a+b+c)} \geq P_{2a+c}P_{2(a+b)+c}$ can be recovered only with the binomials of the form $P_{2v}P_{2v+4}\geq P_{2v+2}^2$ when $c$ is even, and $P_{2v}P_{2v+4}\geq P_{2v+2}^2$ with $P_{2v}P_{2v+2}\geq P_{2v+1}^2$ when $c$ is odd.

For $P_{2l+pk}P_{2l}^{k-1}\geq P_{2l+p}^k$, if $pk$ is even, this can again be recovered only with the binomials of the form $P_{2v}P_{2v+4}\geq P_{2v+2}^2$ and $P_{2v}P_{2v+2}\geq P_{2v+1}^2$ since the conic combination 

\begin{align*}
& \sum_{i=l}^{l+\lceil\frac{p}{2}\rceil -2}(l+1-i)(k-1)\cdot (y_{2i}-2y_{2i+2}+y_{2i+4} \geq 0)\\
\bigg[+ &\frac{k}{2} \cdot (y_{2l+p-1} - 2y_{2l+p} + y_{2l+p+1}\geq 0)\bigg]\\
+ & \sum_{i=l+\lceil\frac{p}{2}\rceil-1}^{l+\frac{pk}{2} - 2} (l+\frac{p}{2}k-1-i) \cdot (y_{2i}-2y_{2i+2} + y_{2i+4} \geq 0),
\end{align*}
where the line in brackets is present only if $p$ is odd, yields $(k-1)y_{2l}-ky_{2l+p}+y_{2l+pk} \geq 0$ as desired. Otherwise, if $p$ and $k$ are both odd, $P_{2l+pk}P_{2l}^{k-1}\geq P_{2l+p}^k$ can be retrieved from the generalized Erd\H{o}s-Simonovits inequalities. Indeed, the conic combination

$$\sum_{i=\frac{2l+p+1}{2}}^{\frac{2l+pk-1}{2}} \frac{kp}{(2i+1-2l)(2i-1-2l)}\cdot (y_{2l}-(2i+1-2l)y_{2i-1} + (2i-1-2l)y_{2i+1} \geq 0)$$
yields $(k-1)y_{2l}-ky_{2l+p}+y_{2l+pk} \geq 0$ as desired.

To the best of our knowledge for a parametrized family of valid binomial inequalities for $\NU$, there does not necessarily exist a nice unified parametrized family of conic combinations that gives certificates for the entire family. One may have to find several  different families of parameterized certificates. 
Moreover, a natural inequality may require a complicated conic combination of the extremal inequalities that, though maybe not as natural, can be certified more easily.

 \bibliographystyle{alpha}
\bibliography{tropofpathsreferences}
 
\end{document}